%% file: IEEEmain.tex
\pdfoutput=1
\documentclass[journal]{IEEEtran}
\usepackage{enumerate}
\usepackage[OT1]{fontenc}
\usepackage[usenames]{color}
\usepackage[colorlinks,
linkcolor=red,
anchorcolor=blue,
citecolor=blue
]{hyperref}
\usepackage{mathrsfs}
\usepackage{hyperref}
\usepackage[protrusion=true, expansion=true]{microtype}
\usepackage{float}
\usepackage{subfigure}
\usepackage{amsfonts,amsmath,amssymb,amsthm,url,xspace}
\usepackage{tikz}
\usepackage{verbatim}
\usetikzlibrary{arrows,shapes}
\usepackage{mathtools}
\usepackage{authblk}
\usepackage[bottom]{footmisc}
\usepackage{enumitem} 
\usepackage{arydshln,leftidx,mathtools}
\usepackage{lscape}
\usepackage{algorithm}
\usepackage{algorithmic}
\usepackage{chngcntr}
\counterwithout{equation}{section}
\usepackage{multirow}
\usepackage{multicol}

\allowdisplaybreaks[1]

\newcommand{\black}{\color{black}}

\newcommand{\mR}{\mathbb{R}}

\newcommand{\mL}{\mathcal{L}}

\newcommand{\diag}{{\rm diag}}

\def\P{{\mathcal P}}

\def\0{{\boldsymbol 0}}
\def\ba{{\boldsymbol{a}}}
\def\bb{{\boldsymbol{b}}}
\def\N{{\mathcal N}}

\def\bl{{\boldsymbol{l}}}
\def\br{{\boldsymbol{r}}}
\def\bbeta{{\boldsymbol{\beta}}}
\def\balpha{{\boldsymbol{\alpha}}}
\def\bw{{\boldsymbol{w}}}
\def\bp{{\boldsymbol{p}}}
\def\bq{{\boldsymbol{q}}}

\def\be{{\boldsymbol{e}}}
\def\bv{{\boldsymbol{v}}}
\def\bx{{\boldsymbol{x}}}
\def\bu{{\boldsymbol{u}}}
\def\bd{{\boldsymbol{d}}}
\def\bl{{\boldsymbol{l}}}
\def\bz{{\boldsymbol{z}}}
\def\blambda{{\boldsymbol{\lambda}}}
\def\tlambda{{\boldsymbol{\lambda}}^\star}

\def\tx{{\boldsymbol{x}}^\star}
\def\tu{{\boldsymbol{u}}^\star}
\def\tz{{\boldsymbol{z}}^\star}
\def\bblambda{\bar{\boldsymbol{\lambda}}}
\def\bardelta{\bar{\delta}}
\def\barkappa{\bar{\kappa}}

\def\Id{\text{Id}}

\def\hbz{{\hat{\boldsymbol{z}}}}
\def\hbx{{\hat{\boldsymbol{x}}}}
\def\hbu{{\hat{\boldsymbol{u}}}}
\def\hblambda{{\hat{\boldsymbol{\lambda}}}}

\def\bbeta{{\boldsymbol{\beta}}}
\def\tbx{{\tilde{\boldsymbol{x}}}}
\def\tbu{{\tilde{\boldsymbol{u}}}}
\def\bbx{\bar{\boldsymbol{x}}}

\def\bbz{\bar{\boldsymbol{z}}}

\def\tbz{{\tilde{\boldsymbol{z}}}}
\def\tblambda{{\tilde{\blambda}}}

\def\tbd{{\tilde{\boldsymbol{d}}}}

\def\tbw{{\tilde{\boldsymbol{w}}}}

\allowdisplaybreaks[1]

\usepackage{multicol}
\usepackage{lipsum}
\usepackage{mathtools}
\usepackage{cuted}
\usepackage{cite}

\newtheoremstyle{mytheoremstyle} 
{0em}                    
{0em}                    
{\normalfont}                   
{1em}                           
{\itshape}                   
{.}                          
{.5em}                       
{}  

\theoremstyle{mytheoremstyle}

\newtheorem{theorem}{Theorem}
\newtheorem{definition}{Definition}
\newtheorem{lemma}{Lemma}
\newtheorem{remark}{Remark}
\newtheorem{corollary}{Corollary}
\newtheorem{assumption}{Assumption}

\ifx\BlackBox\undefined
\newcommand{\BlackBox}{\rule{1.5ex}{1.5ex}}  
\fi

\ifx\QED\undefined
\def\QED{~\rule[-1pt]{5pt}{5pt}\par\medskip}
\fi

\ifx\proof\undefined
\newenvironment{proof}{\par\noindent{\bf Proof\ }}{\hfill\BlackBox\\[2mm]}
\fi

\newcommand{\rbr}[1]{\left(#1\right)}

\newcommand{\cbr}[1]{\left\{#1\right\}}

\ifCLASSINFOpdf
\else
\fi

\begin{document}
\title{Superconvergence of Online Optimization \\for Model Predictive Control}



\author{Sen~Na, Mihai Anitescu,~\IEEEmembership{Member,~IEEE}%
\thanks{S. Na is with the Department of Statistics, University of Chicago, Chicago, IL 60615, USA (email: senna@uchicago.edu).}%
\thanks{M. Anitescu is with the Mathematics and Computer Science Division, Argonne National Laboratory, Lemont, IL 60439, USA (email: anitescu@mcs.anl.gov).}%
}

\markboth{Accepted by IEEE TRANSACTIONS ON AUTOMATIC CONTROL}%
{Na \MakeLowercase{\textit{et al.}}: Superconvergence of Online Optimization for Model Predictive Control}

\maketitle

\begin{abstract}

We develop a one-Newton-step-per-horizon, online, lag-$L$, model predictive control (MPC) algorithm for solving discrete-time, equality-constrained, nonlinear dynamic programs. Based on recent sensitivity analysis results for the target problems class, we prove that the approach exhibits a behavior that we call \textit{superconvergence}; that is, the tracking error with respect to the full horizon solution is not only stable for successive horizon shifts, but also decreases with increasing shift order to a minimum value that decays exponentially in the length of the receding horizon. The key analytical step is the decomposition of the one-step error recursion of our algorithm into \textit{algorithmic error} and \textit{perturbation error}. We show that the perturbation error decays exponentially with the lag between two consecutive receding horizons, while~the algorithmic error, determined by Newton's method, achieves quadratic convergence instead. Overall this approach induces our local exponential convergence result in terms of the receding horizon length for suitable values of $L$. Numerical experiments validate our theoretical findings.

\end{abstract}

%

\IEEEpeerreviewmaketitle

\input{sec1}

\input{sec2}

\input{sec3}

\input{sec4}

\input{sec5}

\input{sec6}

\section*{Acknowledgment}
We thank the Associated Editor and three anonymous reviewers for helpful suggestions, which have led to a better presentation. This material was based upon work supported by the U.S. Department of Energy, Office of Science, Office of Advanced Scientific Computing Research (ASCR) under Contract DE-AC02-06CH11347 and by NSF through award CNS-1545046.

\section*{Supplementary material}

The proofs of all results in Section \ref{sec:4} are collected in the supplementary material.

%


%
%

\ifCLASSOPTIONcaptionsoff
  \newpage
\fi



%

\bibliographystyle{IEEEtran}
\bibliography{ref}

\vskip-1.5cm
\begin{flushright}
	\scriptsize \framebox{\parbox{3.5in}{Government License: The	submitted manuscript has been created by UChicago Argonne, LLC, Operator of Argonne National Laboratory (``Argonne"). Argonne, a U.S. Department of Energy Office of Science laboratory, is operated under Contract No. DE-AC02-06CH11357.  The U.S. Government retains for
			itself, and others acting on its behalf, a paid-up
			nonexclusive, irrevocable worldwide license in said
			article to reproduce, prepare derivative works, distribute
			copies to the public, and perform publicly and display
			publicly, by or on behalf of the Government. The Department of Energy will provide public access to these results of federally sponsored research in accordance with the DOE Public Access Plan. http://energy.gov/downloads/doe-public-access-plan. }}
	\normalsize
\end{flushright}

\begin{figure*}[!t]
\begin{minipage}{\textwidth}
\centering
{\LARGE Supplementary material: Superconvergence of Online Optimization for Model Predictive Control}		
		
Sen Na, Mihai Anitescu
\end{minipage}
\end{figure*}
\appendices
\ \\
\newpage

\input{appen}

\end{document}

%% file: sec1.tex
\section{Introduction}

\IEEEPARstart{M}{odel} predictive control (MPC), also known as receding horizon control or moving horizon control, is a central paradigm of modern engineering. It meets the specialized~control needs of power plants and petroleum refineries, and has now been widely used in industrial areas such as food processing, automotive transportation, metallurgy, and artificial intelligence (AI). We refer to \cite{Broeck2011model, Dunbar2012Distributed, Kraus2013Moving} for recent applications of MPC.

A simplified version of nonlinear MPC, which excludes~the inequality constraints (that could be studied by generalizing the proposed framework with more involved analysis, e.g. \cite{Xu2019Exponentially}), aims to solve the following discrete-time, equality-constrained, nonlinear dynamic programming (NLDP), $\P_{0:N}(\bd)$:
\begin{equation}\label{pro:1}
\begin{aligned}
\min_{\bx, \bu} & \sum_{k=0}^{N-1} g_k(\bx_k, \bu_k; \bd_k) + g_N(\bx_N),\\
\text{s.t.}\text{\ } & \bx_{k+1} = f_k(\bx_k, \bu_k; \bd_k), \text{\ } k = 0,1,\ldots, N-1, \\
&\bx_0 = \bbx_0,
\end{aligned}
\end{equation}
where $\bx_k\in\mR^{n_x}$ is the state variable, $\bu_k\in\mR^{n_u}$ is the control variable, $\bd_k\in\mR^{n_d}$ is the reference variable; $g_k: \mR^{n_x\times n_u\times n_d}\rightarrow \mR$ and $f_k: \mR^{n_x\times n_u\times n_d}\rightarrow \mR^{n_x}$ are the cost function and the dynamic function for stage $k$, respectively; $\bbx_0$, also denoted as $\bd_{-1}$, is the initial state variable; and $N$ is the temporal horizon length, which is supposed to be large. We assume $g_k, f_k$ (and $g_N$) are twice continuously differentiable throughout the paper; and some other extra conditions, such as local Lipschitz continuity, will be imposed  later. Problem \eqref{pro:1} is also called an optimal control problem. It is parameterized by $\bd = \bd_{-1:N-1} = (\bbx_0; \bd_0; \ldots; \bd_{N-1})$, which is also the input problem data. Here, we use semicolons in $f_k, g_k$ to separate input variables and decision variables. A similar setup for MPC without inequalities is also studied in \cite{Diehl2005Real, Diehl2005Nominal, Wynn2014Convergence}.

One of the central issues in solving \eqref{pro:1} is the \textit{time to solution}: the control action $\bu_k$ must be known \textit{before} time $k$ so it can be applied. For long horizons this is impossible. Moreover, the reference $\bd$ comes into the system as a data stream and is rarely known at time $0$ all the way to $N$. Consequently, a MPC approach, which we will set up in detail in Section~\ref{sec:3}, solves much shorter horizon problems (see \eqref{pro:2}): starting at time $n_1$, it first solves a subproblem within $[n_1, n_2]$, then advances the computation time by lag $L$, incorporates the solution with new data, solves a subsequent subproblem from $n_1+L$  to $n_2+L$, and repeats the process. Despite the reduction in computation brought about by the shrinking of the time horizon, however, for complex systems with large state spaces and multiple characteristic time scales, the resulting reduced-horizon nonlinear programs are still too expensive. To address this issue, several authors have proposed solving the MPC approximately as a parametric nonlinear equation. In this online optimization approach, one solves only a few  Newton iterations (and, sometimes, only one iteration) for the reduced-horizon problems before shifting the horizon by $L$ stages (see \cite{Ohtsuka1997Real, Diehl2002Real, Diehl2005Real, Diehl2005Nominal, Zavala2010Real, Dinh2012Adjoint, Wynn2014Convergence, Wang2010Fast, Zanelli2021Lyapunov}). Under some conditions on the data and the system, one can  show that, if started close enough to the original solution, the approach tracks the parametric optimal solution (see \cite{Zavala2010Real}). Detailed comparisons with this kind of iteration algorithms are discussed below.

In this work, we improve on existing results by proving that, when initialized close enough to the solution of \eqref{pro:1}, not only does the online MPC approach track the solution of~\eqref{pro:1}, as would be expected from the parametric optimization inspired analysis in \cite{Ohtsuka1997Real, Diehl2002Real, Zavala2010Real, Dinh2012Adjoint}, but \textit{it actually improves on the error for successive horizon shifts up to a minimum value that decays exponentially with the length of the receding horizon}. We call such a phenomenon \textit{superconvergence}, by analogy to a situation by the same name in finite elements where certain methods exhibit better convergence at particular points in the domain (the mesh nodes) compared to others \cite{Douglas1973Superconvergence}.

One of the key ingredients in proving this result is the recent sensitivity analysis on long-horizon NLDP established in \cite{Na2020Exponential}. In that result the authors showed that for NLDP \eqref{pro:1} that satisfies the second-order sufficient condition (SOSC) and controllability condition at a solution uniformly in $N$, perturbing $\bd_k$ results in a perturbation of the optimal solution that decays exponentially fast away from $k$. Their analysis is built on a convexification procedure adjusted from \cite{Verschueren2017Sparsity}. The exponential decay of the sensitivity is a key ingredient for showing fast convergence of several methods for solving dynamic programs. For instance, in the offline (where all of $\bd_k$ are available at one time, as occurs, for example in electricity planning) linear-quadratic case, where $g_k(\bx_k, \bu_k; \bd_k) = (\bx_k - \bd_k)^\top Q_k(\bx_k - \bd_k) + \bu_k^\top R_k\bu_k$ with $Q_k, R_k\succ 0$, $\bu_k$ is box constrained, and $f_k$ is affine, the authors in \cite{Xu2018Exponentially} showed that exponential decay of sensitivity helps to quantify the convergence properties of temporal decomposition. To that end, the time horizon is decomposed into multiple intervals, and two consecutive intervals are overlapped with $2b$ stages. The authors proved that exponential decay of sensitivity induces exponential convergence of the solution in $b$, and the solution is computed independently for each interval and concatenated over the entire horizon. Moreover, \cite{Na2020Overlapping} showed similar results for \eqref{pro:1}. For online solutions to \eqref{pro:1}  (where only $M+1$ components of $\bd_{k+j}$, $j=0,1,\ldots,M$ can be known at time $k$), the authors proved in \cite{Xu2019Exponentially} that, under similar setup as \cite{Xu2018Exponentially} but replacing bound constraints by path constraints, the exponential decay of sensitivity induces a lag-$L$ MPC strategy to converge to the solution of the full problem exponentially fast with respect to the receding-horizon length.

Our problem is more challenging because of nonlinearity and possible nonconvexity. The key to proving superconvergence of one-Newton-step online optimization approach is to recognize that, under positive definiteness of the reduced Hessian (as implied by SOSC) and controllability conditions, the one-step error recursion consists of two components: \textit{perturbation error}, which originates from the fact that some of the information used in the current Newton step has never been iterated on before (since new data is always coming into the optimization problem from the terminal horizon end), and \textit{algorithmic error}, which controls how fast the Newton step itself would converge on the full-horizon problem. The former is controlled by exponential decay of sensitivity considerations from \cite{Na2020Exponential}, whereas the latter automatically contributes to higher-order terms as the procedure iterates. For suitably large $L$ such that the effect of the perturbation has sufficiently diminished in the middle stages, our fast lag-$L$ online MPC strategy achieves the exponential convergence rate in $M$, but only in the middle stages of the problem (the vast majority of them). While the approach itself is unrelated to temporal decomposition ideas from \cite{Barrows2014Time, Xu2018Exponentially, Na2020Overlapping, Na2021Fast}, the buffer zone (i.e., lag $L$) we use to suppress the perturbation error is  motivated by that technique. It is also worth mentioning that recently a sequence of works studied solving a large-scale continuous-time linear-quadratic problems via different types of neural networks (NNs), such as gradient-based NN \cite{Zhang2009Performance}, zeroing NN \cite{Zhang2017Three}, recurrent NN \cite{Zhang2018Robustness, Zhang2019Power}, and convergent-differential NN \cite{Zhang2018New}. The idea is to design different neural dynamics to force certain error functions (usually the KKT residual) converge to zero. However, these works focus on general quadratic problems while we focus on nonlinear optimal control problems; in addition they require the differentiability in time of error functions, which is not applicable in our discrete case. Some papers also considered applying NN to solving MPC subproblems \cite{Li2015Missile, Li2016Trajectory}, which belongs to the exact MPC framework and is generally more inefficient than online MPC. Additionally the convergence analysis of iterates is missing in such NN-based MPC works.

\vskip2pt

\noindent{\bf Relationship to real-time iteration (RTI) algorithms}: Our paper is related to RTI algorithms, where a cheap approximate solution is computed at each sampling time. The strategy of solving a single Newton step in RTI schemes has been studied in \cite{Li1988Process, Oliveira1995extension, Ohtsuka1997Real, Diehl2002Real, Ohtsuka2004continuation/GMRES, Diehl2005Real, Diehl2005Nominal, Zavala2010Real, Dinh2012Adjoint, Wynn2014Convergence, Zanelli2017partially}. We point to \cite{Diehl2001Real, Gros2016linear, Nurkanovic2019Advanced} for a detailed survey of RTI and its relationship to (nonlinear) MPC. However, to our knowledge, the superconvergence result for online MPC/RTI-based MPC is first established in our paper for nonlinear problems. We explicitly show that the error of RTI decreases linearly in shift order, to a minimum value that decays exponentially in receding horizon length. Such a result is stronger than existing RTI results. We select some representative RTI works to briefly illustrate the differences. We refer the detailed comparisons in terms of problem setup, overhead, and theory to Remarks \ref{rem:setup}, \ref{rem:overhead}, \ref{rem:result}, respectively.
\begin{enumerate}[label=(\alph*),topsep=0pt]
\setlength\itemsep{0.0em}
\item \cite{Diehl2005Real} designed an algorithm for solving \eqref{pro:1} that performs one step Newton for each subproblem as well. Different from MPC, the authors only move the starting stage each time but fix the end stage at $N$, so their subproblems are embedded sequentially and the length of subproblems decreases from $N$ monotonically. Our MPC regime is more efficient since our subproblems have the same length $M$ which is most times much smaller than in \cite{Diehl2005Real}. We will also see that truncating both ends is intrinsically more difficult for the analysis than truncating the left end only (see Lemma \ref{lem:1} and Theorem \ref{thm:1}).

\item \cite{Diehl2005Nominal} designed a RTI-based MPC. The authors assume the objective is lower bounded by a quadratic function and hence each Newton step is well-posed. In our case, we assume only the full horizon SOSC, and allow the possibility of the Newton step for subproblems being ill-posed, which we correct by adjusting the terminal cost of the subproblems objective (see \eqref{pro:2}). As discussed in Remark \ref{rem:setup}, our time-varying problem is also more challenging to deal with than the time-invariant problem in \cite{Diehl2005Nominal}. Our theory is also carried out in a different way (see Remark \ref{rem:result}). A recent work \cite{Zanelli2021Lyapunov} extended \cite{Diehl2005Nominal} by~allowing inequality constraints and proposed a framework for RTI analysis by constructing a Lyapunov function for the combined system-optimizer dynamics. \cite{Zanelli2021Lyapunov} is not limited to Newton step but any step with $Q$-linear rate. However, it still considers time-invariant problems and studies RTI via parametric optimization as introduced bellow in (c), which is not applicable for our discrete-time, time-varying problems.

\item \cite{Diehl2002Real, Zavala2010Real, Dinh2012Adjoint} all studied RTI and its application on MPC via the lens of parametric optimization. They assume the (sub)problem is parameterized by a continuous parameter, denoted by $P(\xi)$, which is the time in control problems, and at each instant \cite{Diehl2002Real, Zavala2010Real} perform one step Newton while \cite{Dinh2012Adjoint} performs one step predictor-corrector method. \cite{Diehl2002Real} is an empirical study, while the analyses of \cite{Zavala2010Real, Dinh2012Adjoint} heavily rely on certain continuity of the problem on $\xi$, so that $P(\xi)$ and $P(\xi')$ together with their solutions are similar when $\|\xi - \xi'\|$ is small. However, this critical property does not hold in our case: even if two subproblems are successive, the distance between their two solutions (e.g. the state solution variation $\|\tx_{n_1:n_2} - \tx_{(n_1+1):(n_2+1)}\|$) need not be small.

\end{enumerate}

\noindent{\bf Structure of the paper:} In Section \ref{sec:2}, we introduce definitions and assumptions as preparation and then propose our MPC strategy in Section \ref{sec:3}. We analyze the convergence of the proposed strategy in Section \ref{sec:4}, and summarize numerical experiments and conclusions in Sections \ref{sec:5} and \ref{sec:6}, respectively.

\vskip5pt

\noindent{\bf Notations:} Throughout the paper, we use boldface fonts to denote column vectors and regular fonts to denote either scalars or matrices. Given a positive integer $k$, we let $[k] = \{0,1,\ldots,k\}$ be the index set from $0$ to $k$. For integers $k_1< k_2$, we abuse interval notations and let $[k_1, k_2]$, $[k_1, k_2)$, $(k_1, k_2]$, $(k_1, k_2)$ be the corresponding integer sets. For example, we have $(k_1, k_2) = \{k_1+1, \ldots, k_2-1\}$. We also use $\vee$ ($\wedge$) to represent $\max$ ($\min$) between two scalars. As usual, $\lfloor\cdot\rfloor$ and $\lceil\cdot\rceil$ represent floor and ceiling functions that map a real number to the greatest preceding or the least succeeding integer, respectively. For any two positive quantities $a, b$, we write $a\asymp b$ if they are of the same order, namely, $a\leq cb$ and $b\leq ca$ for some constant $c$. We let $(\ba_1; \ba_2;\ldots)$ denote a long column vector by stacking $\{\ba_i\}$ together. Similarly, $\diag(A_1, A_2, \ldots)$ is the block diagonal matrix with each block being specified by matrix $A_i$ successively. When dimensions of $\{A_i\}_{i=k_1}^{k_2}$ are compatible, we define $\prod_{i=k_1}^{k_2}A_i = A_{k_2}A_{k_2-1}\cdots A_{k_1}$ if $k_1\leq k_2$, and $I$ otherwise. We use $\|\cdot\|$ to denote the $\ell_2$-norm for vector and the operator norm for matrix. For any vector-valued function $f: \mR^{n_1}\rightarrow \mR^{n_2}$, we let $\nabla f\in \mR^{n_1\times n_2}$ be its gradient. When evaluating a function or matrix, we always use semicolon  to separate the given variables.

We also reserve the following notations for specific usage. $I$ is the identity matrix whose dimension is clear from the context; $\bx=\bx_{0:N} = (\bx_0;\ldots;\bx_N)$ (similar for $\bu$, $\bd$) is the ordered state (control, reference) variable; $\bz = (\bx_0;\bu_0;\ldots;\bx_{N-1};\bu_{N-1};\bx_N)$ is the ordered decision variable over the entire horizon; and $n_z = (N+1)n_x + Nn_u$ is its dimension. We may also write $\bz = (\bx, \bu)$ and $\bz_k = (\bx_k, \bu_k)$ ($\bz_N = \bx_N$) when specifying their components. Further, we let $N$, $M$, and $L$ be the length of the entire horizon, single receding horizon, and lag (i.e. the difference of starting stages of two successive subproblems), respectively.

%% file: sec2.tex
\section{Preliminaries}\label{sec:2}

In this section, we summarize definitions that will be used frequently  later, and we introduce the second-order sufficient condition (SOSC) and the controllability condition.

Define the Lagrange function of Problem \eqref{pro:1} as
\begin{align}\label{equ:Lagrange:full}
\mL(\bz, \blambda; \bd) \coloneqq &\sum_{k=0}^{N-1}\mL_k(\bz_k, \blambda_{k-1:k}; \bd_k) + \mL_N(\bz_N, \blambda_{N-1})\nonumber\\
 & - \blambda_{-1}^\top\bd_{-1},
\end{align}
where $\mL_N(\bz_N, \blambda_{N-1}) = g_N(\bz_N) + \blambda_{N-1}^\top\bz_N$ and
\begin{align*}
\mL_k(\bz_k, \blambda_{k-1:k}; \bd_k) = g_k(\bz_k; \bd_k) + \blambda_{k-1}^\top\bx_k - \blambda_k^\top f_k(\bz_k; \bd_k)
\end{align*}
for $k\in[N-1]$. Here, $\blambda = \blambda_{-1:N-1} = (\blambda_{-1}; \blambda_0; \cdots; \blambda_{N-1})$ is the Lagrange multiplier vector with $\blambda_k$ being associated with the $k$-th constraint in \eqref{pro:1}. Using \eqref{equ:Lagrange:full}, we define Jacobian and Hessian matrices  of the Lagrangian next.

\begin{definition}\label{def:1}

Given an evaluation primal-dual point $(\bz, \blambda; \bd)$, we let $A_k = \nabla_{\bx_k}^\top f_k(\bz_k; \bd_k)$, $B_k = \nabla_{\bu_k}^\top f_k(\bz_k; \bd_k)$ for $k\in[N-1]$. Subsequently, the evaluation point of $A_k, B_k$ is suppressed for simplicity. We define the Hessian as
\begin{align*}
H_k(\bz_k, \blambda_k; \bd_k) = \begin{pmatrix}
Q_k & S_k^\top\\
S_k & R_k
\end{pmatrix} = \begin{pmatrix}
\nabla_{\bx_k}^2\mL_k & \nabla_{\bx_k\bu_k}^2\mL_k\\
\nabla_{\bu_k\bx_k}^2\mL_k & \nabla_{\bu_k}^2\mL_k
\end{pmatrix}
\end{align*}
and $H_N(\bz_N) = \nabla_{\bz_N}^2\mL_N(\bz_N, \blambda_{N-1}) = \nabla_{\bz_N}^2g_N(\bz_N)$. For ease of notation, we use $Q_N$ for $H_N$ interchangeably. Note that $\mL_k$ depends on $\blambda_{k-1}$ linearly; hence $H_k$ does not depend on $\blambda_{k-1}$. Further, we splice all blocks and define the corresponding matrices for the entire horizon:
\begin{align*}
&\small G(\bz_{0:N-1}; \bd_{0:N-1}) = \left(\begin{smallmatrix}
I\\
-A_0 & -B_0 & I\\
&&-A_1 & -B_1 & I\\
&&&&\ddots & \ddots\\
&&&&&-A_{N-1} & -B_{N-1} & I
\end{smallmatrix}\right),\\
&\small H(\bz, \blambda_{0:N-1}; \bd_{0:N-1}) = \diag(H_0, \ldots, H_{N-1}, H_N).
\end{align*}
Here, $G\in\mR^{(N+1)n_x\times n_z}$ and $H\in\mR^{n_z\times n_z}$. Given $G$, we let $Z(\bz_{0:N-1}; \bd_{0:N-1})\in\mR^{n_z\times Nn_u}$ be a matrix whose columns are orthonormal and span the null space of $G$. Then, the reduced Hessian matrix evaluated at $(\bz, \blambda_{0:N-1}, \bd_{0:N-1})$ is given by
\begin{align*}
H_{\rm Re}(\bz, \blambda_{0:N-1}; \bd_{0:N-1}) = Z^\top HZ.
\end{align*}
\end{definition}

We may drop the evaluation point $(\bz, \blambda_{0:N-1}; \bd_{0:N-1})$~hereinafter.  For example, the reference variable, $\bd$, is always the same and need not be updated. Note, however, that all the above definitions do not depend on the first Lagrange multiplier, $\blambda_{-1}$, and the initial condition, $\bd_{-1}$.

Since $G$ has full row rank at any evaluation point, the linear independence constraint qualification (LICQ) is guaranteed. To characterize the system controllability, we define the following controllability matrix.

\begin{definition}[Controllability matrix]\label{def:2}
Given an evaluation primal point $(\bz; \bd)$, for any stage $k\in[N-1]$ and evolution length $t\in[1, N-k]$, we define the controllability matrix as
\begin{align*}
\Xi_{k, t}&(\bz_{k:k+t-1}; \bd_{k:k+t-1}) \\
& = \bigl(\begin{smallmatrix}
B_{k+t-1} & A_{k+t-1}B_{k+t-2} & \cdots & \rbr{\prod_{l=1}^{t-1}A_{k+l}}B_k
\end{smallmatrix}\bigr)\in\mR^{n_x\times tn_u}.
\end{align*}
\end{definition}

Based on these two definitions, we are ready to introduce the (uniform) SOSC and the controllability condition in the next two assumptions.

\begin{assumption}[Uniform SOSC]\label{ass:unif:SOSC}
For a prespecified $\bd$ and a primal-dual solution $(\tz, \tlambda) = (\tz(\bd), \tlambda(\bd))$ of $\P_{0:N}(\bd)$ in \eqref{pro:1}, we assume 
\begin{align}\label{equ:SOSC}
H_{\rm Re}(\tz, \tlambda_{0:N-1}; \bd_{0:N-1}) \succeq \gamma_H I
\end{align}
for some positive constant $\gamma_H$ independent from $N$.
\end{assumption}

\begin{assumption}[Controllability condition]\label{ass:control}
There exist constants $\gamma_C, t>0$, independent from $N$, such that $\forall k\in[N-t]$
\begin{align}\label{equ:control}
\Xi_{k, t_k}\Xi_{k, t_k}^\top\succeq \gamma_C I, \text{\ \ for some\ } t_k\in[1, t],
\end{align}
where $\Xi_{k, t_k}$ is evaluated at $(\tz_{k:k+t_k-1}; \bd_{k:k+t_k-1})$. 
\end{assumption}

Although SOSC is sufficient but not necessary for $(\tz, \tlambda)$ being a local solution, it is widely assumed in sensitivity analysis since $(\tz, \tlambda)$ can be guaranteed to be a strict local solution of $\P_{0:N}(\bd)$. Thus, the directional directives of ($\tz(\bd)$, $\tlambda(\bd)$) with respect to $\bd$, given by the solutions of the quadratic approximation of \eqref{pro:1}, are well defined. See Theorem 5.53 in \cite{Bonnans2000Perturbation} for the theoretical underpinnings of this statement.

Assumption \ref{ass:control} is borrowed from \cite{Xu2018Exponentially, Xu2019Exponentially, Na2020Exponential} (with slight modifications). In principle, it guarantees that the linearized dynamic system, $\bx_{k+1} = A_k\bx_k + B_k\bu_k$, can be controlled in at most $t$ stages: for any initial state $\bx_k$, one can let the system transit to any terminal state $\bx_{k+t_k}$ by setting $\bu_{k:k+t_k-1}$ properly, with $t_k$ not exceeding $t$ for all $k$. In fact, there are several equivalent statements of Assumption \ref{ass:control}. For example, in both  \cite{Xu2018Exponentially} and \cite{Xu2019Exponentially} the authors considered a truncation of an infinite-horizon linear-quadratic problem. They assumed \eqref{equ:control} to hold for any $k\in[N-1]$ since they have access to $\{A_k, B_k\}_{k=N}^{\infty}$. Alternatively, in  \cite{Na2020Exponential} the authors assumed \eqref{equ:control} for $k\in[N-1]$ and $t_k\in[1, (N-k)\wedge t]$. Comparing with these related works, we simplify the condition further by considering $k\in[N-t]$ only, since for $k\in(N-t, N-1]$ the dynamics can  evolve only $t-1$ stages at most, so that the controllability on the tail is not essential. We mention that the results from \cite{Xu2018Exponentially, Xu2019Exponentially, Na2020Exponential} all hold if we use Assumption \ref{ass:control} in place of their controllability assumptions.

Under Assumptions \ref{ass:unif:SOSC} and \ref{ass:control} and extra boundedness conditions (that we introduce later), the authors  of \cite{Na2020Exponential} established the sensitivity result for \eqref{pro:1}. They showed that if $\bd(\epsilon) = \bd + \epsilon \be_i$ is the perturbation of $\bd$ at stage $i$, with $\be_i\in\mR^{Nn_d}$ being any unit vector with support within $(i \cdot n_d, (i+1)n_d]$, then $\|\tz_k\big(\bd(\epsilon)\big) - \tz_k(\bd)\|$ decays exponentially with $|k-i|$ for small enough $\epsilon$. By the uniformity in Assumptions \ref{ass:unif:SOSC} and \ref{ass:control}, the decay rate is a function of $\gamma_H$, $\gamma_C$, $t$ and is independent from $N$ as well. \textit{Our work is built on this result but used differently than in the original reference}. Specifically, we apply the sensitivity results to show the decay structure of the KKT inverse, and then focus on the local convergence of the MPC strategy around $(\tz, \tlambda)$. We first define a \textit{neighborhood} concept.

\begin{definition}\label{def:3}

Let $\epsilon$ be a strictly positive constant. For any $k\in[N-1]$, the $\epsilon$-cube of $(\tz_k, \tlambda_k)$ is $\{(\bz_k, \blambda_k): \|\bx_k - \tx_k\| \vee \|\bu_k - \tu_k\|\vee \| \blambda_k - \tlambda_k\|\leq \epsilon\}$, denoted by $\N_\epsilon(\tz_k, \tlambda_k)$. We define $\N_\epsilon(\tz_N)$ similarly. Given the point $(\tz_{k_1:k_2}, \tlambda_{k_1:k_2})$, we let the $\epsilon$-hypercube be $\N_\epsilon(\tz_{k_1:k_2}, \tlambda_{k_1:k_2}) \coloneqq\otimes_{j=k_1}^{k_2}\N_\epsilon(\tz_j, \tlambda_j)$, where $\otimes$ denotes the Cartesian product. When $k_2 = N$, we adjust it to $(\tz_{k_1:N}, \tlambda_{k_1:N-1})$. Moreover, for the full-horizon~point $(\tz, \tlambda_{0:N-1})$ and any integer $M$, we let \textit{$(M, \epsilon)$-neighborhood of $(\tz, \tlambda_{0:N-1})$} be
\begin{align*}
\N_{\epsilon, M}(\tz, &\tlambda_{0:N-1}) = \big\{(\bz, \blambda_{0:N-1}): \exists [k_1, k_2]\subseteq [0, N], \\
&k_2-k_1 = M, \text{\ s.t.\ } \bz = (\tz_{0:k_1-1}; \bz_{k_1:k_2}; \tz_{k_2+1:N}), \\
&\blambda = (\tlambda_{0:k_1-1}; \blambda_{k_1:k_2}; \tlambda_{k_2+1:N-1}), \text{\ and\ }\\
&(\bz_{k_1:k_2}, \blambda_{k_1:k_2})\in\N_\epsilon(\tz_{k_1:k_2}, \tlambda_{k_1:k_2})\big\}.
\end{align*}
\end{definition}

In other words, $\N_{\epsilon, M}(\tz, \tlambda_{0:N-1})$ consists of all the~vectors $(\bz, \blambda_{0:N-1})$ for which $M$ consecutive components are perturbed around those of  $(\tz, \tlambda_{0:N-1})$. Note that this set is a union of relative neighborhoods, but it is not a topological neighborhood. It does not contain an open set around $(\tz, \tlambda_{0:N-1})$, in contrast with $\N_{\epsilon}(\tz, \tlambda_{0:N-1})$, which does.

Recall that all quantities defined in Definitions \ref{def:1} and \ref{def:2} are not related to $\blambda_{-1}$, the multiplier associated to the initial condition. Thus, considering a neighborhood around it is not necessary. We sometimes drop  the subscript of $\blambda_{0:N-1}$ and abuse the notation $\blambda$, whose index range will be clear from the context. For simplicity, we also let $\N_\epsilon(\tz_{k_1:k_2}) = \{\bz_{k_1:k_2}: \|\bx_j - \tx_j\| \vee \|\bu_j - \tu_j\|\leq \epsilon, \forall j\in[k_1, k_2]\}$ and similarly have $\N_\epsilon(\tlambda_{k_1:k_2})$. Based on Definition \ref{def:3}, we have
\begin{multline}\label{equ:setseq}
(\tz, \tlambda) = \N_{\epsilon, 0}(\tz, \tlambda)\subseteq \cdots\subseteq\N_{\epsilon, M}(\tz, \tlambda)\\ \subseteq\cdots\subseteq\N_{\epsilon, N}(\tz, \tlambda) = \N_\epsilon(\tz, \tlambda).
\end{multline}

In order to end this section, we extend Assumptions \ref{ass:unif:SOSC} and~\ref{ass:control} to the $(M, \epsilon)$-neighborhood of $(\tz, \tlambda)$.

\begin{assumption}[Uniform positive definiteness of the reduced Hessian in the $(M, \epsilon)$-neighborhood]\label{ass:M:SOSC}

We say that the reduced Hessian at $(\tz, \tlambda)$ is uniformly positive definite in $(M, \epsilon)$-neighborhood if there exists a positive constant $\gamma_H$ independent from $N$ such that 
\begin{align*}
H_{\rm Re}(\bz, \blambda; \bd)\succeq \gamma_H I, \quad \forall \big(\bz, \blambda\big)\in \N_{\epsilon, M}(\tz, \tlambda).
\end{align*}

\end{assumption}

\begin{assumption}[Controllability condition in $\epsilon$-hypercube]\label{ass:M:control}

There exist constants $\gamma_C, t>0$, independent from $N$, such that, $\forall k\in[N-t]$, condition \eqref{equ:control} holds for any evaluation point $\bz_{k:k+t_k-1}\in \N_\epsilon(\tz_{k:k+t_k-1})$.
\end{assumption}

For fixed $N$, Assumptions \ref{ass:M:SOSC} and \ref{ass:M:control} are implied by Assumptions \ref{ass:unif:SOSC} and \ref{ass:control} if $\{f_k, g_k\}$ are twice continuously~differentiable, but they assume the respective parameters to be uniform in $N$. For conciseness, we do not change~the notation for $\gamma_H$, $\gamma_C$ between assumptions. We note that in Assumption \ref{ass:M:SOSC} we require the positive definiteness of the reduced Hessian only in $\N_{\epsilon, M}(\tz, \tlambda)$ instead of in $\N_{\epsilon, N}(\tz, \tlambda)$. Recall that $M$ is the length of a single receding horizon, which is much smaller than $N$ and does not grow with $N$. Essentially, Assumption \ref{ass:M:SOSC} means the positive definiteness of the reduced Hessian is robust to perturbations on every $M$ consecutive stages.

In the next section, we formalize the subproblem on each receding horizon and propose our MPC strategy. The usefulness of these assumptions will become apparent in Section \ref{sec:4}.

%% file: sec3.tex
\section{Fast Lag-$L$ Online MPC Strategy}\label{sec:3}

In our setup, each receding horizon has length $M$ (except possibly the last one), and two successive horizons have lag $L$. The total number of horizons we have is given by
\begin{align*}
T = \lceil\frac{N-M}{L}\rceil + 1. 
\end{align*}
Note that the last horizon may have shorter length. Furthermore, the initial and terminal stages of each horizon are given by
\begin{align*}
n_1(i) = (i-1)L, \quad n_2(i) = \big(n_1(i) + M\big)\wedge N, \text{\ } \forall i\in[1,T].
\end{align*}
Therefore, the entire horizon is further decomposed into $[0, N]\subseteq \cup_{i = 1}^T[n_1(i), n_2(i)]$. For simplicity, we assume $M =SL$ for some integer $S>2$. We define two important quantities.

\begin{definition}\label{def:sTk}

(a) Given any stage $k\in[N]$, we let $T_k\in[1, T]$ be the index of the last subproblem that contains the stage $k$. By simple calculation, 
\begin{align}\label{equ:Tk}
T_k& = \max\{i: k\in[n_1(i), n_2(i)]\} \nonumber\\
= & \begin{cases}
i, & \text{\ if\ } k\in[n_1(i), n_1(i+1)\big) \text{\ for\ } i\in[1, T-1],\\
T, & \text{\ if\ } k\in[n_1(T), N].
\end{cases}
\end{align}
(b) Given subproblem $i\in[1, T]$ and stage $k\in[n_1(i), n_2(i)]$, we let $s(k, i)$ denote the number of times that stage $k$ was ``scanned,'' that is,  was part of an active receding-horizon \textit{before} subproblem $i$. In particular,
\begin{align*}
s(k, i) = \sum_{h=1}^{i-1}\pmb{1}_{k\in[n_1(h), n_2(h))}.
\end{align*}
For example, $s(k, 1) = 0$, $\forall k\in[0, M]$; $s(k, 2) = 1$ if $k\in[L, M)$ and $0$ if $k\in[M, M+L]$.

\end{definition}

Based on this definition, we can further define a critical quantity, $s_k \coloneqq s(k, T_k)$, which characterizes the total number of times that stage $k$ is scanned \textit{before its last occurrence} as we advance the process. One can immediately see that as $k$ moves from $0$ to $N$, $s_k$ increases from $0$ to $S-1$ and then decreases to $0$ again. At most, $s_k$ will increase or decrease by $1$ each time. We will assume $T\geq S-1$, which is the case of interest (since we are in the regime $N \gg M = SL$). Then for $k\in [0, N]$, we have $s_N = 0$; $s_k = s-1$ if $k\in[n_1(s), n_1(s+1)) \cup [n_1(T)+M-n_1(s+1), \rbr{n_1(T)+M - n_1(s)}\wedge N)$ for some $s\in[1, S-1]$; and $s_k = S-1$ if $k\in[n_1(S), n_1(T)+L\big)$. We plot $s_k$ in Figure~\ref{fig:sk} to make its definition more intuitive. Since the middle stages from $n_1(S)$ to $n_1(T)+L$ are scanned by $S$ subproblems, we expect that the increased number of Newton steps experienced by them will result in a smaller error relative to the exact MPC solution, compared with stages closer to the endpoints of the overall time horizon. We also mention that, for general $M$, $s_k$ for middle stages $k$ will vary between $\lfloor \frac{M}{L}\rfloor - 1$ and $\lfloor \frac{M}{L}\rfloor$.

\begin{figure}[!tp]
	\centering     
	\includegraphics[width=9cm,height=1.6cm]{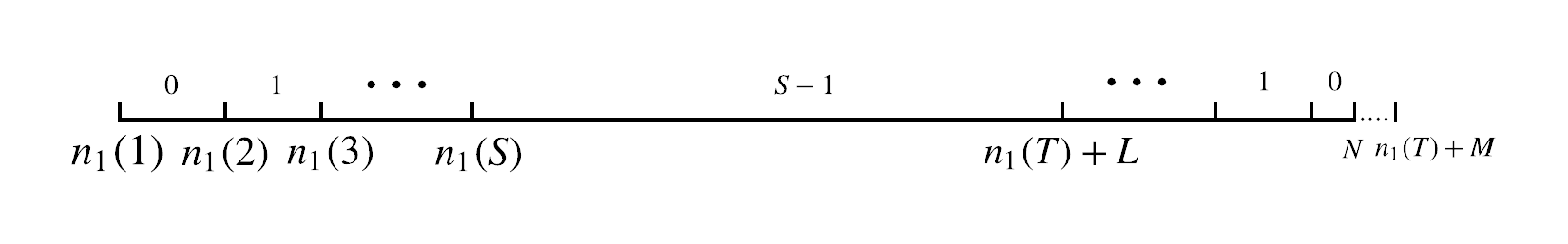}
	\caption{\textit{Variation of $s_k$ along whole horizon. The integer above the line indicates the value of $s_k$, and the integer under the line indicates the interval of receding horizons.}}\label{fig:sk}
\end{figure}

We build on this setup, and define subproblems within each receding horizon. Suppose a \textit{guess} point $(\bz^0, \blambda^0) \in \N_\epsilon(\tz, \tlambda)$ is available. The $i$-th subproblem on $[n_1(i), n_2(i)]$, denoted by $\P_i(\tbd_i)$, is defined as
{\small
\begin{align}\label{pro:2}
\min_{\substack{\bx_{n_1(i):n_2(i)}, \\ \bu_{n_1(i):n_2(i)-1}}}\  & \sum_{k=n_1(i)}^{n_2(i)-1} g_k(\bz_k; \bd_k) + g_{n_2(i)}(\bx_{n_2(i)}, \bu_{n_2(i)}^0; \bd_{n_2(i)}) \nonumber\\
&\quad - (\blambda_{n_2(i)}^0)^\top f_{n_2(i)}(\bx_{n_2(i)}, \bu_{n_2(i)}^0; \bd_{n_2(i)}) \nonumber\\
&\quad + \frac{\mu\|\bx_{n_2(i)} - \bx^0_{n_2(i)}\|^2}{2}, \nonumber\\
\text{s.t.}\text{\ \ } & \bx_{k+1} = f_k(\bz_k; \bd_k), \text{\ } k \in[n_1(i), n_2(i)-1],\\
& \bx_{n_1(i)} = \bbx_{n_1(i)}, \nonumber
\end{align}
}where $\tbd_i = (\bbx_{n_1(i)}; \bd_{n_1(i):n_2(i)}; \bz_{n_2(i)}^0; \blambda_{n_2(i)}^0)$ is the \textit{extended} reference variable. We deliberately use different notations for the initial state since in our procedure $\bbx_{n_1(i)}$ is inherited from the previous subproblem. Compared with classical MPC \cite{Diehl2005Nominal} and parametric optimization \cite{Diehl2002Real, Zavala2010Real, Dinh2012Adjoint}, in Problem \eqref{pro:2} we modify the terminal objective by including the corresponding dual function and an extra quadratic term. We note that the scale parameter $\mu$ does not depend on a specific~$i$. Later, we will set it globally such that all subproblems can be uniformly analyzed and enjoy similar properties. As usual, when $i = T$ (i.e., $n_2(i)  = N$), all modifications should be reinstated, and we simply use $g_N(\bx_N)$ as the terminal objective.

\begin{remark}
The subproblem \eqref{pro:2} can be set up differently under stronger assumptions. For example, when $g_k$, $f_k$ are separable with respect to the state and control, one need not specify $\bu_{n_2(i)}^0$ in the terminal objective; when $g_{n_2(i)}$ is strongly convex, one can let $\blambda_{n_2(i)}^0 = 0$ and $\mu = 0$ (see \cite{Wang2010Fast}); when $B_k = \nabla_{\bu_k}^\top f_k$ (see Definition \ref{def:1}) has full row rank in the neighborhood, one can add the extra condition $\bx_{n_2(i)} = \bx_{n_2(i)}^0$ to make $\bx_{n_2(i)}$ be fixed without violating LICQ. All these modifications make \eqref{pro:2} well posed in specific cases. In our general case, $\blambda_{n_2(i)}^0$ and $\mu $ are needed to ensure that the positive definiteness of the reduced Hessian for \eqref{pro:1} induces the positive definiteness of the reduced Hessian for \eqref{pro:2}.

\end{remark}

Analogous to Definition \ref{def:1}, we define the Lagrangian and its Jacobian and Hessian matrices for subproblem $i$ with $i\in[1, T]$.

\begin{definition}\label{def:4}

We let $\tbx_i = \bx_{n_1(i):n_2(i)}$, $\tbu_i = \bu_{n_1(i):n_2(i)-1}$, $\tblambda_i = \blambda_{n_1(i)-1:n_2(i)-1}$ be the aggregated primal and dual variables; $\tbz_i = (\tbx_i, \tbu_i) = (\bx_{n_1(i)};\bu_{n_1(i)};\ldots;\bx_{n_2(i)})$ be the ordered primal variables. The Lagrange function of the $i$-th subproblem is given by
{\small
\begin{align*}
\mL^i(\tbz_i, \tblambda_i;\tbd_i) &=  \sum_{k=n_1(i)}^{n_2(i)-1}\mL_k(\bz_k, \blambda_{k-1:k};\bd_k) \\
&\quad + \mL_{n_2(i)}(\bx_{n_2(i)},\bu_{n_2(i)}^0, \blambda_{n_2(i)-1}, \blambda_{n_2(i)}^0;\bd_{n_2(i)})\\
& \quad  + \frac{\mu}{2}\|\bx_{n_2(i)} - \bx_{n_2(i)}^0\|^2 - \blambda_{n_1(i)-1}^\top\bbx_{n_1(i)}.
\end{align*}
}Further, given an evaluation point $(\tbz_i, \tblambda_i; \tbd_i)$, we define the Jacobian and Hessian matrices of the Lagrangian of the $i$-th subproblem as
{\small
\begin{align*}
\small G^i(\bz_{n_1(i):n_2(i)-1}&; \bd_{n_1(i);n_2(i)-1}) \\
&= \left(\begin{smallmatrix}
I\\
-A_{n_1(i)} & -B_{n_1(i)} & I\\
&&\ddots & \ddots\\
&&&-A_{n_2(i)-1} & -B_{n_2(i)-1} & I
\end{smallmatrix}\right),\\
\small H^i(\tbz_i, \tblambda_i; \tbd_i) = \small &\diag(H_{n_1(i)}, \ldots, H_{n_2(i)-1}, H^i_{n_2(i)}),
\end{align*}
}where $\{A_k, B_k, H_k\}_{k = n_1(i)}^{n_2(i)-1}$ are defined in Definition \ref{def:1}. Specifically, $A_k, B_k$ are evaluated at $(\bz_k, \bd_k)$, and $H_k$ is evaluated at $(\bz_k, \blambda_k, \bd_k)$. Here, the last block of $H^i$ is
\begin{multline*}
\small H^i_{n_2(i)}(\bx_{n_2(i)}; \bu_{n_2(i)}^0, \blambda_{n_2(i)}^0, \bd_{n_2(i)}) \\= Q_{n_2(i)}(\bx_{n_2(i)}, \bu_{n_2(i)}^0, \blambda_{n_2(i)}^0; \bd_{n_2(i)}) + \mu  I.
\end{multline*}
Furthermore, let $Z^i$ denote a matrix whose columns span the null space of $G^i$ with orthonormal columns. The reduced Hessian is $ H^i_{\rm Re}(\tbz_i, \tblambda_i; \tbd_i) = (Z^i)^\top H^iZ^i$. For all notations such as $G^i, Z^i, H^i$, $H^i_{\rm Re}$, we may replace superscript $i$ by $n_1(i):n_2(i)$ to explicitly indicate the horizon location. 
\end{definition}

We now set the stage for proposing the fast lag-$L$ online MPC strategy. Starting from the first subproblem $\P_1(\tbd_1)$ with $\tbd_1 = (\bbx_0; \bd_{0:M}; \bz_M^0; \blambda_M^0)$, we initialize variables by letting $\tbz_1^0 = (\bbx_0; \bu_0^0;\bz_1^0; \ldots; \bz_{M-1}^0; \bx_M^0)$ and $\tblambda_1^0 = \blambda_{-1:M-1}^0$. Given $(\tbz_1^0, \tblambda_1^0, \tbd_1)$, we calculate the Newton step by solving the following linear system,
\begin{align}\label{equ:Newton}
\left(\begin{smallmatrix}
H^1 & (G^1)^\top\\
G^1 & 0
\end{smallmatrix}\right)\left(\begin{smallmatrix}
\Delta\tbz_1\\
\Delta\tblambda_1
\end{smallmatrix}\right) = -\left(\begin{smallmatrix}
\nabla_{\tbz_1}\mL^1\\
\nabla_{\tblambda_1}\mL^1\\
\end{smallmatrix}\right),
\end{align}
and then update the iterate as
\begin{align}\label{equ:Update}
\left(\begin{smallmatrix}
\tbz_1^1\\
\tblambda_1^1
\end{smallmatrix}\right) \coloneqq \left(\begin{smallmatrix}
\tbz_1^0\\
\tblambda_1^0
\end{smallmatrix}\right) + \left(\begin{smallmatrix}
\Delta\tbz_1\\
\Delta\tblambda_1
\end{smallmatrix}\right). 
\end{align}
Then, we move to subproblem $\P_2(\tbd_2)$. For general $\P_i(\tbd_i)$, we first let $\tbd_i = (\bx_{n_1(i), i-1}^1; \bd_{n_1(i):n_2(i)}; \bz_{n_2(i)}^0; \blambda_{n_2(i)}^0)$, where $\bx_{k, i}^\Id$ (similar for $\bu$, $\bz$, $\blambda$) denotes the state iterate of stage $k$ in subproblem $i$, with $\Id = 1$ indicating the output and $\Id = 0$ indicating the input. Then, $\tbz_i^0 = (\bz_{n_1(i), i}^0;\ldots; \bz_{n_2(i)-1, i}^0;\bx_{n_2(i),i}^0)$ is initialized as follows:
\begin{align}\label{equ:trans1}
\bz_{k, i}^0 = \begin{cases}
\bz_{k, i-1}^1 & \text{\ for\ } k\in[n_1(i), n_2(i) - 2L],\\
\bz_k^0 & \text{\ for\ } k\in(n_2(i) - 2L, n_2(i)-1],
\end{cases}
\end{align}
and $\bx_{n_2(i),i}^0 = \bx_{n_2(i)}^0$. Similarly, $\tblambda_i^0 = \blambda_{n_1(i)-1:n_2(i)-1, i}^0$ is initialized as
\begin{align}\label{equ:trans2}
\blambda_{k, i}^0 = \begin{cases}
\blambda_{k, i-1}^1 & \text{\ for\ } k\in[n_1(i)-1, n_2(i) - 2L],\\
\blambda_k^0 & \text{\ for\ } k\in(n_2(i) - 2L, n_2(i)-1].
\end{cases}
\end{align}
Using the tuple $(\tbz_i^0, \tblambda_i^0, \tbd_i)$, we calculate the Newton's direction $(\Delta\tbz_i, \Delta\tblambda_i)$, as shown in \eqref{equ:Newton}, and update the iterate to get $(\tbz_i^1, \tblambda_i^1)$, as shown in \eqref{equ:Update}. The outputs are given by (see Definition \ref{def:sTk} for $T_k$)
\begin{align}\label{equ:output}
\hbz_k = \bz_{k, T_k}^0 \text{\ \ and\ \ } \hblambda_k = \blambda_{k, T_k}^0, \text{\ \ } \quad \forall k.
\end{align}
Note that based on ``output-input" transition in \eqref{equ:trans1}, the initial constraint of each subproblem is always attained.

\begin{algorithm}[!tp]
\caption{Fast Lag-$L$ Online MPC Strategy}
\begin{algorithmic}[1]\label{alg:1}
\STATE Input: whole-horizon length $N$, receding-horizon length $M$, lag $L$, objective functions $\{g_k\}_{k=0}^N$, dynamic functions $\{f_k\}_{k=0}^{N-1}$, reference $\{\bd_k\}_{k=-1}^{N-1}$ ($\bbx_0 = \bd_{-1}$), initial point $(\bz^0, \blambda^0)$ with $\bx_0^0 = \bbx_0$;
\STATE Let $T = \lceil \frac{N-M}{L}\rceil + 1$;
\STATE Set $\big\{\big(\bz_{k, 0}^{1}, \blambda_{k, 0}^{1}\big) = \big(\bz_k^0, \blambda_k^0\big)\big\}_{k\in[0, M-2L]}$, $\blambda_{-1,0}^{1}= \blambda_{-1}^0$;

\FOR{$i = 1:T$}
		\STATE Let $n_1 = (i-1)L$, $n_2 = (n_1+M)\wedge N$, $n_m = n_1 + M - 2L$;
		\STATE Let $\tbd_i = (\bx_{n_1, i-1}^1; \bd_{n_1:n_2}; \bz_{n_2}^0; \blambda_{n_2}^0)$;
		\STATE Set $\big\{\big(\bz_{k, i}^{0}, \blambda_{k, i}^{0}\big) = \big(\bz_{k, i-1}^{1}, \blambda_{k, i-1}^{1}\big)\big\}_{k\in[n_1:n_m]}$ and $\big\{\big(\bz_{k, i}^{0}, \blambda_{k, i}^{0}\big) = \big(\bz_{k}^{0}, \blambda_{k}^{0}\big)\big\}_{k\in(n_m:n_2)}$;
		\STATE Set $\blambda_{n_1-1, i}^{0} = \blambda_{n_1-1, i-1}^{1}$, $\bx_{n_2,i}^0 = \bx_{n_2}^0$;
		\STATE Order the iterates in  lines 7--8 properly, and define $\tbz_i^0 = (\bz_{n_1:n_2-1, i}^0; \bx_{n_2, i}^0)$ and $\tblambda_i^0 = \blambda_{n_1-1:n_2-1,i}^0$;
		\STATE Conduct \eqref{equ:Newton}-\eqref{equ:Update} for $i$th subproblem using $(\tbz_i^0, \tblambda_i^0, \tbd_i)$, and derive $(\tbz_i^1, \tblambda_i^1)$;
		\STATE  Denote $\tbz_i^1 = (\bz_{n_1:n_2-1, i}^1; \bx_{n_2, i}^1)$, $\tblambda_i^1 = \blambda_{n_1-1:n_2-1,i}^1$;
		\ENDFOR
		\STATE Output: $\{\hat{\bx}_k = \bx_{k, T_k}^{0}\}_{k = 0:N}$,  $\{\hat{\bu}_k = \bu_{k, T_k}^{0}\}_{k = 0:N-1}$, $\{\hat{\blambda}_k = \blambda_{k, T_k}^{0}\}_{k = 0:N-1}$, and $\hblambda_{-1} = \blambda_{-1}^0$, where $T_k$ is defined in \eqref{equ:Tk}.
	\end{algorithmic}
\end{algorithm}
The complete ``one Newton step per receding horizon" algorithm is presented in Algorithm \ref{alg:1}. We next make a few remarks about Algorithm \ref{alg:1}.

\begin{remark}\label{rem:KKT:inv}

To make Algorithm \ref{alg:1} well posed, we need the KKT matrix in \eqref{equ:Newton}, evaluated at $(\tbz_i^0, \tblambda_i^0; \tbd_i)$, to be nonsingular for any $i\in[1, T]$. In fact, this is guaranteed by LICQ $+$ positive definiteness of the reduced Hessian; see Lemma 16.1 in \cite{Nocedal2006Numerical}. By the structure of $G^i$ in Definition~\ref{def:4}, we know that LICQ always holds. To show  (uniform) positive definiteness of the reduced Hessian requires two steps. (i) We need show that the algorithm is stable in the sense that \textit{all iterates in all subproblems stay in the neighborhood $\N_{\epsilon}(\tbz^\star, \tblambda^\star)$}, where $(\tbz^\star, \tblambda^\star)$ denotes the truncated true solution. This suggests that the initial point will not jump out of the neighborhood, where we have no conclusions on the behavior of the Lagrangian in the reduced space. (ii) We need to show that positive definiteness of reduced Hessian for a subproblem can be inherited from the full problem: if the reduced Hessian for the full problem is positive definite at one point, then the reduced Hessian at the corresponding truncated (receding horizon problem) point is also positive definite, provided $\mu$ is chosen properly (and uniformly). Details are provided in the next section.
	
\end{remark}

\begin{remark}
In \eqref{equ:Update}, we choose a step size of one in all iterations since we are interested in the local behavior of online MPC. As is common in such setups, we assume $(\bz^0, \blambda^0)$ is sufficiently close to  $(\tz, \tlambda)$\footnote{For example, in convex optimization it is standard to assume that the radius of this neighborhood is bounded by the reciprocal of the condition number; see (9.33) in \cite{Boyd2004Convex}.} such that Newton's method already arrives at the second (quadratically convergent) phase. A similar setup is studied in Algorithm 18.1 in \cite{Nocedal2006Numerical} and \cite{Diehl2005Real, Diehl2005Nominal, Zavala2010Real, Gros2016linear}. Some extensions on \eqref{equ:Update}, such as incorporating it with an Armijo or Wolfe backtracking line search, are worth studying as well.	
\end{remark}

\begin{remark}\label{rem:2}

We adopt two critical techniques in Algorithm \ref{alg:1}.
\begin{enumerate}[label=(\alph*),topsep=1pt]
\setlength\itemsep{-0.0em}
\item ``Discard the tail": for each subproblem $i$, the initial point within stages $[n_2(i-1), n_2(i)]$ are set to $(\bz^0, \blambda^0)$. But one big difference from typical MPC strategies is that we also let $(\bz^0, \blambda^0)$ initialize stages within $(n_2(i-1)-L, n_2(i-1)]$, the last $L$ stages of $(i-1)$-th subproblem. Thus, we discard iterates $\bz_{n_2(i-1)-L+1:n_2(i-1), i-1}^1$ (same for $\blambda$) when moving to the $i$-th subproblem.
\item ``Stop early": for each stage $k$, our output $(\hbz_k, \hblambda_k)$ is set to be the initial point of the last subproblem that is going to update the stage $k$, instead of the updated point. More specifically, we use $\bz_{k, T_k}^0$ (same for $\blambda$), instead of $\bz_{k, T_k}^1$, as the final output.
\end{enumerate}
Both these algorithmic strategies are adopted for similar reasons. We aim to show that middle stages in the horizon see the benefit of multiple Newton steps (since they are scanned multiple times), but the terminal stages, which are iterated by only one Newton step, may have a precision no better than without iteration. Using these iterates may subsequently affect the precision of the entire process. To improve accuracy, we  discard those iterates that are updated just once on the tail. The  usefulness of this strategy will be seen in the analysis.
\end{remark}

\begin{remark}\label{rem:setup}
(Algorithm and subproblem formulation comparisons  with RTI schemes). Both \cite{Diehl2005Nominal} and \cite{Wang2010Fast} can be viewed as lag-$1$ online MPC, albeit \cite{Diehl2005Nominal} used a perturbed Hessian in Newton equation and \cite{Wang2010Fast} performed multiple Newton steps and adopted sparse factorization technique when solving Newton equation. Our algorithm is similar to such online approaches except that we allow the flexibility to choose larger lags and adopt an ad-hoc strategy to update the boundary iterates of each subproblem (see Remark \ref{rem:2}). \cite{Diehl2002Real, Zavala2010Real, Dinh2012Adjoint} studied solving parametric nonlinear optimization in real time. Relative to the MPC framework, \cite{Diehl2002Real, Zavala2010Real} are similar to \cite{Diehl2005Nominal} since they perform a single Newton step as well, while \cite{Dinh2012Adjoint} performs a single predictor-corrector step that is particularly designed for convex objective with nonlinear constraints. \cite{Diehl2005Real} is not precisely comparable with the aforementioned MPC schemes since the subproblems horizons are embedded sequentially.
	
Our subproblem formulation makes our analysis different from the above work. \cite{Diehl2005Nominal, Wang2010Fast, Zanelli2021Lyapunov} studied time-invariant optimal control and subproblems at different time instants only differ by the initial state. This can be viewed as parametric problems with continuous parameter being the given initial state, studied in \cite{Diehl2002Real, Zavala2010Real, Dinh2012Adjoint}. This setup is not applicable here because our two subproblems may have different solutions even if they are successive with the same initial state. Thus, our subproblems are not parameterized by any continuous parameters. The formulation \eqref{pro:2} plays a key role in analysis and the regularization term appears to be only considered here for the purpose of convergence.

\end{remark}

\begin{remark}\label{rem:overhead}

We discuss the overhead of the proposed algorithm. We have $T$ subproblems in total, and for each subproblem the computational complexity is dominated by solving the Newton equation in \eqref{equ:Newton}. In particular, the KKT matrix is a square matrix with dimension $(2(M+1)n_x + Mn_u)$. For sparse $LDL^\top$ factorization, \cite{Rao1998Application,Wang2010Fast} showed that solving \eqref{equ:Newton} only costs $O(M(n_x + n_u)^3)$ flops, which is linear in $M$ (see Section C in \cite{Wang2010Fast}). Thus, the total computational complexity is $O(TM(n_x+n_u)^3)\asymp O(NM(n_x+n_u)^3/L)$. Note that such complexity is also applicable for other lag-$1$ online MPC schemes \cite{Diehl2005Nominal} except that our result depends on the lag explicitly (\cite{Wang2010Fast} has larger complexity since it performs multiple Newton steps for each subproblem). It is worth pointing out that the lag $L$ here is a trade-off between computational complexity and precision. For large $L$, we have fewer subproblems and hence less computation, while the total number of ``scans" of each stage, i.e. $s_k\in[0, M/L]$, will also decrease and thus the iterates are less precise. Also note that $O(NM(n_x+n_u)^3/L)$ is the same as performing $M/L$ Newton iterations for the full problem directly and $N/L$ iterations of the online algorithm. Since $N/L$ is very large in our case this points out to the advantage of online MPC when time to solution is preferred. An important point to note, since we advocate tuning $M$ and $L$, is that the cost does increase linearly with $M$, but we prove that the accuracy increases exponentially, so the trade-off of increasing $M$ for accuracy may be worth making. 
	
In terms of the memory cost, our algorithm only needs to save all matrices and vectors associated to each subproblem at each time instant, which is dominated by the storage of the KKT matrix. Due to its sparsity, it is easy to see the memory cost is $O(M(n_x+n_u)^2)$, which is also linear in $M$. Such memory cost is standard among online MPC regimes. 


\end{remark}

%% file: sec4.tex
\section{Local Convergence Analysis}\label{sec:4}

In this section, we carry out a theoretical analysis of Algorithm \ref{alg:1}. We will show that the proposed algorithm converges exponentially with respect to the receding-horizon length $M$ on a per-stage basis, that is, with respect to the stage index. In particular, we will show for some constant $\rho\in(0, 1)$ in \eqref{equ:super} that $\|\hbz_k - \tz_k\| \lesssim \rho^M$ for any $k\in[M-L, TL)$  (and a similar result for dual variables). Throughout the analysis, we assume that $(\bz^0, \blambda^0)\in\N_\epsilon(\tz, \tlambda)$ is a fixed, known point and is used when specifying $\tbd_i$, shown in \eqref{pro:2}. 

A sketch of our technical analysis is as follow.
\begin{enumerate}[label=(\alph*),topsep=1pt]
\setlength\itemsep{-0.0em}
\item We first show that if $(\tbz_i^0, \tblambda_i^0) \in \N_\epsilon(\tbz_i^\star, \tblambda_i^\star)$, then the reduced Hessian $H^i_{\rm Re}(\tbz_i^0, \tblambda_i^0; \tbd_i)$ is uniformly positive definite provided $\mu$ is set large enough.
\item We delve into Newton's iteration and derive the one-step error recursion. Based on the recursion, we can further study the stability of the iterates tracking to the solution $(\tz, \tlambda)$.
\item We derive the local convergence rate for all stages based on the preceding two steps.
\end{enumerate}
To clarify, $(\tbz_i^\star, \tblambda_i^\star)$ denotes the solution of \eqref{pro:1} truncated on subproblem $i$, that is, $\tbz_i^\star = (\tz_{n_1(i):n_2(i)-1}; \tx_{n_2(i)})$ and $\tblambda_i^\star = \tlambda_{n_1(i)-1:n_2(i)-1}$. Similarly, $\tbd_i^\star = \big(\tx_{n_i}; \bd_{n_i:m_i}; \tz_{m_i}; \tlambda_{m_i}\big)$ denotes the underlying true data. We target the preceding three steps in the following three subsections. We now state boundedness and regularity assumptions on Problem \eqref{pro:1}.

\begin{assumption}[Upper boundedness condition]\label{ass:M:Ubound}
There exists a constant $\Upsilon$ such that $\|H_N(\bx_N)\|\leq \Upsilon$ for $\bx_N \in \N_\epsilon(\tx_N)$ and, for any $ k\in[N-1]$ and $(\bz_k, \blambda_k) \in \N_\epsilon(\tz_k, \tlambda_k)$,
\begin{align}\label{equ:bound}
\small \|A_k(\bz_k; \bd_k)\| \vee \|B_k(\bz_k; \bd_k)\| \vee \|H_k(\bz_k, \blambda_k; \bd_k)\|\leq \Upsilon.
\end{align}

\end{assumption}

\begin{assumption}[Lipschitz continuity]\label{ass:M:Lip:cond}

There exists a constant $\Upsilon_L$ such that for any $k$ and any two points $(\bz_k, \blambda_k), (\bz_k', \blambda_k')\in \N_\epsilon(\tz_k, \tlambda_k)$,
\begin{align*}
&\small \|H_k(\bz_k, \blambda_k; \bd_k) - H_k(\bz_k', \blambda_k'; \bd_k)\|\leq \Upsilon_L\|(\bz_k - \bz_k'; \blambda_k - \blambda_k')\|,\\
&\small \|\nabla_{\bz_k}f_k(\bz_k; \bd_k) - \nabla_{\bz_k}f_k(\bz_k'; \bd_k)\|\leq \Upsilon_L\|\bz_k - \bz_k'\|.
\end{align*}	
\end{assumption}

Assumption \ref{ass:M:Ubound} is also assumed in \cite{Xu2018Exponentially, Xu2019Exponentially, Na2020Exponential}. It restricts the problem data to having uniform boundedness property. Assumption \ref{ass:M:Lip:cond} is also standard in analyzing Newton's method. See, for example, equation (9.31) in \cite{Boyd2004Convex}. It always holds if $\{f_k, g_k\}$ are thrice continuously differentiable. Similar to Assumptions \ref{ass:M:SOSC} and \ref{ass:M:control}, Assumptions \ref{ass:M:Ubound} and \ref{ass:M:Lip:cond} are made uniformly over the entire horizon and are required only in a neighborhood of the solution of \eqref{pro:1}.

\subsection{Positive definiteness of the reduced Hessian of subproblems}\label{sec:4.1}

Let us consider the $i$-th subproblem and suppose $(\tbz_i^0, \tblambda_i^0) \in \N_\epsilon(\tbz_i^\star, \tblambda_i^\star)$. We note that $G^i$, $H^i$ in Definition \ref{def:4} are submatrices of $G$, $H$ in Definition \ref{def:1} (except the last block matrix of $H^i$). However, $Z^i$, the matrix that spans the null space of $G^i$, is not a submatrix of $Z$, thus complicating the analysis.

The following lemma shows truncating the horizon from the left maintains the positive definiteness of the reduced Hessian. 

\begin{lemma}\label{lem:1}

For any point $(\bz, \blambda)$, if $H_{\rm Re}(\bz, \blambda_{0:N-1};\bd_{0:N-1})\\\succeq \gamma_H I$, then $\forall k\in[N-1]$,
\begin{align*}
H^{k:N}_{\rm Re}(\bz_{k:N}, \blambda_{k:N-1}; \bd_{k:N-1})\succeq \gamma_H I,
\end{align*}
where $H^{k:N}_{\rm Re}$ denotes the reduced Hessian of the subproblem starting from the stage $k$ to the end.
\end{lemma}

\begin{proof}
See Appendix \ref{pf:lem:1}.
\end{proof}

From Lemma \ref{lem:1}, we see that truncating the full-horizon vector backward in time preserves the positive definiteness of the reduced Hessian. However, the same property does not hold  forward in time in general. Therefore, we rely on the extra controllability condition and  make use of the quadratic penalty formulation in \eqref{pro:2} to make sure that a forward truncation preserves such property as well. We first present an immediate deduction based on Lemma \ref{lem:1}.

\begin{corollary}\label{cor:1}
Suppose Assumption \ref{ass:M:SOSC} holds. For the $T$-th (i.e., last) subproblem, if $(\tbz_T^0, \tblambda_T^0) \in \N_\epsilon(\tbz_T^\star, \tblambda_T^\star)$, then
\begin{align*}
H^T_{\rm Re}(\tbz_T^0, \tblambda_T^0; \tbd_T) \succeq \gamma_H I.
\end{align*}
\end{corollary}

\begin{proof}
See Appendix \ref{pf:cor:1}.
\end{proof}

Next, we consider the first $T-1$ horizons. The following theorem suggests that subproblems can inherit uniform positive definiteness of the reduced Hessian from the full problem provided $\mu$ is uniformly large enough.
\begin{theorem}[SOSC for subproblems]\label{thm:1}
Suppose Assumptions \ref{ass:M:SOSC}, \ref{ass:M:control}, and \ref{ass:M:Ubound} hold. For any $i\in[1, T-1]$, if $(\tbz_i^0, \tblambda_i^0) \in \N_\epsilon(\tbz_i^\star, \tblambda_i^\star)$, then, by setting 
\begin{align}\label{equ:set:mu}
\mu\geq \frac{16\Upsilon(\Upsilon^{6t} - \Upsilon^{4t})}{\gamma_C^2} \coloneqq\mu(\Upsilon, \gamma_C, t),
\end{align}
(where $\Upsilon$ is defined in Assumption \ref{ass:M:Ubound}, $\gamma_C, t$ are defined in Assumption \ref{ass:M:control}) we have
\begin{align*}
H^i_{\rm Re}(\tbz_i^0, \tblambda_i^0; \tbd_i)\succeq\gamma_H I.
\end{align*}

\end{theorem}

\begin{proof}
See Appendix \ref{pf:thm:1}.
\end{proof}

In Theorem \ref{thm:1}, we prove that the reduced Hessian is positive definite for subproblems in the same neighborhood provided that the scale parameter $\mu$ is set larger than the stated threshold. Moreover, the lower bound of the reduced Hessian is still $\gamma_H$. Consequently, we know that the KKT matrices in \eqref{equ:Newton} are invertible in the neighborhood. See Lemma 16.1 in \cite{Nocedal2006Numerical} for a simple proof. We emphasize that $\mu$ is set globally and independent from the horizon location; hence all subproblems have common properties. In the next subsection, we focus on the error recursion of the Newton step and derive the stability of our MPC strategy.

\begin{remark}	
The parameter $\mu$ may enlarge the upper bound~of the Hessian matrices. However, the uniform boundedness~condition in Assumption \ref{ass:M:Ubound} still holds for the subproblems, since the threshold does not grow with $i$ (i.e. $\mu$ need not go to infinity).
\end{remark}

\subsection{One-step error recursion}\label{sec:4.2}

We focus on the Newton iteration and analyze how errors in each stage get updated by it. We establish the stability of our iterates in the sense that they keep staying in the same neighborhood of $(\tz, \tlambda)$ and, combining with consequences of positive definiteness of reduced Hessian, show Algorithm \ref{alg:1} can be carried out successfully. Here, by successfully we mean that the initial point $(\tbz_i^0, \tblambda_i^0)$ is indeed in the neighborhood that ensures the nonsingularity of the KKT matrix.

One of the key results we establish is related to the structure of the KKT inverse, which is implied by the results in \cite{Na2020Exponential}. For ease of presentation, we partition the KKT matrix into several blocks. We  study the full problem first and then illustrate how to apply the same results on each subproblem.

\begin{definition}\label{def:5}
Given the evaluation point $(\bz, \blambda, \bd)$ ($\blambda_{-1}$ and $\bd_{-1}$ are not needed), we define the KKT matrix by
\begin{align*}
K(\bz, \blambda; \bd) = \begin{pmatrix}
H(\bz, \blambda; \bd) & G^\top(\bz_{0:N-1}; \bd)\\
G(\bz_{0:N-1}; \bd) & 0
\end{pmatrix}.
\end{align*}
We partition its inverse as
\begin{align*}
K^{-1}(\bz, \blambda; \bd) = \left(\begin{smallmatrix}
\cbr{K^{-1}_{(i, j), 1}}_{i, j = 0}^N & \cbr{K^{-1}_{(i, j), 2}}_{i = 0, j = -1}^{N, N-1} \\
\cbr{(K^{-1}_{(i, j), 2})^\top}_{i = 0, j = -1}^{N, N-1} & \cbr{K^{-1}_{(i, j), 3}}_{i, j = -1}^{N-1}
\end{smallmatrix}\right).
\end{align*}
In particular, $\forall i, j\in[N]$, $K^{-1}_{(i, j), 1} \in\mR^{(n_x+n_u)\times (n_x+n_u)}$ is the $(i, j)$-block corresponding to  stage $i$ in the row and  stage $j$ in the column. Analogously, $K^{-1}_{(i, j), 2} \in\mR^{(n_x+n_u)\times n_x}$ and $K^{-1}_{(i, j), 3} \in\mR^{n_x\times n_x}$. We can also define similar partitions for $K^i(\tbz_i, \tblambda_i; \tbd_i)$, the KKT matrix for subproblem $i$. Its blocks are denoted by $(K^i)^{-1}_{(j_1, j_2), h}$ with $h = 1,2,3$.
\end{definition}

The next lemma characterizes the structure of $K^{-1}$.

\begin{lemma}[Structure of KKT inverse]\label{lem:2}

For any point $(\bz, \blambda)$ such that the uniform SOSC \eqref{equ:SOSC}, controllability condition \eqref{equ:control}, and upper boundedness condition \eqref{equ:bound} hold, we have
\begin{align*}
\|K^{-1}_{(i, j), 1}\|\leq &\Upsilon_K\rho^{|i - j|}, \text{\ \ \ \ \ \ for\ } i, j\in[N],\\
\|K^{-1}_{(i, j), 2}\|\leq &\begin{cases}
\Upsilon_K\rho^{i} & \text{\ \ for\ } i\in[N], j=-1,\\
\Upsilon_K\rho^{|i - j|} & \text{\ \ for\ } i\in[N], j\in[N-1],
\end{cases}\\
\|K^{-1}_{(i, j), 3}\|\leq &\begin{cases}
\Upsilon_K\rho^{i+1} & \text{for\ } i\in[-1, N-1],  j=-1,\\
\Upsilon_K\rho^{|i+1- j|} & \text{for\ } i\in[-1, N-1], j\in[N-1],
\end{cases}
\end{align*}
for some constants $\Upsilon_K>0$, $\rho\in(0, 1)$, independent from $N$.
\end{lemma}

\begin{proof}
See  Appendix \ref{pf:lem:2}.
\end{proof}

Lemma \ref{lem:2} shows that the inverse of the KKT matrix has exponential decay in the sense that the operator norm of each $(i, j)$-block decreases exponentially as $|i-j|$ increases. We note that the first $n_x$ columns in $K^{-1}_{\cdot, 2}$, $K^{-1}_{\cdot, 3}$ correspond to the perturbation on the initial condition. To make the inverse structure visible, we plot the bounds of Lemma \ref{lem:2} in Figure~\ref{fig:3}.

\begin{figure}[!tp]
\centering     
\includegraphics[width=4.5cm,height=4.5cm]{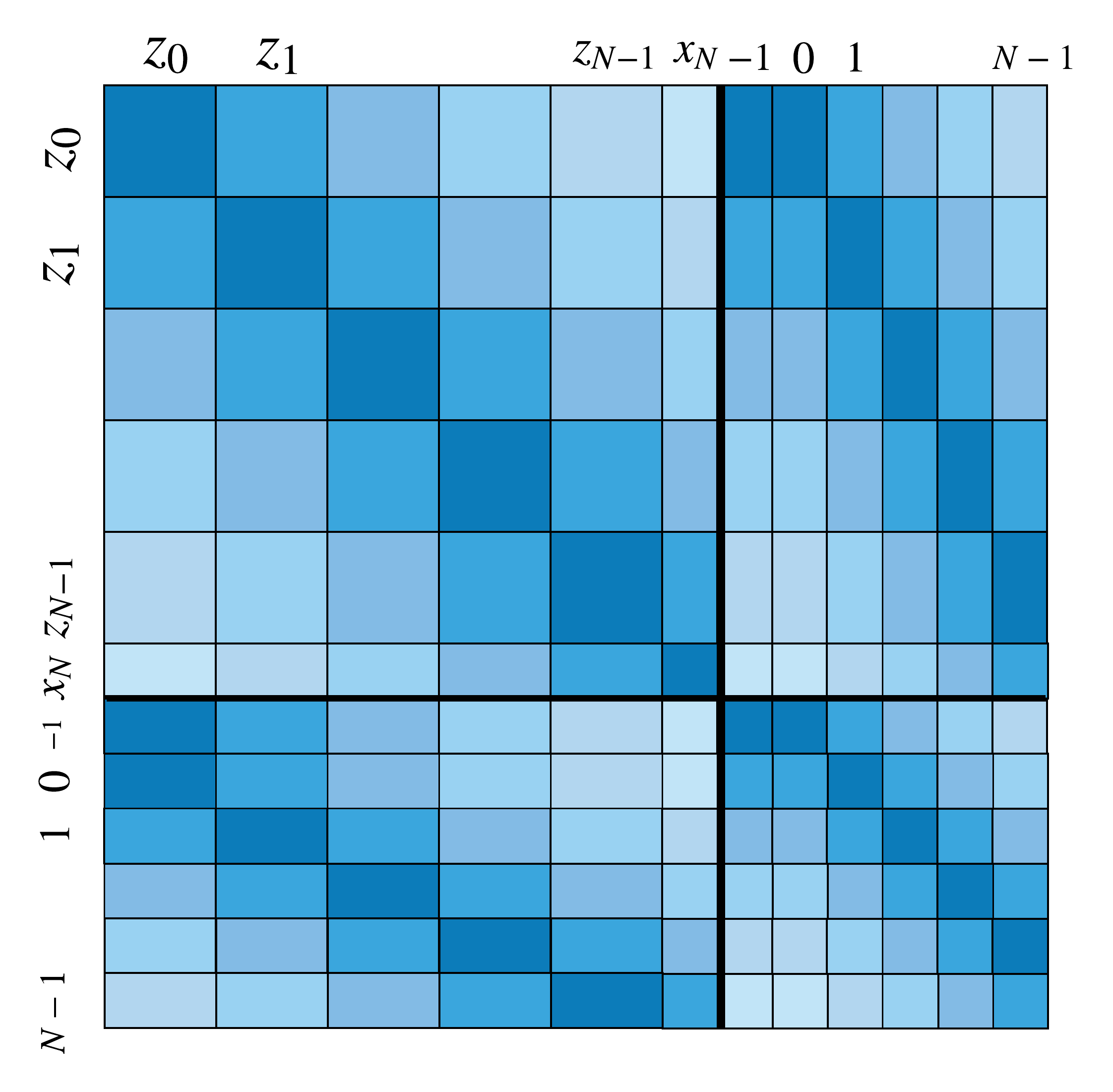}
\caption{\textit{Operator norm trend for KKT inverse.} We plot the decay trend for all blocks in $K^{-1}$. The darker the color, the larger the magnitude.}\label{fig:3}
\end{figure}

In Lemma \ref{lem:2}, $\Upsilon_K$ and $\rho$ are functions of $\Upsilon, \gamma_H, \gamma_C$, $t$, which are independent from the horizon length $N$. Therefore, the same decay trend can be applied on dynamic programs for~any length provided three conditions hold: positive definiteness of the reduced Hessian,~controllability condition, and upper boundedness condition. 

\begin{remark}\label{rem:4}
Lemma \ref{lem:2} is implied by the results in \cite{Na2020Exponential}. In particular, let us consider the Newton equation $K\bb = \ba$. The magnitude of each block of $K^{-1}$ can be reflected by the magnitude of $\bb$ by setting $\ba$ to different unit vectors, since $\bb = K^{-1}\ba$. Note that the Newton equation is the necessary and sufficient condition of a linear-quadratic dynamic program, whose solution (i.e. $\bb$) is similar to the one of Problem (3.1) in \cite{Na2020Exponential}.  We also note that an exponential decay structure has been established for the inverse of banded symmetric matrices in \cite{Frommer2017Bounds}. While qualitatively the results are quite similar, the two methods result in different parameters in the resulting exponential decay bounds. If we permute the KKT matrix to make it have block tridiagonal structure and directly apply \cite{Frommer2017Bounds}, we derive an entrywise exponential decay structure for KKT inverse with a rate roughly being $(\rho')^{1/(n_x+n_u)}$ for $\rho'\in(0,1)$, where $n_x+n_u$ (up to a factor) is the bandwidth of the (permuted) KKT matrix. Such a result is related but does not imply the result in Lemma \ref{lem:2}: (i) we desire a blockwise decay structure of KKT inverse to study the error at each stage; (ii) by investigating the underlying dynamic program of the Newton system and making use of the conditions of the program (such as controllability), our decay rate (see \cite[(5.9)]{Na2020Exponential}) does not depend on dimensions $n_x, n_u$, but depends only on constants in Assumptions \ref{ass:M:SOSC}-\ref{ass:M:Ubound}. Although rescaling the problem and considering blockwise implications of \cite{Frommer2017Bounds} reduces some of these variances, there are other quantities that still behave quite differently, as far as we can infer (for example the derived operator norm bound of each block).

\end{remark}

The next corollary suggests that the same argument holds for~subproblems.

\begin{corollary}\label{cor:2}
Suppose Assumptions \ref{ass:M:SOSC}, \ref{ass:M:control}, and \ref{ass:M:Ubound} hold and $\mu$ satisfies \eqref{equ:set:mu}. Then, for any $i\in[1, T]$, if $(\tbz_i^0, \tblambda_i^0)\in\N_\epsilon(\tbz_i^\star, \tblambda_i^\star)$,
\begin{align*}
\|(K^i)^{-1}_{(j_1, j_2), 1}\|\leq &\Upsilon_K\rho^{|j_1- j_2|}, \text{\ \ for\ } j_1, j_2\in[n_1(i), n_2(i)],\\
\|(K^i)^{-1}_{(j_1, j_2), 2}\|\leq &\begin{cases}
\Upsilon_K\rho^{j_1-n_1(i)} & \text{for\ }\substack{ j_1\in[n_1(i), n_2(i)],\\ j_2=n_1(i)-1,}\\
\Upsilon_K\rho^{|j_1 - j_2|} & \text{for\ } \substack{j_1\in[n_1(i), n_2(i)], \\j_2\in[n_1(i), n_2(i)-1],}
\end{cases}\\
\|(K^i)^{-1}_{(j_1, j_2), 3}\|\leq &\begin{cases}
\Upsilon_K\rho^{j_1+1 - n_1(i)} & \text{for\ } \substack{j_1\in[n_1(i)-1, n_2(i)-1], \\ j_2=n_1(i)-1,}\\
\Upsilon_K\rho^{|j_1+1- j_2|} & \text{for\ } \substack{j_1\in[n_1(i)-1, n_2(i)-1],\\ j_2\in[n_1(i), n_2(i)-1],}
\end{cases}
\end{align*}
for some global constants $\Upsilon_K$ and $\rho\in(0, 1)$.
\end{corollary}

\begin{proof}
See Appendix \ref{pf:cor:2}.
\end{proof}

\begin{remark}
We use the same notations for $\Upsilon_K$, $\rho$ in Lemma~\ref{lem:2} and Corollary \ref{cor:2}. However, since the upper bound of subproblems, $\Upsilon + \mu(\Upsilon, \gamma_C, t)$, is different from the one for the full problem, $\Upsilon$, constants $\Upsilon_K$, $\rho$ are actually different. We do not make this distinction in this paper since we always deal with subproblems.
\end{remark}

\begin{remark}
We mention that the provable decay structure in Corollary \ref{cor:2} is not symmetric, as shown in Figure~\ref{fig:3} (see the bottom right block). But we can rescale $\Upsilon_K$ by $\Upsilon_K/\rho$ to make it symmetric.
\end{remark}

We then establish the error recursion. For $i \in [1, T]$ and $\Id = 0$ or $1$, we define $\Psi_{k, i}^\Id = \|(\bz_{k, i}^\Id - \tz_k; \blambda_{k, i}^\Id - \tlambda_k)\|$ for $k\in[n_1(i), n_2(i)-1]$ and $\Psi_{n_2(i), i}^\Id = \|\bx_{n_2(i), i}^\Id - \tx_{n_2(i)}\|$  (for the definition of superscript ${\Id}$, see the discussion before
\eqref{equ:trans1}). The result is presented in the next theorem.

\begin{theorem}[One-step error recursion]\label{thm:2}

Under the same setup of Corollary \ref{cor:2} and also supposing that Assumption \ref{ass:M:Lip:cond} holds, we have that for $i\in[1, T]$, Algorithm \ref{alg:1} satisfies 
\begin{footnotesize}
\begin{multline}\label{equ:error}
\Psi_{k, i}^1\leq \Upsilon_C\bigg(\overbrace{\sum_{j=n_1(i)}^{n_2(i)}\rho^{|k-j|}(\Psi_{j, i}^0)^2}^{\text{algorithmic error}} \\
+  \underbrace{\rho^{k - n_1(i)}\|\bx_{n_1(i), i}^0 - \tx_{n_1(i)}\| + \rho^{n_2(i) - k}\epsilon}_{\text{perturbation error}}\bigg)
\end{multline}
\end{footnotesize}
for some constant $\Upsilon_C$ and $\rho\in(0, 1)$.

\end{theorem}

\begin{proof}
See  Appendix \ref{pf:thm:2}.
\end{proof}

We ignore the error recursion of $\|\blambda_{n_1(i)-1, i}^\Id - \tlambda_{n_1(i)-1}\|$ although it is similar to \eqref{equ:error}, because $T_{n_1(i) - 1} = i - 1$ (cf. Definition \ref{def:sTk}); that is,  the output error for stage $n_1(i) - 1$ is provided by the $(i-1)$-th subproblem. We see from \eqref{equ:error} that the updated error, $\Psi_{\cdot, i}^1$, consists of two components: algorithmic error and perturbation error. The algorithmic error comes from the Newton iteration, which enjoys the customary quadratic convergence, whereas the perturbation error occurs because the data needed at either end of the horizon is not the solution of \eqref{pro:1}. On the other hand, the exponential decay of the sensitivity of the solution of \eqref{pro:1}, proved in \cite{Na2020Exponential}, ensures that the perturbation effect of the horizon truncation decays exponentially from both sides. This explains our strategy of discarding the receding-horizon endpoint estimates between iterations and indicates that the horizon truncation effects will be negligible in the middle stages.

\begin{remark}\label{rem:3}
Two extreme cases  correspond to the one without perturbation error and the one without algorithmic error, respectively. First, without doing any horizon truncation, one can show that
\begin{align*}
\Psi_k^1\leq \Upsilon_C\sum_{j=0}^N\rho^{|j-k|}(\Psi_j^0)^2, \text{\ \ } \forall k\in[N],
\end{align*}
with $\Psi_k^\Id = \|(\bz_k^\Id - \tz_k; \blambda_k^\Id - \tlambda_k)\|$ for $k\in[N-1]$ and $\Psi_N^\Id = \|\bx_N^\Id - \tx_N\|$. Second, if Problem \eqref{pro:1} is linear-quadratic with $\{g_k\}$ being quadratic and $\{f_k\}$ being linear, then one can show that (since one-step Newton attains the optimal)
\begin{align}\label{equ:LQerror}
\Psi_{k, i}^1\leq \Upsilon_C\big(\rho^{k - n_1(i)}\|\bx_{n_1(i), i}^0 - \tx_{n_1(i)}\| + \rho^{n_2(i) - k}\epsilon\big).
\end{align} 
Using this  inequality, we can further extend the results in \cite{Xu2019Exponentially}. A detailed discussion is presented in the next subsection.
\end{remark}

Returning to \eqref{equ:error}, we observe that we cannot provably  guarantee that the error on the boundary gets improved by doing a one-step Newton iteration. Specifically, if $k$ is close to either $n_1(i)$ or $n_2(i)$, then $\Psi_{k, i}^1 \nleq \Psi_{k, i}^0$. Therefore, using certain lag $L$ is critical for Algorithm \ref{alg:1} coupled with the strategy described in Remark \ref{rem:2}.

The next theorem characterizes the stability of this lag-$L$ MPC strategy.

\begin{theorem}[Stability]\label{thm:3}
Under the setup of Theorem \ref{thm:2}, suppose $L$ and $\epsilon$ satisfy
\begin{align}\label{equ:cond:rhoeps}
\rho^L\leq \epsilon \leq \frac{1-\rho}{\Upsilon_C(5+\rho)} \coloneqq \frac{1}{\kappa}.
\end{align}
Then $\forall i\in[1, T-1]$
\begin{align*}
(\tbz_i^0, \tblambda_i^0)\in\N_\epsilon(\tbz_i^\star, \tblambda_i^\star) \Longrightarrow (\tbz_{i+1}^0, \tblambda_{i+1}^0)\in\N_\epsilon(\tbz_{i+1}^\star, \tblambda_{i+1}^\star).
\end{align*}	
\end{theorem}

\begin{proof}
See Appendix \ref{pf:thm:3}.	
\end{proof}

The quantity $\kappa$ in \eqref{equ:cond:rhoeps} can be interpreted as the condition number of the problem. It is fully characterized by the~quantities of the problem. The condition \eqref{equ:cond:rhoeps} also indicates that, in practice, the optimal lag choice should satisfy $\rho^L\asymp 1/\kappa$, i.e. $L \asymp \log \kappa/\log(1/\rho)$.

\begin{remark}\label{rem:5}
At a first glance, the condition on lag in \eqref{equ:cond:rhoeps} seems counter-intuitive, in the sense that the results in \cite{Zavala2010Real, Dinh2012Adjoint, Zanelli2021Lyapunov} all require the sampling time to be sufficiently small and standard MPC regimes \cite{Diehl2005Nominal, Wang2010Fast} all use lag $1$ instead. However, those results are for continuous data, and do not apply to our discrete-time time-varying problems, where we only have uniform bounded data. As emphasized before, the continuity breaks in our case since even two successive subproblems $\P_i(\tbd_i^\star)$ and $\P_{i+1}(\tbd_{i+1}^\star)$ may have large solutions distance. The condition on lag $L$ is the price we pay for requiring only bounded data, in contrast to continuity combined with a sufficiently small sampling time. 

\end{remark}

Using the stability property in Theorem \ref{thm:3} and combining with uniform positive definiteness of the reduced Hessian guarantee in Theorem \ref{thm:1}, we can  show that  Algorithm \ref{alg:1} can be carried out successfully, in the sense that the series of subproblems are well defined and satisfy the same error recursion rule.

\begin{theorem}\label{thm:4}
Consider using Algorithm \ref{alg:1} to solve \eqref{pro:1}. Suppose Assumptions \ref{ass:M:SOSC}, \ref{ass:M:control}, \ref{ass:M:Ubound}, \ref{ass:M:Lip:cond} and conditions on $\mu$, $L$, $\epsilon$ in \eqref{equ:set:mu} and \eqref{equ:cond:rhoeps} hold. Then, for any $i\in[1, T]$, the KKT matrix in \eqref{equ:Newton} is nonsingular, and the error recursion in \eqref{equ:error} holds as well.
\end{theorem}

\begin{proof}
See Appendix \ref{pf:thm:4}.
\end{proof}

Up to this point we have established the one-step error recursion and the stability of Algorithm \ref{alg:1}. Based on these results, we can now state our local convergence rate result.

\subsection{Local convergence rate}\label{sec:4.3}

In this subsection we study the local convergence rate for each stage. We note that some adjacent stages should have similar rates. For example, stages in $[0, L)$ are iterated only once in the first subproblem, hence we can group them together to analyze. Similarly, stages in $[L, 2L)$ are iterated twice in the first two subproblems, and they can be grouped together. We formalize this inner grouping structure in the following definition.

\begin{definition}
For any $s\in[S-1]$, we group stages by $O_s = \{k\in[N]: s_k = s\}$, where $s_k = s(k, T_k)$ with $s(\cdot, \cdot)$, $T_k$ defined in Definition \ref{def:sTk}. Further, for any $i\in[1, T]$, we group stages within subproblem $i$ by $O^i_s = \{k\in[n_1(i), n_2(i)]: s(k, i) = s\}$. We  let $\Omega_s = \max_{i\in[1, T]}\max_{k\in O^i_s}\Psi_{k, i}^0$.
\end{definition}

We show how $\{\Omega_s\}_{s\in[S-1]}$ characterize the groups in Figure \ref{fig:4}. By our output setup in \eqref{equ:output}, the error at stage~$k$ is given by $\Omega_{s_k}$, that is for any stages $k\in O_s$ their errors are characterized by $\Omega_s$. Thus, in what follows, we focus only on $\{\Omega_s\}_{s\in[S-1]}$; in particular, we are more interested in $\Omega_{S-1}$.

\begin{figure}[!tp]
	\centering     
	\includegraphics[width=8.7cm,height=3cm]{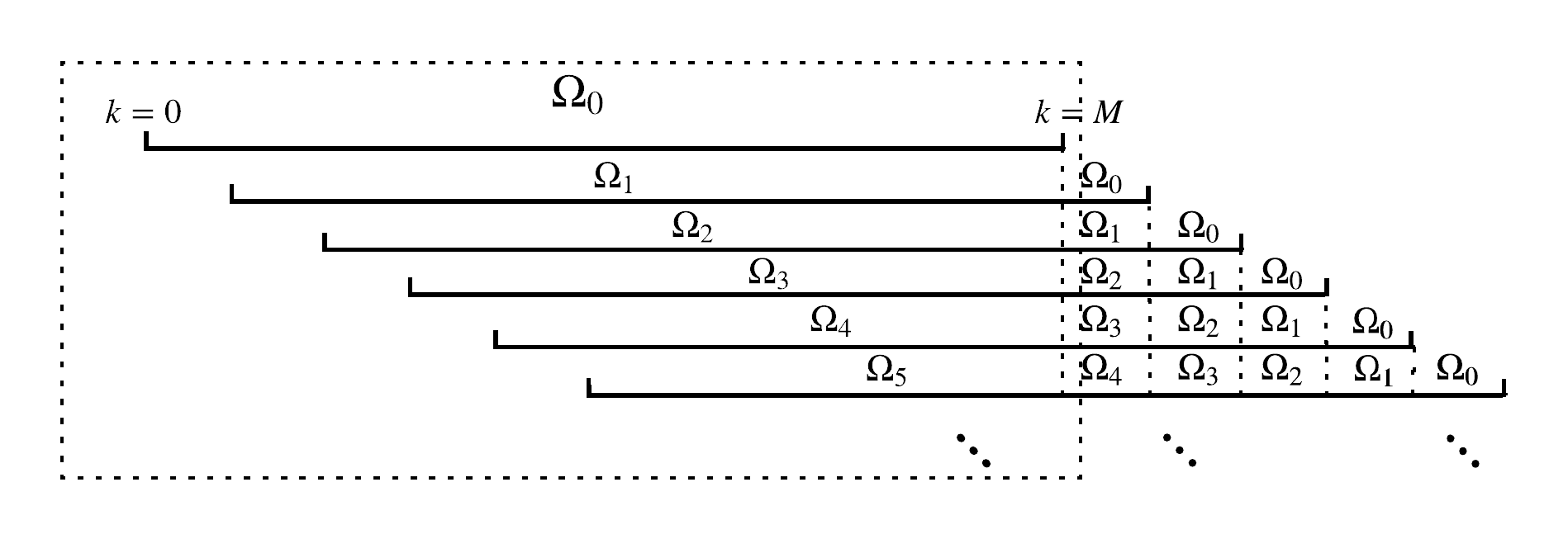}
	\caption{\textit{Stages in  groups.}}\label{fig:4}
\end{figure}

\begin{theorem}[Local convergence rate]\label{thm:5}
Under the setup in Theorem \ref{thm:4}, we have
\begin{align*}
\Omega_s\leq \begin{cases}
3\epsilon & \text{\ \ if\ } s = 0, 1,\\
3\kappa\epsilon^2 & \text{\ \ if\ } s = 2, \\
3\kappa^{s-2}\epsilon^{s-1} &\text{\ \ if\ } s \geq 3.
\end{cases}
\end{align*}

\end{theorem}

\noindent{\bf \textit{Proof idea}}. We only need analyze the first $S$ periods, more specifically, the dashed part in Figure \ref{fig:4}. If we can find a sequence $\{\delta_s\}_{s\in[S-1]}$ that satisfies two critical properties---(i) it bounds $\Omega_s$ : $\Omega_s\leq \delta_s$, $\forall s\in[S-1]$, and (ii) it is nonincreasing: $\delta_{s-1}\geq \delta_s$, $\forall s\in[1, S-1]$---then this sequence controls the convergence speed of $\{\Omega_s\}$. Moreover, because of the monotonicity of $\{\delta_s\}$, the dashed part can represent the entire horizon. Note that the $(S+1)$-th subproblem can be seen as the $S$-th subproblem starting from the second subproblem.

\begin{proof}
See Appendix \ref{pf:thm:5}.
\end{proof}

We can interpret Theorem \ref{thm:5} from two perspectives. First, when $L\asymp \log \kappa/\log(1/\rho)$, our algorithm achieves the exponential convergence rate with respect to the  number of ``scans"; in other words, the error of each stage $k$ decays exponentially in $s_k$ (see Definition \ref{def:sTk} for $s_k$). This matches the online result in \cite{Diehl2005Real} where the authors solves a sequentially embedded subproblems. Second, we plug in $\rho^L\asymp \epsilon$ and have
\begin{align}\label{equ:super}
\Omega_{S-1} \leq &3\kappa^{S-3}\rho^{L(S-2)} = 3\kappa^{\frac{M}{L}-3}\rho^{M-2L} = \frac{3}{\kappa}(\kappa^{1/L}\rho)^{M-2L} \nonumber\\
\coloneqq &\frac{3}{\kappa}(\rho')^{M-2L}.
\end{align}
Here, $\kappa^{1/L}\rho \coloneqq \rho'<1$ whenever $\epsilon\kappa<1$. Thus, when the initial \textit{guess}, $(\bz^0, \blambda^0)$, is close enough to the true value (radius bounded by reciprocal of condition number), our MPC strategy converges exponentially in terms of the receding-horizon length for all middle stages that have been scanned the most and are the vast majority among all stages in $[0, N]$. This result matches the results in \cite{Xu2019Exponentially} for the linear-quadratic problem.

\begin{remark}[Comparison with the linear-quadratic problem; balance between algorithmic error and perturbation error]
Suppose $\{g_k\}$ are quadratic and $\{f_k\}$ are linear. Then, as mentioned in Remark \ref{rem:3}, the error recursion reduces to \eqref{equ:LQerror}. Based on \eqref{equ:LQerror}, we extend the results in \cite{Xu2019Exponentially} by incorporating the error of the Lagrange multipliers. For general nonlinear functions $\{g_k, f_k\}$, we have a trade-off between the algorithmic error and perturbation error. At the beginning, the algorithmic error is the dominant term since the perturbation occurs only on the tail. Thus, errors for stages in $[0, L)$ are on the order of $\epsilon^2\asymp \rho^{2L}$, instead of $\rho^{M - L}$ in the linear-quadratic case. Since the algorithmic error decreases quadratically, however, after a few steps the perturbation error will dominate. Overall, we achieve an exponential rate with respect to the receding horizon length. 
\end{remark}

\begin{remark}[Result comparison with RTI schemes]\label{rem:result}
Our convergence result in Theorem \ref{thm:5} is stronger than existing RTI results in \cite{Diehl2005Nominal, Zavala2010Real, Dinh2012Adjoint, Wynn2014Convergence}, and the technical tools we use are different. When the first Newton stepsize is small enough, \cite{Diehl2005Nominal} showed the Newton steps contract to a neighborhood of zero (see \cite[Corollary 5.2]{Diehl2005Nominal}), while \cite{Wynn2014Convergence} showed the Newton iterates contract to a neighborhood of the subproblem solution (see Theorem 3). Both \cite{Diehl2005Nominal,Wynn2014Convergence} considered quadratic objectives (or objectives that are lower bounded by quadratic functions). Similarly, \cite{Zavala2010Real, Dinh2012Adjoint} showed that when the sampling time is small enough, the iterates track the solution of the exact MPC policy. The recent work \cite{Zanelli2021Lyapunov} allowed inequality constraints, and proposed a framework for analyzing RTI by constructing a Lyapunov function for combined system-optimizer dynamics. However, all aforementioned works do not establish the explicit convergence rate of the error of~RTI. Here, we consider the error with respect to the full-horizon solution on a per-stage basis: we show that the error decreases with increasing the shift order to a minimum value, which vanishes to zero exponentially in the receding horizon length. In particular, we show that certain stages have much better accuracy than others for online MPC,  which we call a superconvergence phenomenon. We note that the solution of subproblems in our time-varying problems cannot be reasonably regarded as the baseline since it varies with the subproblems' boundary parameters (if boundary parameters are set by the full-horizon solution then the truncated full-horizon solution is also the solution of subproblems). Moreover, the explicit exponential convergence rate in terms of the ``number of scans" is established in Theorem \ref{thm:5}, which is more refined than the results in the preceding literature, whereby all stages get a similar error bound. Furthermore, in  \eqref{equ:super} the dependence of the rate on the receding horizon length is first established for nonlinear problems in this paper and bridges the gap with convex linear-quadratic problems \cite{Xu2019Exponentially}.

In addition, the technical tools of the preceding literature are based on the convergence of the adopted algorithm (either Newton or predictor-corrector) and the continuity of the problem with respect to certain parameters (usually the initial state). As discussed in Remark \ref{rem:5}, we work exclusively with the uniform bounded data, a weaker assumption. Our technical tool is  the sensitivity analysis of \eqref{pro:1}, which induces the exponential decay structure of the KKT inverse (Lemma \ref{lem:2}). As a comparison, results in \cite{Zavala2010Real, Dinh2012Adjoint} apply  general nonlinear programming techniques which do not address per-stage behavior.

\end{remark}

\begin{remark}[Some potential extensions on analysis]\label{rem:6}
One of limitations of our analysis is that we exclude the inequality constraints in \eqref{pro:1}. If the inequalities are on the control only, the analysis is more involved but straightforward provided we assume either strict complementarity or a strengthened version of SOSC, in addition to the assumptions here holding on the active set of the long-horizon problem (see \cite{Xu2018Exponentially} for a flavor of how we think the analysis would look like). For state constraints, the problem is considerably more complicated and the controllability condition needs to be strengthened, even in the convex case \cite{Xu2019Exponentially}. The main idea, however, is to reduce the problem by an active set argument to the case described here, which we believe would work in many cases of interest.
		
A methodological issue is the fact that we exhibit loss of accuracy at the beginning and end of the horizon. For many applications, the end of the horizon is special and we would like to improve convergence in that region even if we need more computations. A possible way to fix this is to combine this work with ideas from \cite{Diehl2005Real}.
\end{remark}

%% file: sec5.tex
\section{Simulation}\label{sec:5}

In this section, we apply Algorithm \ref{alg:1} on long-horizon dynamic programs to validate the theoretical results we established in {\black Section} \ref{sec:4}. The efficiency and superiority of fast online MPC strategies have been probed in \cite{Wang2010Fast}, where the authors also discussed how to numerically solve Newton's step. We will not go into such details here.

Specifically, we let $n_x = n_u = n_d = 1$ and consider the following problem:
\begin{footnotesize}
\begin{align}\label{pro:3}
\min_{\bx, \bu} &\sum_{k = 0}^{N-1}\big(2\cos(x_k -d_k)^2+ C_1(x_k - d_k)^2 - C_2(u_k - d_k)^2\big) + C_1x_N^2, \nonumber\\
\text{s.t.}\  & x_{k+1} = x_k + u_k + d_k, \text{\ \ \ } \forall k\in[N-1],\\
& x_0 = 0, \nonumber
\end{align}
\end{footnotesize}
\hskip -4pt
that Assumptions \ref{ass:M:SOSC}, \ref{ass:M:control}, \ref{ass:M:Ubound}, and \ref{ass:M:Lip:cond} all hold: the controllability condition is satisfied with $t = 1$ and $\gamma_C = 1$; the uniform boundedness condition is satisfied with $\Upsilon = 1\vee 4 + 2|C_1|\vee 2|C_2|$; the Lipschitz condition is satisfied with $\Upsilon_L = 4$; and the uniform SOSC is also satisfied provided $C_1 - 2>4|C_2|$. In fact, the Hessian of Problem \eqref{pro:3} has two components. One is $\diag(4-4\cos\big(2(x_k - d_k)\big), 0)$, which is positive semidefinite in the whole space; the other is $\diag(2C_1-4, -2C_2)$. As discussed in Remark 3.2 in \cite{Na2020Exponential}, when $2C_1-4>8|C_2| \Rightarrow C_1-2>4|C_2|$, the second component is positive definite in the null space with $\gamma_H = (C_1 - 2 - 4|C_2|)/4$.

\noindent{\bf Simulation setting:} We consider three  setups summarized in Table \ref{tab:1}. As shown in the table, the entire horizon length varies from $5000$ to $40000$. For each $N$, we implement Algorithm \ref{alg:1} four times with different receding-horizon lengths. The full-horizon solution that is used in the computation of the error was obtained by solving \eqref{pro:3} with the JuMP/Julia package and the IPOPT solver. Throughout the experiment the initial point is set as $(\bz^0, \blambda^0) = (0, 0)$, and the parameter $\mu$ is fixed at $10$.

\begin{table}[!htp]
	\centering
	\caption{Simulation setup.}\label{tab:1}
	\begin{tabular}{c|c|c|c|c } 
		\hline
		Cases  & $d_k$ & $N$ &  $M$ & $(L, C_1, C_2, \mu)$  \\	
		\hline
		Case 1 & $1$ & $5000$ &  $[10, 20, 30, 40]$ & $(5, 8, 1, 10)$\\
		\hline
		Case 2 & $5\sin(k)$ & $10000$ & $[30, 40, 50, 60]$ & $(10, 12, 2, 10)$ \\
		\hline
		Case 3 & $10\sin(k)^2$ & $40000$ & $[50, 60, 70, 80]$ & $(10, 40, 5, 10)$ \\
		\hline
	\end{tabular}
	
\end{table}

\noindent{\bf Result summary:} For each case, we present four figures: 
\begin{enumerate}[label=(\alph*),topsep=1pt]
\setlength\itemsep{-0.1em}
\item Solution trajectory: $\hat{x}_k$ v.s. $k$, $\hat{u}_k$ vs. $k$, $\hat{\lambda}_k$ vs. $k$. We expect to see the trajectories obtained by Algorithm~\ref{alg:1} approximate to the true trajectory in the middle stages, although larger discrepancy might be observed at the two ends. To make the figure readable, we  plot only the first $100$ stages.
\item Group error trend: $\log(\max_{k\in[n_1(i), n_1(i+1))}\Psi_{k, T_k}^0)$ v.s. $i$. For any $i \in[1, N/L]$, the stages in $[n_1(i), n_1(i+1))$ are grouped together since they are scanned the same number of times (see the definition of $s_k$ in Definition \ref{def:sTk}). Among $N/L$ groups, the iteration number for the first $S (= M/L)$ groups is increasing by one to $S-1$, while for the last $S$ groups it is decreasing by one to zero. The middle groups are all iterated on $S-1$ times. From this figure, we expect to see that the numerical error converges to constant value for all choices of $M$ in the middle and decrease or increase linearly at the tail ends, as suggested by Theorem \ref{thm:5}.
\item We zoom in Figure 4(b) at the left and right horizon ends to check the linearity (in log plot) with respect to group $i$, which certifies the exponential convergence with respect to the iteration count.
\item Error against receding horizon length: $\log \Omega_{S-1}$ v.s. $M$. We expect to see linear decay in the log plot, as predicted by Theorem~\ref{thm:5}.
	
\end{enumerate}

Simulation results for the three cases are shown in Figure~\ref{fig:case1},~\ref{fig:case2}, and \ref{fig:case3}, respectively. They are all consistent with expectations.

\begin{figure}[!htp]
	\centering     
	\subfigure[]{\label{Case11}\includegraphics[width=43.7mm]{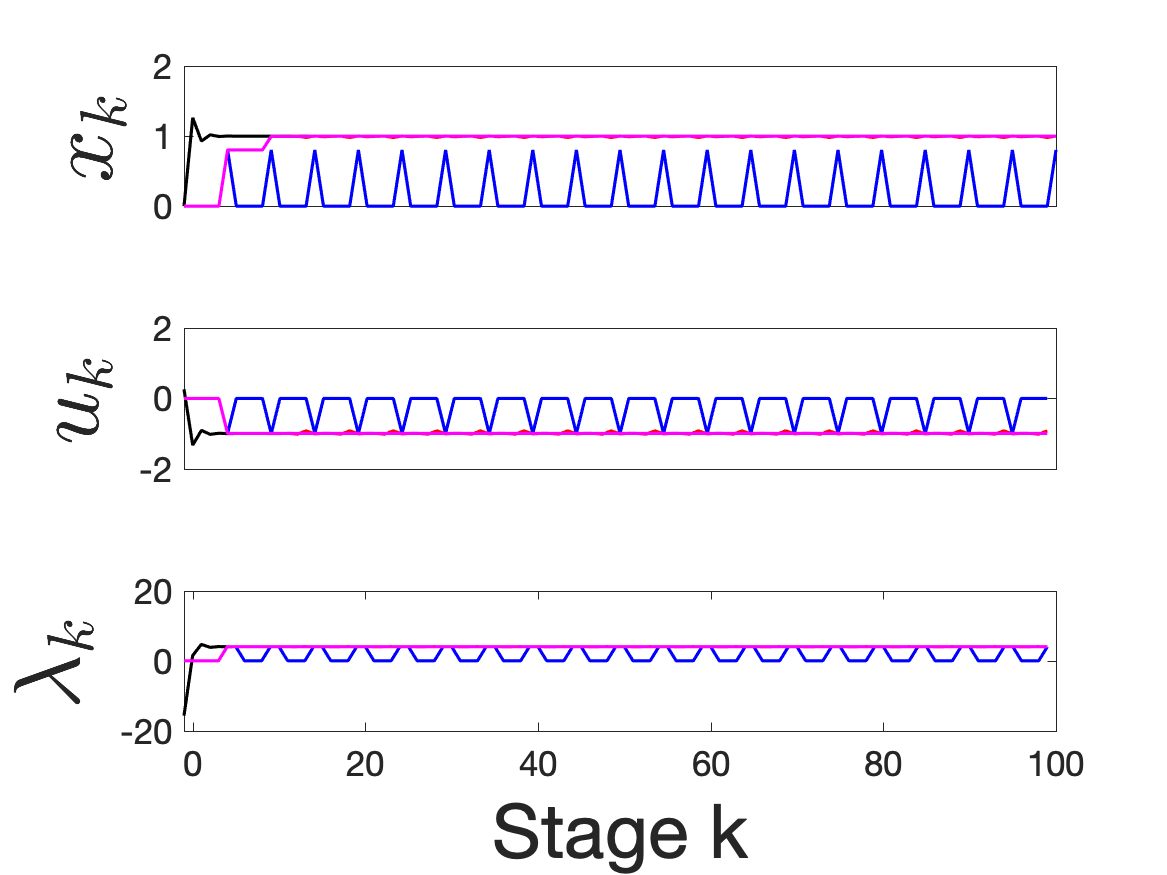}}
	\subfigure[]{\label{Case12}\includegraphics[width=43.7mm]{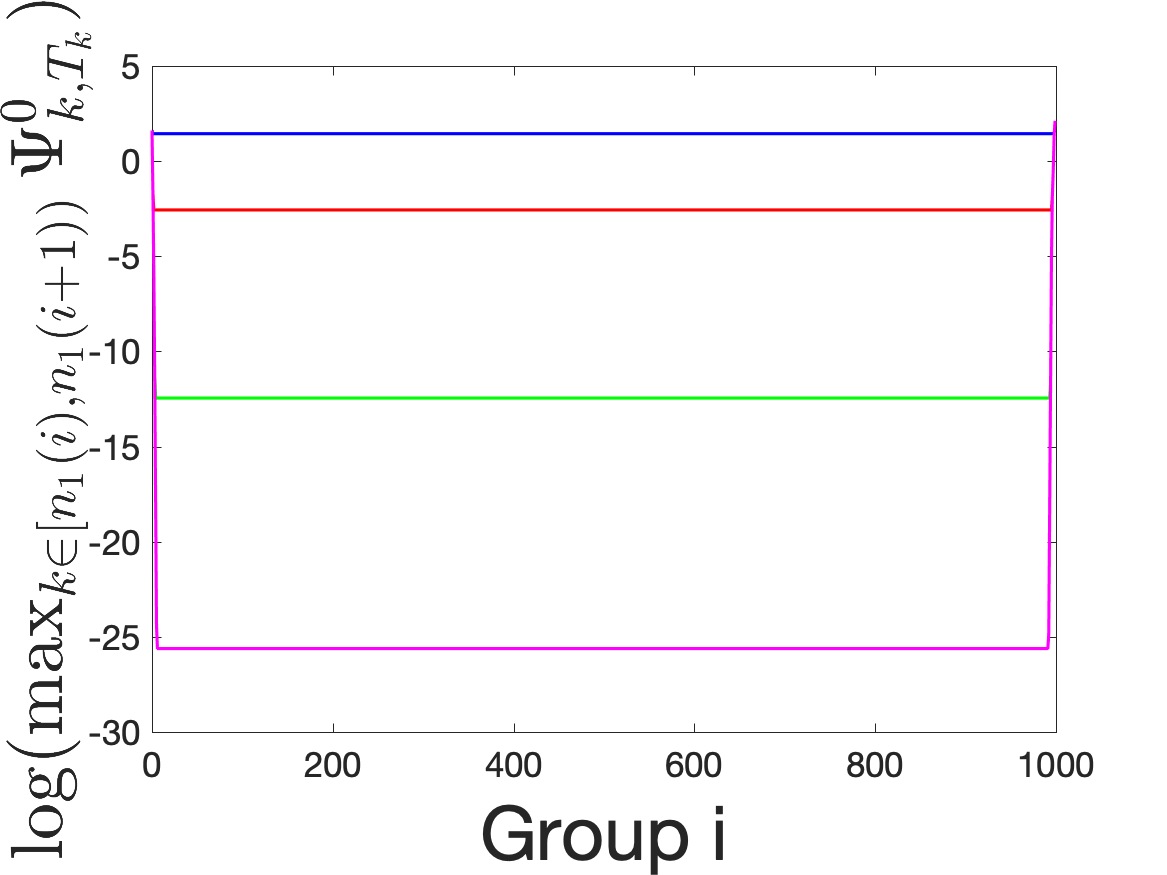}}
	\subfigure[]{\label{Case13}\includegraphics[width=43.7mm]{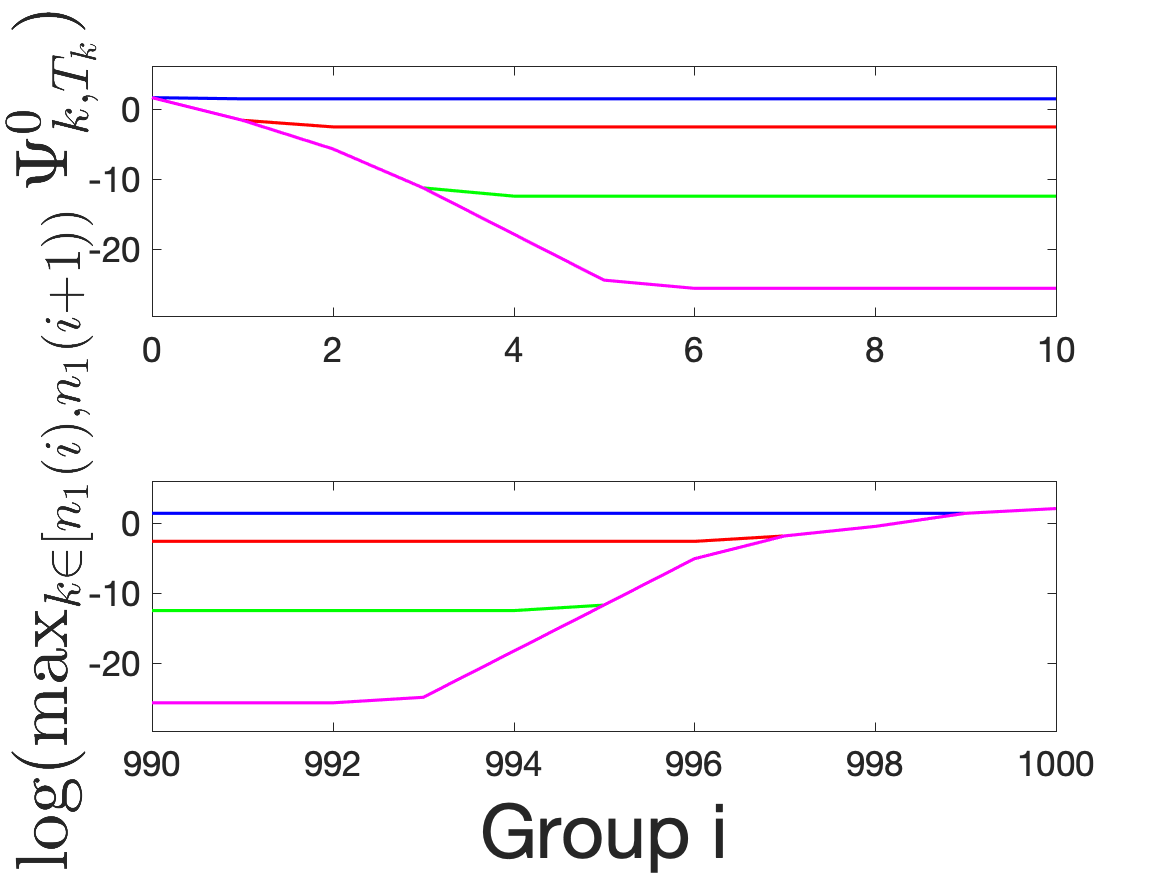}}
	\subfigure[]{\label{Case14}\includegraphics[width=43.7mm]{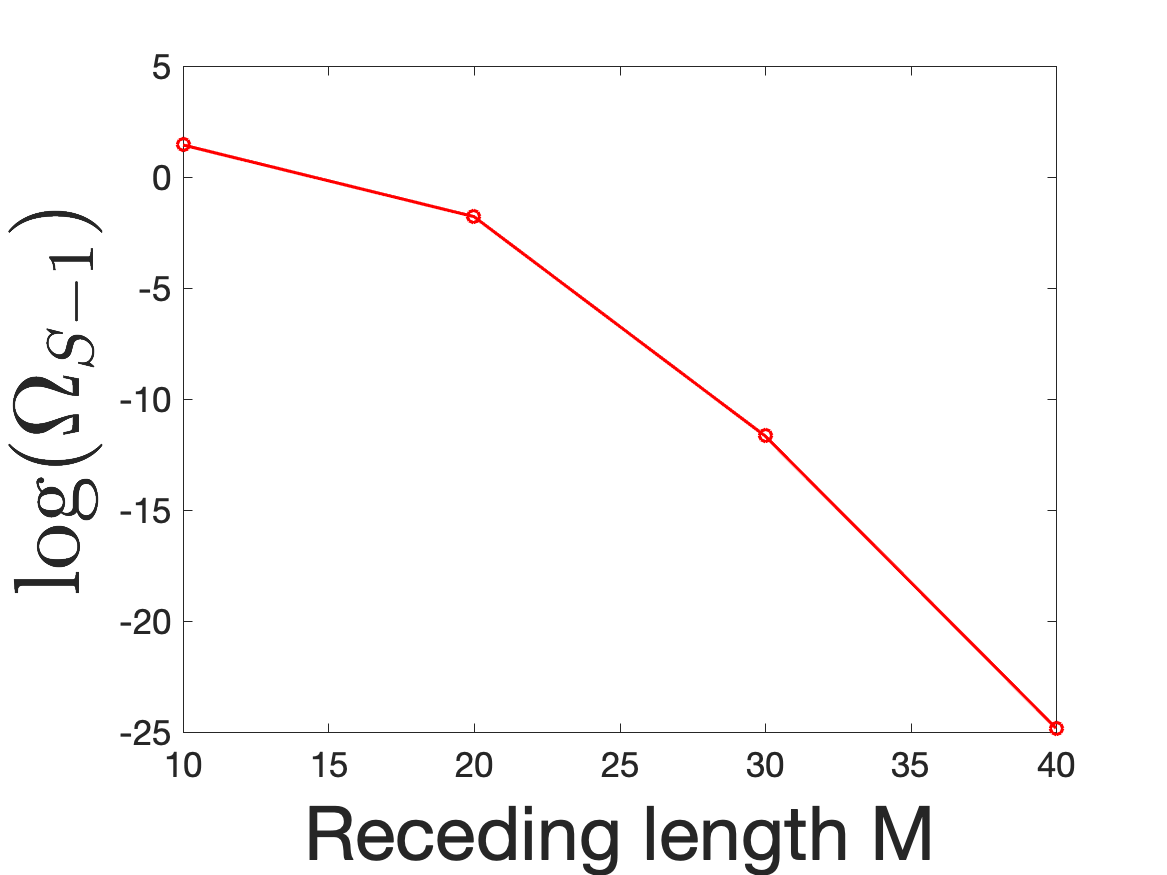}}
	\includegraphics[width=5cm, height=0.25cm]{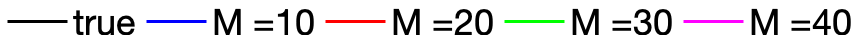}
	\caption{\textit{Simulation result for Case 1. We see from (a) that the middle stages are more and more precise as $M$ increases, although all trajectories lose precision at the horizon tail. From (b) we see that the errors for the middle groups are stable for all trajectories, since they have the same iteration times. From (c) we recover the linear decay with respect to the iteration count. From (d) we recover the linear decay with respect to the receding horizon length. All figures are implied by Theorem \ref{thm:5}.}}\label{fig:case1}
\end{figure}

\begin{figure}[!htp]
	\centering     
	\subfigure[]{\label{Case21}\includegraphics[width=43.7mm]{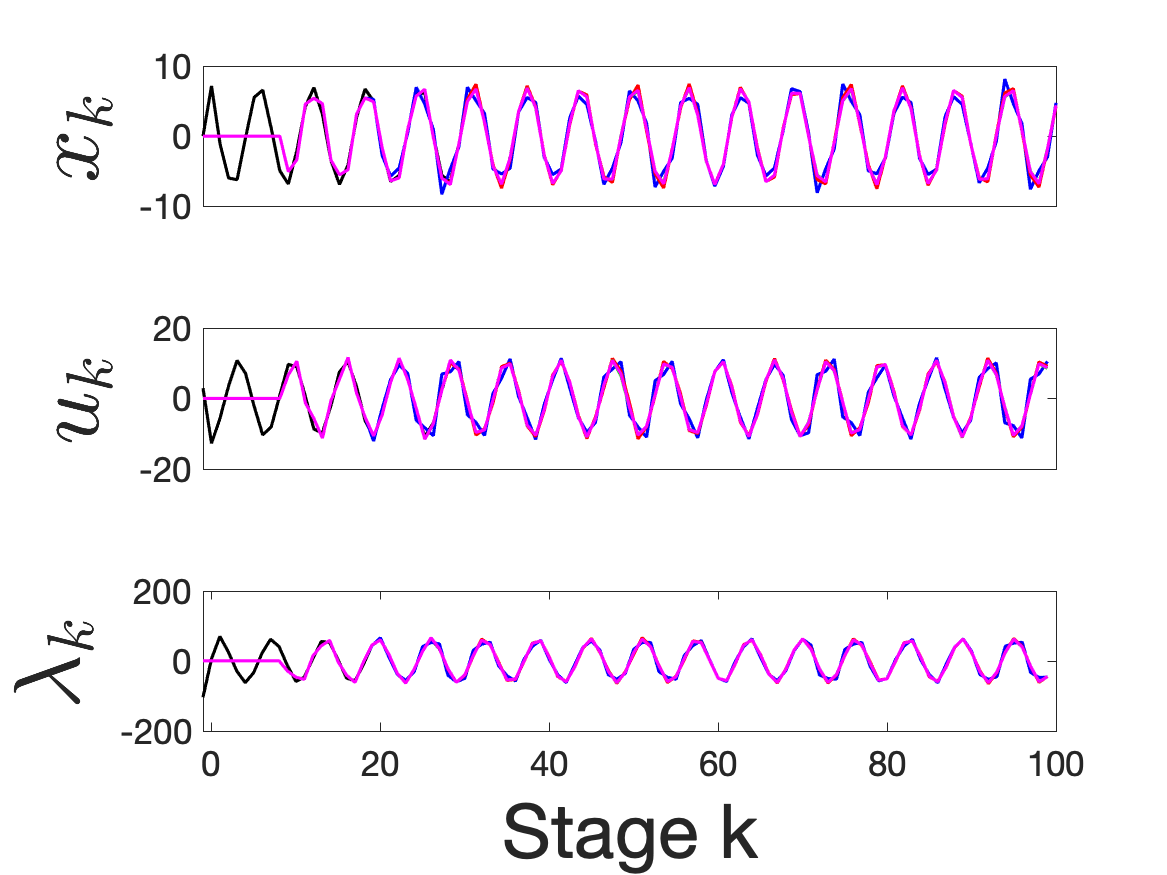}}
	\subfigure[]{\label{Case22}\includegraphics[width=43.7mm]{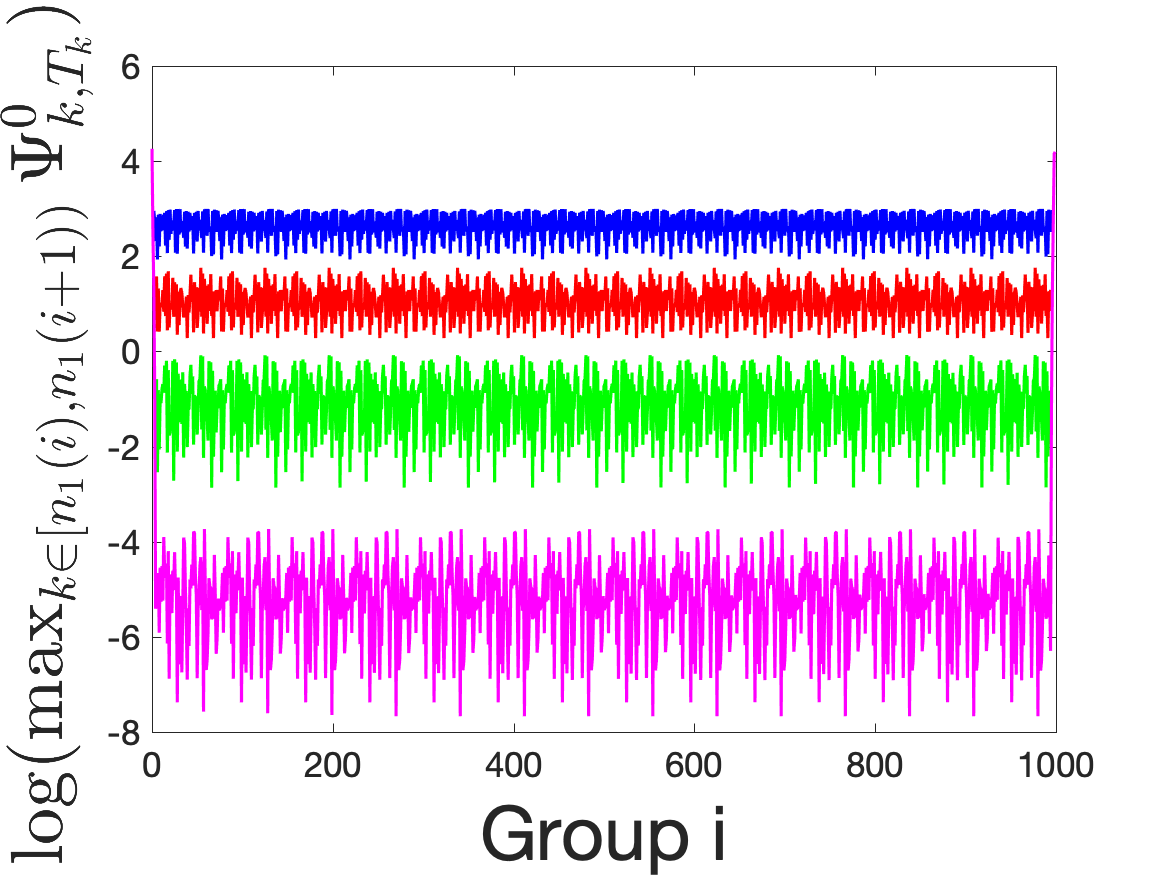}}
	\subfigure[]{\label{Case23}\includegraphics[width=43.7mm]{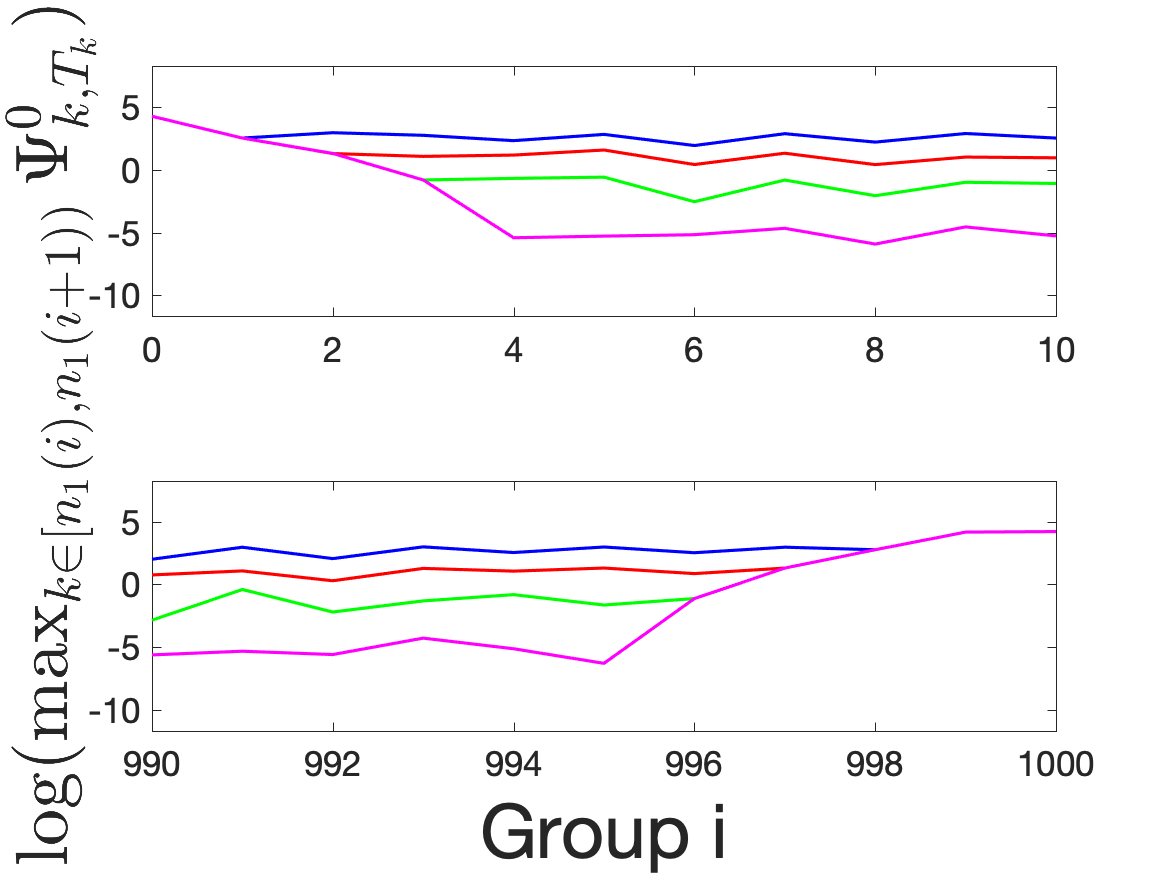}}
	\subfigure[]{\label{Case24}\includegraphics[width=43.7mm]{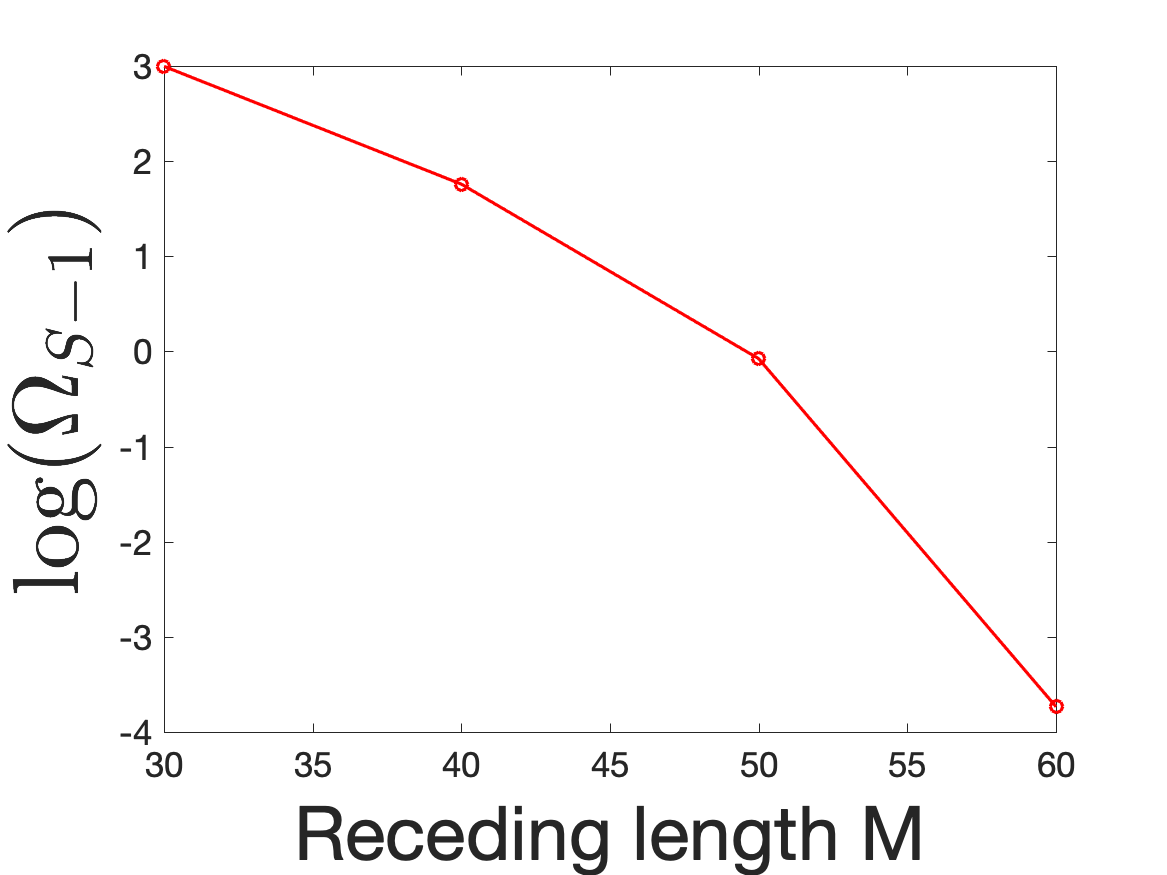}}
	\includegraphics[width=5cm, height=0.25cm]{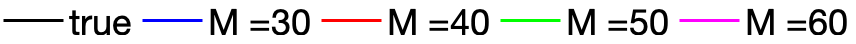}
	\caption{\textit{Simulation result for Case 2. Same interpretation as for Figure \ref{fig:case1}.}}\label{fig:case2}
\end{figure}

\begin{figure}[!htp]
	\centering     
	\subfigure[]{\label{Case31}\includegraphics[width=43.7mm]{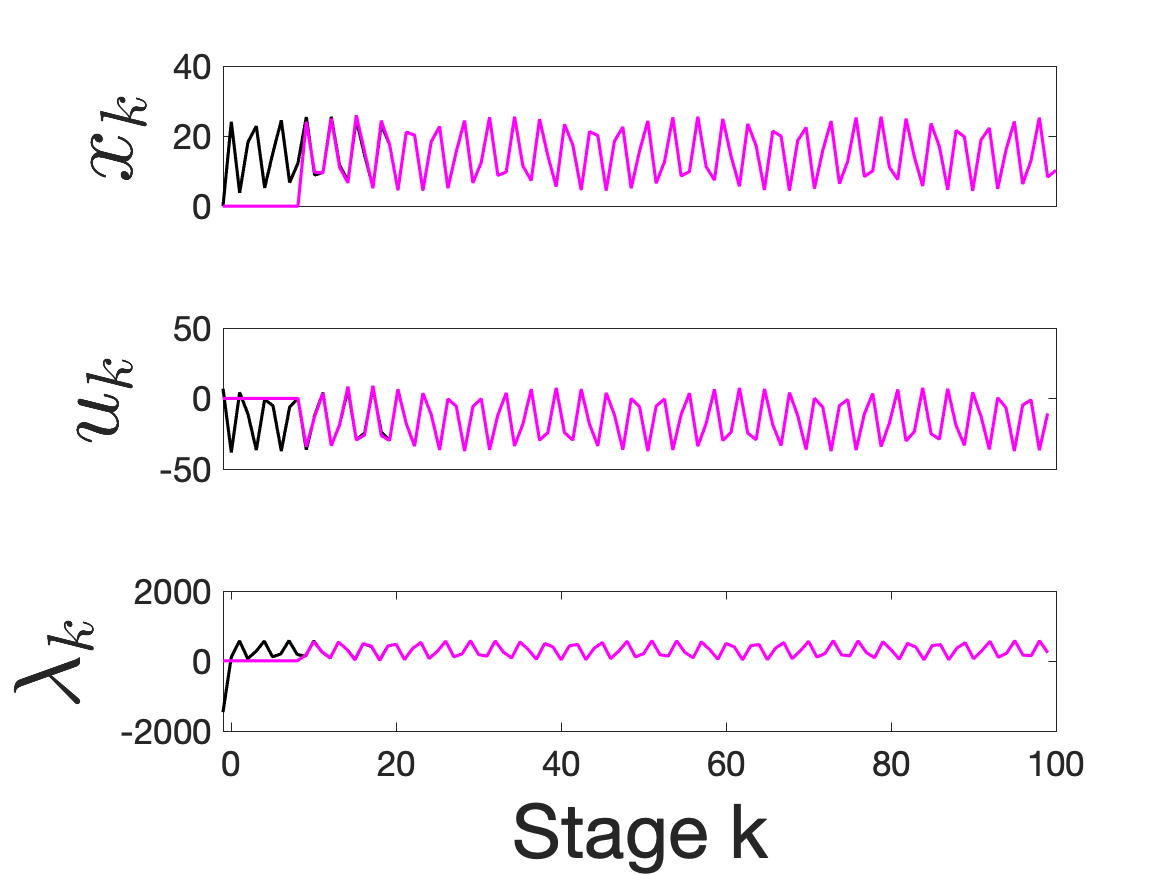}}
	\subfigure[]{\label{Case32}\includegraphics[width=43.7mm]{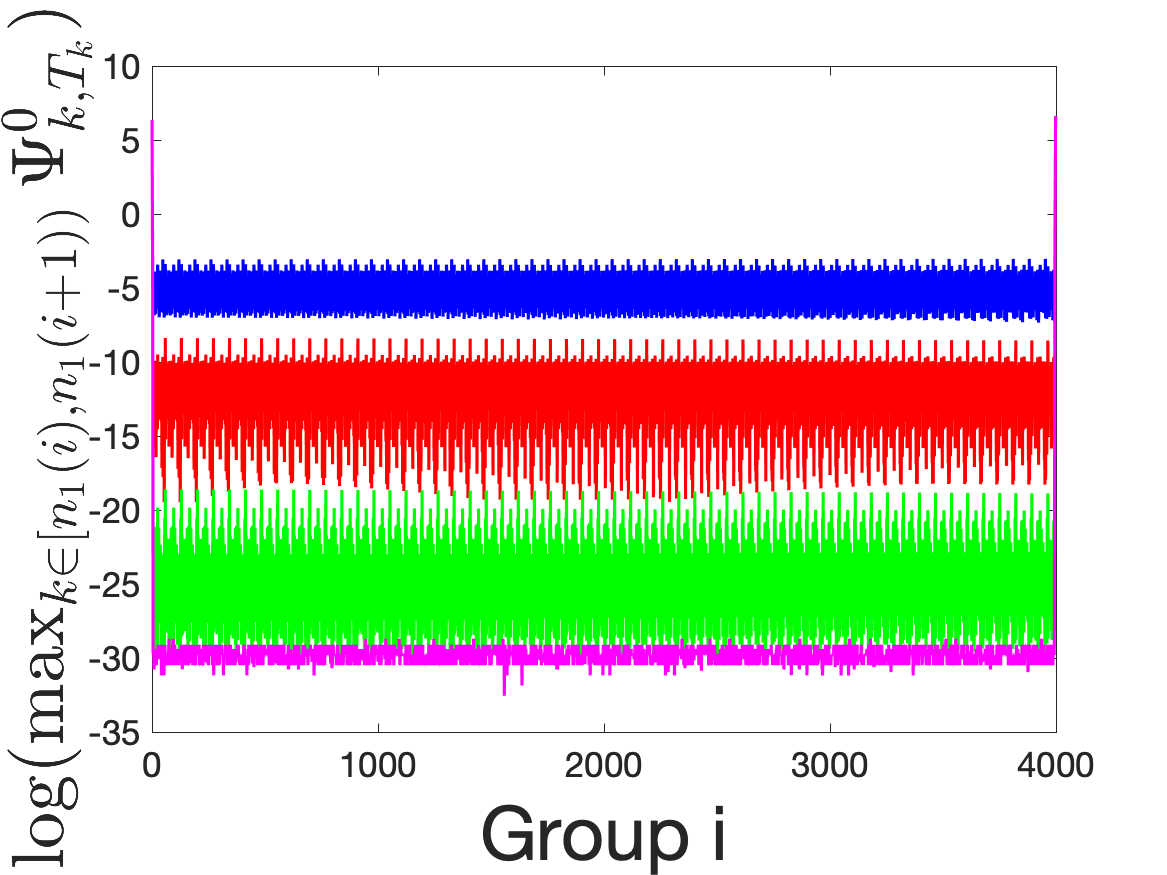}}
	\subfigure[]{\label{Case33}\includegraphics[width=43.7mm]{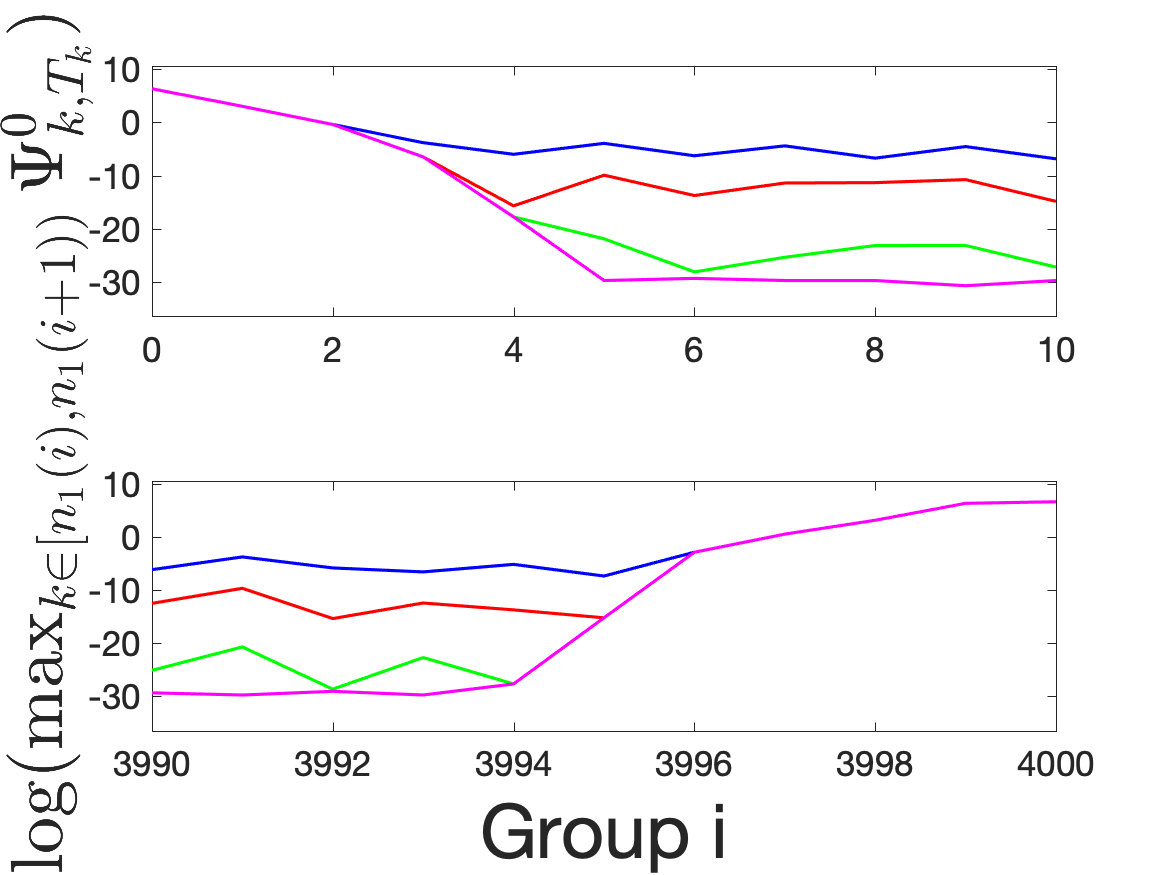}}
	\subfigure[]{\label{Case34}\includegraphics[width=43.7mm]{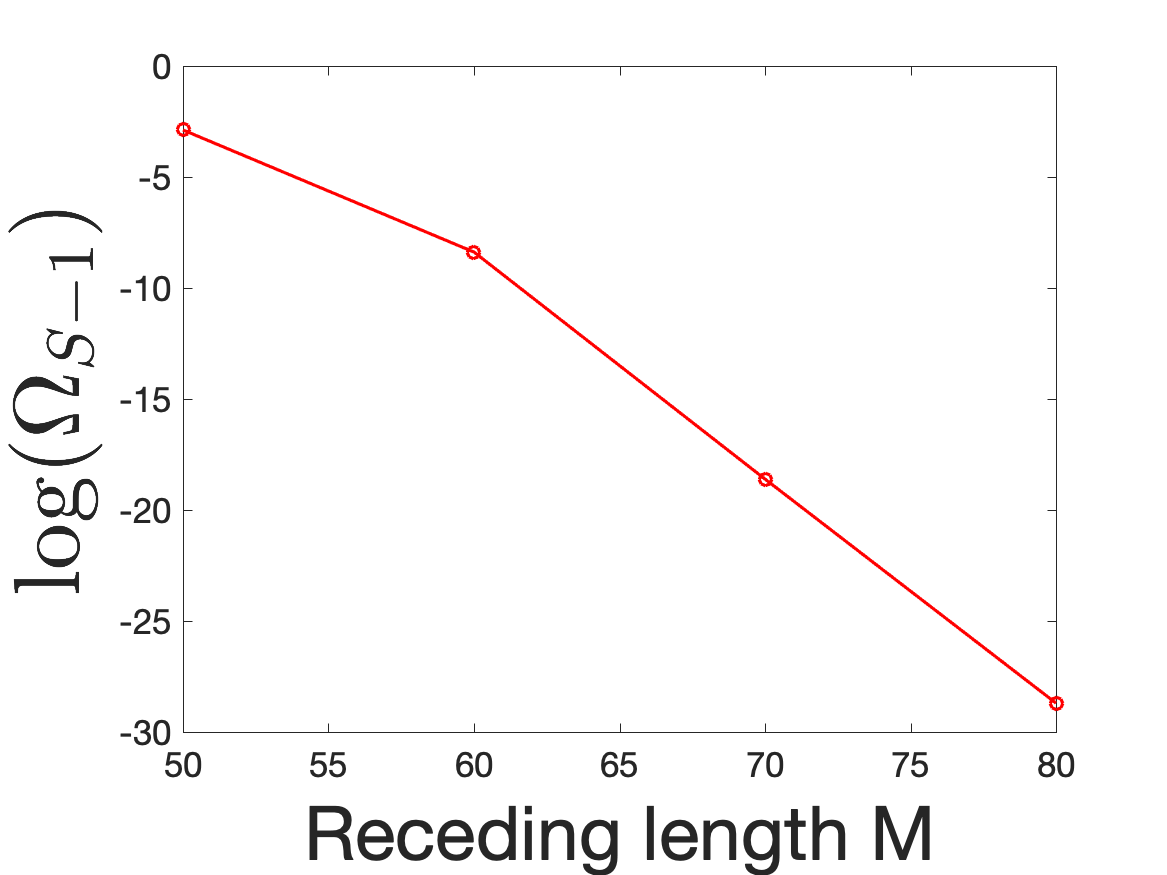}}
	\includegraphics[width=5cm, height=0.25cm]{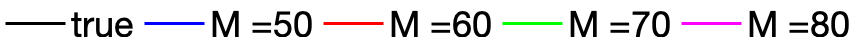}
	\caption{\textit{Simulation result for Case 3. Same interpretation as for Figure \ref{fig:case1}.}}\label{fig:case3}
\end{figure}

\begin{table*}[!thp]\label{tab:2}
\caption{The overhead and precision comparison with related methods (with standard deviation for runtimes in ()).}
\centering
\begin{tabular}{l|c|c|c|c|c|c}
\hline
\multirow{2}{*}{Method} & \multicolumn{2}{|c|}{Case 1} & \multicolumn{2}{|c|}{Case 2} & \multicolumn{2}{|c}{Case 3}\\
\cline{2-7}
& Time(s) & StageErr ($10^{-14}$) & Time(s) & StageErr ($10^{-6}$) & Time(s) & {StageErr ($10^{-11}$)} \\
\hline
\cite{Diehl2005Nominal} & 13.25 (0.33) & 0.02 & 27.28 (1.34) & 0.01 & 118.23 (6.39) & 0.03\\
\cite{Wang2010Fast} &  38.94 (1.14) & 0.01 & 80.31 (5.69) & 0.00 & 349.87 (16.08) & 0.01\\
Alg \ref{alg:1}&  1.32 (0.02) & 31.41& 3.03 (0.80) & 42.87 & 11.59 (0.58) & 22.33\\
Alg \ref{alg:1} ($L=1$) & 13.01 (0.18) & 0.02  & 28.20 (3.55) & 0.01 & 118.28 (5.26) & 0.02\\
Alg \ref{alg:1} (NoTech) & 1.32 (0.03) & 0.04 & 2.79 (0.27) & 0.01 & 12.18 (1.22) & 0.03\\
		\hline
	\end{tabular}
\end{table*}

Moreover, we compare the compute time and accuracy with some representative RTI methods in \cite{Diehl2005Nominal, Wang2010Fast}. We also modify different components of Algorithm \ref{alg:1} to test its effectiveness. In particular, we still consider solving Problem \eqref{pro:3} under three setups summarized in Table \ref{tab:1}. We also tried to implement \cite{Diehl2005Real}, which performs a single Newton step for a sequence of embedded subproblems whose boundary stages are given by $\{[i, N]\}_{i\in[0, N-1]}$. For the horizon considered, such a method takes more than one hour even solving a mild program in Case 1, which is significantly more time consuming than all other competing methods. Thus, we believe such regime may not be suitable for our setup. While there is a scope to combine that method with our approach (that is use it only towards the end of the MPC loop), we defer testing and tuning to future research. The implementation details of each competing method are as follows:
\begin{enumerate}[label=(\alph*),topsep=1pt]
\setlength\itemsep{-0.1em}

\item  A standard RTI-based MPC scheme from \cite{Diehl2005Nominal}. It uses lag $1$ and hence the boundary stages are given by $\{[i, i+M]\}_{i\in[0, N-M]}$. We apply the method on our subproblem formulation \eqref{pro:2}. For simplicity, we use the original Hessian in Newton equation.

\item An online MPC regime which performs multiple Newton steps for each subproblem \cite{Wang2010Fast}. We perform 3 Newton steps in our implementation. It also uses lag $1$, and we replace that subproblem formulation by \eqref{pro:2}. We notice that \cite{Wang2010Fast} studied convex problems, while its idea of running multiple Newton steps per sampling time is not limited to convex problems. Certainly, performing multiple Newton steps is also allowed by other literature (e.g. \cite{Zanelli2021Lyapunov}).

\item Algorithm \ref{alg:1}: we let $L=10$ and, together with (b) and (c), let $M = 80$ for all three cases.

\item Algorithm \ref{alg:1} (lag $1$): based on (c), we set $L = 1$ to test if the condition \eqref{equ:cond:rhoeps} is necessary for the convergence.

\item Algorithm \ref{alg:1} (without boundary techniques): based on (c), we do not use the boundary iterates updating techniques discussed in Remark \ref{rem:2}. 

\end{enumerate}

We note that two tricks --- adopting a larger lag and boundary techniques --- are critical in our analysis, and we test them in (d) and (e) respectively to see if they are necessary for the convergence. For all methods, we perform 10 independent runs and compare the running time and the precision of each middle stage: $\max_{k\in[M, N-M]}\{\|\hbx_k- \tx_k\| \vee \|\hbu_k - \tu_k\| \vee \|\hblambda_k  - \tlambda_k\|\}$.

The results are summarized in Table \ref{tab:2}; we include the standard deviation of running time for multiple runs. From the table, we see that among all competing methods the proposed Algorithm \ref{alg:1} (with and without boundary techniques) has the least running time. All methods except \cite{Wang2010Fast} perform a single Newton step for each subproblem, while \cite{Wang2010Fast} performs 3 Newton steps. Thus, \cite{Wang2010Fast} is expected to be the most time consuming though it is also the most accurate as our results also show. Faster but less accurate than \cite{Wang2010Fast},  \cite{Diehl2005Nominal} and Algorithm \ref{alg:1} ($L=1$) have similar running time and accuracy, which is expected since they are quite similar. The only difference is that Algorithm \ref{alg:1} ($L=1$) applies the boundary techniques, which has no effect on running time. Faster than either \cite{Wang2010Fast},  \cite{Diehl2005Nominal}, or Algorithm \ref{alg:1} ($L=1$), since, due to the larger lag they adopt, they have fewer subproblems to solve,  Algorithm \ref{alg:1} and Algorithm \ref{alg:1} (NoTech) have similar running time. One notable difference is that  Algorithm \ref{alg:1} (NoTech), the fastest overall, always has better accuracy than Algorithm \ref{alg:1} (though it should also be noted that the absolute accuracy is very stringent in all cases). We conclude that while our boundary techniques in Remark \ref{rem:2} seem necessary for local convergence they do not result in a more accurate algorithm. Similarly, while we cannot guarantee with our theory that the case $L=1$ would converge it seems to do just fine in the problem we illustrated here.

In summary, the feature that we predict for our analysis: the per-stage improvement of the tracking error for online algorithms for sufficiently large lag $L$ is exhibited by numerical experiments, but some of the assumptions we need may be stronger than those where this phenomenon is observed in practice. Thus, we believe that there exists a gap between our theory and practice. How to relax condition \eqref{equ:cond:rhoeps} and remove or modify the boundary techniques deserve further analysis.

%% file: sec6.tex
\section{Conclusion}\label{sec:6}

In this paper, we analyze a one-Newton-step-per-horizon online lag-$L$ MPC approach. We prove that it exhibits a phenomenon that we call \textit{superconvergence}. That is, when initialized close enough to the solution of \eqref{pro:1}, not only does the online optimization approach track the solution of \eqref{pro:1}, but it actually decreases the tracking error as the receding horizon moves forward to a minimum value, which decays exponentially in the length of the receding horizon provided that the lag $L$ is suitably large. Our approach is based on three key steps. (i) We modify the terminal objective for subproblems to make sure that the full-horizon problem having positive definite reduced Hessian induces the MPC problem having positive definite reduced Hessian for a suitable, uniform parameter $\mu$. (ii) We derive the one-step error recursion based on the structure of the KKT inverse. By this recursion, the error consists of two parts: algorithmic error and perturbation error.  (iii) Based on the previous two steps, we prove that the proposed algorithm enjoys exponential convergence in terms of the receding-horizon length. As a practical matter, one may not be interested in getting very high accuracy at the cost of increasing the MPC horizon but our analysis elucidates the asymptotic behavior with respect to it, which may guide the selection of key parameters for MPC. To the best of our knowledge, the theoretical guarantee from this paper for online nonlinear MPC was unknown prior to our work.

We note several potential extensions to this work. Some of them are discussed in Remark \ref{rem:6}. Additionally, having to provide an estimate of the dual variable in the MPC subproblem in \eqref{pro:2} is inconvenient, although we emphasize that it is hard to see an alternative if we want to assume SOSC for the full horizon problem only. In our example in Section \ref{sec:5}, ignoring that term worked fine, but we would like to identify larger classes of problems for which we can drop that term or, as an alternative, determine other regularization or multiplier estimation mechanisms for \eqref{pro:2}. Moreover, we have not tackled here the issue of global convergence (to a stationary point) of the approach. It would be  worth analyzing the performance of the proposed MPC strategy when combined with backtracking line search or when exact second-derivative calculations are replaced by structured quasi-Newton approaches. Finally, we aim to extend the analysis to the case where the control and states may have constraints, as was done in \cite{Xu2019Exponentially}.

%% file: appen.tex
\setcounter{page}{1}

This document collects the proofs of the results in Section \ref{sec:4}. In particular, the proofs of results in Sections \ref{sec:4.1}, \ref{sec:4.2}, \ref{sec:4.3} are provided in Appendices \ref{sec:A}, \ref{sec:B}, \ref{sec:C}, respectively.

\section{Proofs of Results in Section \ref{sec:4.1}}\label{sec:A}

\subsection{Proof of Lemma \ref{lem:1}}\label{pf:lem:1}

It suffices to show that for any $\bw_{k:N} = (\bp_{k:N}, \bq_{k:N-1}) \\= (\bp_k; \bq_k; \ldots; \bp_{N-1}; \bq_{N-1}; \bp_N)$ satisfying the constraints
\begin{equation}\label{pequ:1}
\begin{aligned}
\bp_k =& 0,\\
\bp_{j+1} =& A_j\bp_j + B_j\bq_j, \text{\ \ \ for\ } j \in [k, N-1],
\end{aligned}
\end{equation}
we have $\bw_{k:N}^\top H^{k:N}\bw_{k:N}\geq \gamma_H\|\bw_{k:N}\|^2$. Here, $H^{k:N} = \diag(H_k, \ldots, H_N)$ and $\{A_j, B_j, H_j\}_j$ are evaluated at $(\bz_{k:N}, \\ \blambda_{k:N-1}, \bd_{k:N-1})$. Because the constraints in \eqref{pequ:1} are a subset of the ones in the full problem, we complement $\bw_{k:N}$ by $\bw_{0:k-1}$ with $\bw_{0:k-1} = (\bp_{0:k-1}, \bq_{0:k-1}) = (\bp_0; \bq_0; \ldots; \bp_{k-1}; \bq_{k-1}) = (0; 0; \ldots; 0; 0)$ and let $\bw = (\bw_{0:k-1}; \bw_{k:N})$. Then, we have $G(\bz_{0:N-1}; \bd_{0:N-1})\bw = 0$. Further, 
\begin{align*}
\bw_{k:N}^\top H^{k:N}\bw_{k:N} = \bw^\top H\bw\geq \gamma_H \|\bw\|^2 = \gamma_H \|\bw_{k:N}\|^2,
\end{align*}
where the equalities on two sides are due to the setup of $\bw_{0:k-1}$ and the inequality is due to the premise of our lemma. This completes the proof.

\subsection{Proof of Corollary \ref{cor:1}}\label{pf:cor:1}

Let us write out each stage explicitly. Since $(\tbz_T^0, \tblambda_T^0) = (\bz_{n_1(T):N, T}^0,  \blambda_{n_1(T):N-1, T}^0)\in \N_\epsilon(\tbz_T^\star, \tblambda_T^\star)$,
\begin{align*}
\|\bx_{k, T}^0 - \tx_k\|\vee \|\bu_{k, T}^0 - \tu_k\|\vee \| \blambda_{k, T}^0-\tlambda_k\|\leq \epsilon
\end{align*}
for $ k\in[n_1(T), N-1]$ and $\|\bx_{N, T}^0 - \tx_N\|\leq \epsilon$. We extend the pair $(\tbz^0_T, \tblambda^0_T)$ backward by filling with the primal-dual solution. In particular, we let $\bbz = (\tz_{0:n_1(T)-1}; \tbz_T^0)$ and $\bblambda = (\tlambda_{0:n_1(T)-1}; \tblambda_T^0)$. Since $N - n_1(T)\leq M$,
\begin{align*}
(\bbz, \bblambda) \in \N_{\epsilon, N - n_1(T)}(\tz, \tlambda)\subseteq\N_{\epsilon, M}(\tz, \tlambda),
\end{align*}
where the last relation comes from \eqref{equ:setseq}. By Assumption \ref{ass:M:SOSC}, we know $H_{\rm Re}(\bbz, \bblambda; \bd)\succeq \gamma_H I$. Using Lemma~\ref{lem:1} with $k = n_1(T)$, we finish the proof.

\subsection{Proof of Theorem \ref{thm:1}}\label{pf:thm:1}

Analogous to Lemma \ref{lem:1}, it suffices to show that for any $\tbw_i = \big(\bp_{n_1(i):n_2(i)}, \bq_{n_1(i):n_2(i)-1}\big) = \big(\bp_{n_1(i)}; \bq_{n_1(i)}; \ldots; \bp_{n_2(i)-1};\\ \bq_{n_2(i)-1}; \bp_{n_2(i)}\big)$ satisfying
\begin{equation}
\begin{aligned}
\bp_{n_1(i)} = & 0,\\
\bp_{j+1} = & A_j\bp_j + B_j\bq_j, \text{\ for\ } j \in[n_1(i), n_2(i)-1],
\end{aligned}
\end{equation}
we have $\tbw_i^\top H^i\tbw_i\geq \gamma_H \|\tbw_i\|^2$ with $H^i$ defined in Definition \ref{def:4}. For concreteness, $\{A_j, B_j\}_{j = n_1(i)}^{n_2(i)-1}$, defined in Definition \ref{def:1}, are evaluated at $\big(\bz_{n_1(i):n_2(i)-1, i}^0, \bd_{n_1(i):n_2(i)-1}\big)$; $\{H_j\}_{j = n_1(i)}^{n_2(i)-1}$ are evaluated at $\big(\bz_{n_1(i):n_2(i)-1, i}^0, \blambda_{n_1(i):n_2(i)-1, i}^0, \bd_{n_1(i):n_2(i)-1}\big)$; and $H^i_{n_2(i)}$, defined in Definition \ref{def:4}, is evaluated at $\big(\bx_{n_2(i), i}^0, \\\bu_{n_2(i)}^0, \blambda_{n_2(i)}^0, \bd_{n_2(i)}\big)$. Given such $\tbw_i$, we extend it on both sides such that the extended vector stays in the null space of the full problem. According to the proof of Lemma \ref{lem:1}, we extend backward by filling with 0. In particular, we let
\begin{align*}
\bp_{0:n_1(i) - 1} = 0, \text{\ \ \ \ } \bq_{0:n_1(i)-1} = 0. 
\end{align*}
Correspondingly, we fill the evaluation point with a local solution: $(\tz_{0:n_1(i)-1}, \tlambda_{0:n_1(i)-1})$. As for the forward-in-time extension, we have following two cases.

\noindent{\bf Case 1:} $n_2(i)\geq N-t$. In this case, since the system can evolve at most $t-1$ stages forward, we can easily control it. Specifically, we let $\bq_{n_2(i):N-1} = 0$ and
\begin{equation}\label{pequ:2}
\begin{aligned}
&\bp_{n_2(i) + 1} = A_{n_2(i)}(\bx_{n_2(i), i}^0, \bu_{n_2(i)}^0; \bd_{n_2(i)})\bp_{n_2(i)},\\
&\bp_{j+1} = A_j(\tx_j, \tu_j; \bd_j)\bp_j, \text{for\ } j \in[n_2(i)+1, N-1].
\end{aligned}
\end{equation}
Further, the evaluation point is extended by $\bbz_{n_2(i)+1:N} = \tz_{n_2(i)+1:N}$ and $\bblambda_{n_2(i)+1:N-1} = \tlambda_{n_2(i)+1:N-1}$. To summarize, the whole horizon vector is given by
\begin{multline*}
\bw = \big(\overbrace{0;0;\ldots;0;0}^{[0, n_1(i)-1]}; \overbrace{\bp_{n_1(i)}; \bq_{n_1(i)}; \ldots; \bp_{n_2(i)}}^{\tbw_i}; 0;\\ \underbrace{\bp_{n_2(i)+1}; 0; \ldots; \bp_N}_{\eqref{pequ:2}}\big),
\end{multline*}
and the evaluation point is given by
\begin{align*}
\bbz =& \big(\tz_{0:n_1(i)-1}; \bz_{n_1(i):n_2(i)-1, i}^0; \bx_{n_2(i), i}^0; \bu_{n_2(i)}^0; \tz_{n_2(i)+1:N}\big),\\
\bblambda = &\big(\tlambda_{0:n_1(i)-1}; \blambda_{n_1(i):n_2(i)-1, i}^0; \blambda_{n_2(i)}^0; \tlambda_{n_2(i)+1:N-1}\big).
\end{align*}
Based on this complement, $G(\bbz_{0:N-1}; \bd_{0:N-1})\bw = 0$. Since $(\bbz, \bblambda)\in \N_{\epsilon, M}(\tz, \tlambda)$, by Assumption \ref{ass:M:SOSC},
\begin{align}\label{pequ:3}
\bw^\top H(\bbz, \bblambda; \bd)\bw \geq \gamma_H\|\bw\|^2.
\end{align}
For the left-hand side term, by the definition of $H^i$ in Definition~\ref{def:4},
\begin{align*}
\bw^\top H(\bbz, \bblambda; &\bd)\bw  \\
=& \tbw_i^\top H^i\tbw_i - \mu\|\bp_{n_2(i)}\|^2 + \sum_{j=n_2(i)+1}^{N}\bp_j^\top Q_j\bp_j \\
\leq & \tbw_i^\top H^i\tbw_i - \mu\|\bp_{n_2(i)}\|^2 + \Upsilon \sum_{j=n_2(i)+1}^{N}\|\bp_j\|^2,
\end{align*}
where the inequality is due to Assumption \ref{ass:M:Ubound}. Plugging in \eqref{pequ:3} and using $\gamma_H\|\bw\|^2\geq \gamma_H\|\tbw_i\|^2$, we get
\begin{equation}\label{pequ:4}
\tbw_i^\top H^i\tbw_i \geq \gamma_H\|\tbw_i\|^2+ \mu\|\bp_{n_2(i)}\|^2  - \Upsilon\|\bp_{n_2(i)+1:N}\|^2.
\end{equation}
On the other hand, by \eqref{pequ:2}, boundedness of $\{A_j\}_{j = n_2(i)}^{N-1}$ in Assumption \ref{ass:M:Ubound}, and the condition that $N - n_2(i)\leq t$, we get
\begin{align*}
\small \|\bp_{n_2(i)+1:N}\|^2\leq \sum_{j=1}^{N - n_2(i)}\Upsilon^{2j}\|\bp_{n_2(i)}\|^2\leq \frac{\Upsilon^{2(t+1)} - \Upsilon^2}{\Upsilon^2-1}\|\bp_{n_2(i)}\|^2.
\end{align*}
Together with \eqref{pequ:4}, this implies
\begin{equation}\label{pequ:5}
\small \tbw_i^\top H^i\tbw_i \geq \gamma_H\|\tbw_i\|^2 + \big(\mu - \frac{\Upsilon(\Upsilon^{2(t+1)} - \Upsilon^2)}{\Upsilon^2-1}\big)\|\bp_{n_2(i)}\|^2.
\end{equation}
Therefore, $\mu\geq \frac{\Upsilon(\Upsilon^{2(t+1)} - \Upsilon^2)}{\Upsilon^2-1}$ implies $\tbw_i^\top H^i\tbw_i \geq \gamma_H\|\tbw_i\|^2$.

\noindent{\bf Case 2:} $n_2(i)< N - t$. In this case, we use the controllability condition to do the forward-in-time extension. In particular, we first set $\bq_{n_2(i)} = 0$ and $\bp_{n_2(i)+1}$ as in \eqref{pequ:2}. Then, we apply the following linear dynamics recursively:
\begin{align*}
\small \bp_{j+1} = A_j(\tz_j; \bd_j)\bp_j + B_j(\tz_j; \bd_j)\bq_j, \text{\ for\ } j\in[n_2(i)+1, N-1].
\end{align*}
Thus, we have $\forall l \geq 1$
\begin{equation}\label{pequ:6}
\small \bp_{n_2(i)+1+l} = \rbr{\prod_{h=1}^{l}A_{n_2(i)+h}}\bp_{n_2(i)+1} + \Xi_{n_2(i)+1, l}\begin{pmatrix}
\bq_{n_2(i)+l}\\
\vdots\\
\bq_{n_2(i)+1}
\end{pmatrix}.
\end{equation}
For these matrices, $\{A_j, B_j\}_{j = n_2(i)+1}^{n_2(i)+l}$ are evaluated at the local solution. Let $l = t_{n_2(i)+1}$ with $t_{n_2(i)+1}$ defined in Assumption \ref{ass:M:control}. We set $\bq_{n_2(i)+1: n_2(i)+t_{n_2(i)+1}}$ as
\begin{equation}\label{pequ:7}
\small\begin{pmatrix}
\bq_{n_2(i)+t_{n_2(i)+1}}\\
\vdots\\
\bq_{n_2(i)+1}
\end{pmatrix} = \small - \Xi^\top\rbr{\Xi\Xi^\top}^{-1}\bigg(\prod_{h=1}^{t_{n_2(i)+1}}A_{n_2(i)+h}\bigg)\bp_{n_2(i)+1},
\end{equation}
where $\Xi = \Xi_{n_2(i)+1, t_{n_2(i)+1}}$. Given \eqref{pequ:7}, we calculate $\bp_k$ for $k\in [n_2(i)+2, n_2(i)+1 + t_{n_2(i)+1}]$ using \eqref{pequ:6} with $l = 1, \ldots, t_{n_2(i)+1}$. Then, $\bp_{n_2(i)+1 + t_{n_2(i)+1}} = 0$. Further, we let $\bq_{n_2(i)+1 + t_{n_2(i)+1}:N-1} = 0$ and $\bp_{n_2(i)+2 + t_{n_2(i)+1}:N} = 0$. In summary, the full-horizon evaluation point is same as in Case 1, while the null space vector in this case is
\begin{multline*}
\bw = \big(\overbrace{0;0;\ldots;0;0}^{[0, n_1(i)-1]}; \overbrace{\bp_{n_1(i)}; \bq_{n_1(i)}; \ldots; \bp_{n_2(i)}}^{\tbw_i}; 0;\bp_{n_2(i)+1}; \\
\underbrace{\bq_{n_2(i)+1}; \ldots \bq_{n_2(i)+t_{n_2(i)+1}}}_{\eqref{pequ:6}, \eqref{pequ:7}}; \underbrace{0; 0; \ldots; 0; 0}_{[n_2(i)+1 + t_{n_2(i)+1}, N]}\big).
\end{multline*}
By construction, it follows that $G(\bbz_{0:N-1}; \bd_{0:N-1})\bw = 0$, so \eqref{pequ:3} also holds. Let us bound the magnitude of the forward-in-time extension. Note that
\begin{equation*}
\|\Xi_{n_2(i)+1, l}\|\leq \Upsilon + \Upsilon^2 + \ldots + \Upsilon^l\leq \frac{\Upsilon(\Upsilon^l - 1)}{\Upsilon - 1} \coloneqq \psi_l, \text{\ } \forall l\geq 1. 
\end{equation*}
Therefore, by \eqref{pequ:7} and Assumption \ref{ass:M:control},
\begin{align}\label{pequ:8}
\|\bq_{n_2(i)+1:n_2(i)+t_{n_2(i)+1}}\|\leq &\frac{\psi_{t_{n_2(i)+1}}\Upsilon^{t_{n_2(i)+1}}}{\gamma_C} \|\bp_{n_2(i)+1}\| \nonumber\\
\stackrel{\eqref{pequ:2}}{\leq} &\frac{\psi_t\Upsilon^{t+1}}{\gamma_C}\|\bp_{n_2(i)}\|.
\end{align}
For $\{\bp_k\}_{k = n_2(i)+1}^{n_2(i)+t_{n_2(i)+1}}$, we have $\|\bp_{n_2(i)+1}\|\leq  \Upsilon \|\bp_{n_2(i)}\|$ and, by \eqref{pequ:6} and \eqref{pequ:8}, 
\begin{align*}
\|\bp_{n_2(i)+l}\|&\leq  \Upsilon^l\|\bp_{n_2(i)}\| + \psi_{l-1}\|\bq_{n_2(i)+1:n_2(i)+t_{n_2(i)+1}}\| \\
\leq &\rbr{\Upsilon^l + \frac{\psi_{l-1}\psi_t\Upsilon^{t+1}}{\gamma_C}}\|\bp_{n_2(i)}\|, \text{\ } \small \forall 2\leq l\leq t_{n_2(i)+1}.
\end{align*}
Thus,
\begin{align*}
\|&\bp_{n_2(i)+1:n_2(i)+t_{n_2(i)+1}}\|^2 =  \sum_{l=1}^{t_{n_2(i)+1}}\|\bp_{n_2(i)+l}\|^2 \\
&\leq \bigg(\Upsilon^2 +  2\sum_{l=2}^{t_{n_2(i)+1}} \big(\Upsilon^{2l} + \frac{\psi_{l-1}^2\psi_t^2\Upsilon^{2(t+1)}}{\gamma_C^2}\big)\bigg)\|\bp_{n_2(i)}\|^2\\
&\leq 2 \rbr{\sum_{l=1}^t\Upsilon^{2l} + \frac{\psi_t^2\Upsilon^{2(t+1)}}{\gamma_C^2}\sum_{l=1}^{t-1}\psi_l^2}\|\bp_{n_2(i)}\|^2.
\end{align*}
Combining this inequality with \eqref{pequ:8}, we get
{\small
\begin{multline*}
\|\bq_{n_2(i)+1:n_2(i)+t_{n_2(i)+1}}\|^2 + \|\bp_{n_2(i)+1:n_2(i)+t_{n_2(i)+1}}\|^2\\ \leq \Upsilon'\|\bp_{n_2(i)}\|^2
\end{multline*}
}with $\Upsilon' = 2\rbr{\frac{\Upsilon^{2(t+1)} - \Upsilon^2}{\Upsilon^2-1} + \frac{\psi_t^2\Upsilon^{2(t+1)}}{\gamma_C^2}\cdot\frac{\Upsilon^2}{(\Upsilon - 1)^2}\cdot\frac{\Upsilon^{2t} - 1}{\Upsilon^2-1}}$. On the other hand,
\begin{align*}
\small
&\bw^\top H(\bbz, \bblambda; \bd)\bw \\
&= \tbw_i^\top H^i\tbw_i - \mu\|\bp_{n_2(i)}\|^2+ \sum_{j=n_2(i)+1}^{n_2(i)+t_{n_2(i)+1}}(\bp_j; \bq_j)^\top H_j(\bp_j; \bq_j)\\
&\leq  \tbw_i^\top H^i\tbw_i - \mu\|\bp_{n_2(i)}\|^2 + \Upsilon\sum_{j=n_2(i)+1}^{n_2(i)+t_{n_2(i)+1}}\|\bp_j\|^2 + \|\bq_j\|^2.
\end{align*}
Combining with \eqref{pequ:3},  we get
{\small 
\begin{align*}
\tbw_i^\top&H^i\tbw_i\\
\geq &\gamma_H\|\tbw_i\|^2 + \mu\|\bp_{n_2(i)}\|^2 - \Upsilon\sum_{j=n_2(i)+1}^{n_2(i)+t_{n_2(i)+1}}\|\bp_j\|^2 + \|\bq_j\|^2\\
\geq & \gamma_H\|\tbw_i\|^2 + \rbr{\mu - \Upsilon\Upsilon'}\|\bp_{n_2(i)}\|^2.
\end{align*}
}So in this case we need $\mu \geq \Upsilon\Upsilon'$. Note that this  condition implies the one in Case 1. To simplify it, without loss of generality, we assume $\Upsilon\geq \sqrt{2}/(\sqrt{2}-1)$, then $\Upsilon^2/(\Upsilon^2-1)\leq \Upsilon^2/(\Upsilon - 1)^2\leq 2$. Then, $\Upsilon\Upsilon' \leq \frac{16\Upsilon(\Upsilon^{6t} - \Upsilon^{4t})}{\gamma_C^2}$. This completes the proof.

\section{Proofs of Results in Section \ref{sec:4.2}}\label{sec:B}

\subsection{Proof of Lemma \ref{lem:2}}\label{pf:lem:2}

Let us focus on solving $K\bb = \ba$, where $\bb = (\bw; \bbeta)$ and $\ba = (\bv; \balpha)$ with $\bw = (\bp, \bq)$ and $\bv = (\bl, \br)$ being the variables ordered by stages. Then, the magnitude of each block of $K^{-1}$ can be reflected by the magnitude of $\bb$ if one sets $\ba$ properly. Note that $K\bb = \ba$ is the first-order necessary condition of the following linear-quadratic problem:
\begin{equation}\label{pequ:10}
\begin{aligned}
\min_{\bp, \bq} \text{\ \ \ } & \frac{1}{2}\sum_{k=0}^{N-1} (\bp_k; \bq_k)^\top H_k(\bp_k; \bq_k) + \frac{1}{2}\bp_N^\top H_N\bp_N \\
&\quad\quad\quad- \sum_{k=0}^{N-1}(\bp_k; \bq_k)^\top(\bl_k; \br_k) - \bp_N^\top \bl_N,\\
\text{s.t.}\text{\ \ \ } & \bp_{k+1} = A_k\bp_k + B_k\bq_k + \balpha_k, \text{\ \ } k \in[N-1],\\
& \bp_0 = \balpha_{-1},
\end{aligned}
\end{equation}
where $\{A_k, B_k, H_k\}$ are evaluated at $(\bz, \blambda; \bd)$, and $\bbeta$ in $\bb$ are the Lagrange multipliers. Since SOSC holds for this problem, the point that satisfies the first-order necessary condition is indeed its global solution. Thus, $\bb$ is the optimal primal-dual solution of this problem. Since \eqref{equ:SOSC}, \eqref{equ:control} and \eqref{equ:bound} hold at $(\bz, \blambda)$, by Theorem 5.7 in \cite{Na2020Exponential} and Theorem 5 in \cite{Na2020Overlapping}\footnote{In fact, the linear-quadratic problem in \eqref{pequ:10} is slightly more general than the one in \cite{Na2020Exponential}. But the same argument also holds with a simple extension of that proof. In particular, Theorem 5.7 in \cite{Na2020Exponential} provides the exponential decay of primal variables $(\bp_k; \bq_k)$, while Theorem 5 in \cite{Na2020Overlapping} provides the exponential decay of dual variables $\bbeta_k$.}, there exist constants $\Upsilon_K$, $\rho\in(0, 1)$ such that $\forall i, j \in[0, N]$,
\begin{align*}
\|(\bp_i; \bq_i; \bbeta_{i-1})\|\leq \Upsilon_K\rho^{|i-j|}
\end{align*}
if $\|(\bl_j; \br_j; \balpha_j)\|= 1$ and $\|(\bl_h; \br_h; \balpha_h)\|= 0$ for $h \neq j$. Also, $\|(\bp_i; \bq_i; \bbeta_{i-1})\|\leq \Upsilon_K\rho^i$ if $\|\balpha_{-1}\|= 1$ and other components in $\ba$ are zero. Using the above result, we further bound all blocks as follows. For $i, j\in[N]$, we let $\|(\bl_j; \br_j)\|= 1$ (when $j = N$, we let $\|\bl_N\|= 1$) and let $\balpha_j$ together with $(\bl_h; \br_h; \balpha_h)$ for $h\neq j$ be zero. Then,
\begin{align*}
\|(\bp_i; \bq_i)\| = \|K^{-1}_{(i, j), 1}(\bl_j; \br_j)\|\leq \Upsilon_K\rho^{|i-j|}.
\end{align*}
Since $(\bl_j; \br_j)$ can be any unit vector,
\begin{align}\label{pequ:11}
\|K^{-1}_{(i, j), 1}\| = \max_{\|(\bl_j; \br_j)\|= 1}\|K^{-1}_{(i, j), 1}(\bl_j; \br_j)\|\leq \Upsilon_K\rho^{|i-j|}.
\end{align}
Analogously, we let $\|\balpha_{-1}\| = 1$ and other components in $\ba$ be zero. Then $\forall i\in[N]$ ($\bq_N$ is not present)
\begin{equation}\label{pequ:12}
\begin{aligned}
&\small{\|(\bp_i; \bq_i)\| = \|K^{-1}_{(i, -1), 2}\balpha_{-1}\|\leq \Upsilon_K\rho^{i} \Rightarrow \|K^{-1}_{(i, -1), 2}\|\leq \Upsilon_K\rho^{i}},\\
&\small{\|\bbeta_{i-1}\| =  \|K^{-1}_{(i-1, -1), 3}\balpha_{-1}\|\leq \Upsilon_K\rho^{i} \Rightarrow \|K^{-1}_{(i-1, -1), 3}\|\leq \Upsilon_K\rho^{i}}.
\end{aligned}
\end{equation}
By letting $\|\balpha_{j}\| = 1$ for $j\in[N-1]$, 
\begin{equation}\label{pequ:13}
\begin{aligned}
&\small{\|(\bp_i; \bq_i)\| = \|K^{-1}_{(i, j), 2}\balpha_{j}\|\leq \Upsilon_K\rho^{|i-j|} \Rightarrow \|K^{-1}_{(i, j), 2}\|\leq \Upsilon_K\rho^{|i-j|}},\\
&\small{\|\bbeta_{i-1}\| =  \|K^{-1}_{(i-1, j), 3}\balpha_{j}\|\leq \Upsilon_K\rho^{i} \Rightarrow \|K^{-1}_{(i-1, j), 3}\|\leq \Upsilon_K\rho^{|i-j|}}.
\end{aligned}
\end{equation}
Combining results in \eqref{pequ:11}, \eqref{pequ:12}, \eqref{pequ:13}, we complete the proof.

\subsection{Proof of Corollary \ref{cor:2}}\label{pf:cor:2}

By Lemma \ref{lem:2}, we  need to check only \eqref{equ:SOSC}, \eqref{equ:control}, and \eqref{equ:bound} at $(\tbz_i^0, \tblambda_i^0)$. For any $i\in[1, T]$, by Corollary \ref{cor:1} and Theorem \ref{thm:1} we know that \eqref{equ:SOSC} holds for subproblem $i$ with the same lower bound of the reduced Hessian. Further, \eqref{equ:control} is implied directly from Assumption \ref{ass:M:control}. In fact, for the horizon $[n_1(i), n_2(i)]$, since $\forall k\in[n_1(i), n_2(i) - t]\subseteq [N-t]$ and $\tbz_i^0\in\N_\epsilon(\tbz_i^\star)$, we have
\begin{align*}
\Xi_{k, t_k}\Xi_{k, t_k}^\top\succeq \gamma_C I,
\end{align*}
where $\Xi_{k, t_k}$ is evaluated at $(\bz_{k:k+t_k-1, i}^0; \bd_{k:k+t_k-1})$. This verifies \eqref{equ:control}. For the upper boundedness, we note that $A_k, B_k$ is bounded by $\Upsilon$. Because of the extra quadratic term in the objective of \eqref{pro:2}, each block in the Hessian is bounded by $\Upsilon + \mu$. By \eqref{equ:set:mu}, we know that the lower bound of $\mu$, $\mu(\Upsilon, \gamma_C, t)$, is independent from $N$. Thus, the uniform upper boundedness condition \eqref{equ:bound} also holds. Combining all these observations and applying Lemma \ref{lem:2}, we complete the proof.

\subsection{Proof of Theorem \ref{thm:2}}\label{pf:thm:2}

We take the $i$-th subproblem as an example. All matrices are evaluated at the initial point $(\tbz_i^0, \tblambda_i^0)$ by default. By \eqref{equ:Newton}--\eqref{equ:Update} and noting that the KKT matrix, $K^i$, is invertible due to Corollary~\ref{cor:1}, Theorem \ref{thm:1}, and LICQ (see Remark \ref{rem:KKT:inv}), we get 
\begin{align}\label{pequ:14}
&\left(\begin{smallmatrix}
\tbz_i^1 - \tbz_i^\star\\
\tblambda_i^1 - \tblambda_i^\star
\end{smallmatrix}\right) = \left(\begin{smallmatrix}
\tbz_i^0 - \tbz_i^\star\\
\tblambda_i^0 - \tblambda_i^\star
\end{smallmatrix}\right) - \left(\begin{smallmatrix}
H^i & (G^i)^\top\\
G^i & 0
\end{smallmatrix}\right)^{-1}\left(\begin{smallmatrix}
\nabla_{\tbz_i}\mL^i\\
\nabla_{\tblambda_i}\mL^i
\end{smallmatrix}\right) \nonumber\\
&=  \left(\begin{smallmatrix}
H^i & (G^i)^\top\\
G^i & 0
\end{smallmatrix}\right)^{-1}\bigg\{\left(\begin{smallmatrix}
H^i & (G^i)\top\\
G^i & 0
\end{smallmatrix}\right)\left(\begin{smallmatrix}
\tbz_i^0 - \tbz_i^\star\\
\tblambda_i^0 - \tblambda_i^\star
\end{smallmatrix}\right) - \bigg( \left(\begin{smallmatrix}
\nabla_{\tbz_i}\mL^i\\
\nabla_{\tblambda_i}\mL^i
\end{smallmatrix}\right) \nonumber\\
& \quad - \left(\begin{smallmatrix}
\nabla_{\tbz_i}\mL^{i \star}\\
\nabla_{\tblambda_i}\mL^{i \star}
\end{smallmatrix}\right)\bigg) - \left(\begin{smallmatrix}
\nabla_{\tbz_i}\mL^{i \star}\\
\nabla_{\tblambda_i}\mL^{i \star}
\end{smallmatrix}\right)\bigg\},
\end{align}
where $\nabla\mL^{i \star} = \nabla \mL^i(\tbz_i^\star, \tblambda_i^\star; \tbd_i)$. Let $H^i_{\eta} = H^i\big(\eta(\tbz_i^0, \tblambda_i^0) + (1-\eta)(\tbz_i^\star, \tblambda_i^\star); \tbd_i\big)$ and $G^i_{\eta} = G^i\big(\eta\tbz_i^0 + (1-\eta)\tbz_i^\star; \tbd_i\big)$. Then
\begin{align*}
\left(\begin{smallmatrix}
\nabla_{\tbz_i}\mL^i\\
\nabla_{\tblambda_i}\mL^i
\end{smallmatrix}\right) - \left(\begin{smallmatrix}
\nabla_{\tbz_i}\mL^{i \star}\\
\nabla_{\tblambda_i}\mL^{i \star}
\end{smallmatrix}\right) = \int_{0}^{1} \left(\begin{smallmatrix}
H^i_\eta & (G^i_\eta)^\top\\
G^i_\eta & 0
\end{smallmatrix}\right) \left(\begin{smallmatrix}
\tbz_i^0 - \tbz_i^\star\\
\tblambda_i^0 - \tblambda_i^\star
\end{smallmatrix}\right) d\eta.
\end{align*}
Plugging into \eqref{pequ:14}, we get
\begin{align}\label{pequ:15}
&\left(\begin{smallmatrix}
\tbz_i^1 - \tbz_i^\star\\
\tblambda_i^1 - \tblambda_i^\star
\end{smallmatrix}\right) = \left(\begin{smallmatrix}
H^i & (G^i)^\top\\
G^i & 0
\end{smallmatrix}\right)^{-1} \bigg\{ - \left(\begin{smallmatrix}
\nabla_{\tbz_i}\mL^{i \star}\\
\nabla_{\tblambda_i}\mL^{i \star}
\end{smallmatrix}\right) \nonumber\\
& \quad\quad\quad+ \int_{0}^{1}\overbrace{ \left(\begin{smallmatrix}
	H^i - H^i_\eta & (G^i)^\top - (G^i_\eta)^\top\\
	G^i - G^i_\eta & 0
	\end{smallmatrix}\right)}^{W^i_\eta \coloneqq K^i - K^i_\eta} \left(\begin{smallmatrix}
\tbz_i^0 - \tbz_i^\star\\
\tblambda_i^0 - \tblambda_i^\star
\end{smallmatrix}\right)d\eta\bigg\}.
\end{align}
Let us denote $W^i_\eta\big(\tbz_i^0 - \tbz_i^\star; \tblambda_i^0 - \tblambda_i^\star\big) = (\Delta\bw; \Delta\bbeta)$ and, as in Appendix \ref{pf:lem:2}, $\Delta\bw = (\Delta\bp_{n_1(i)}; \Delta\bq_{n_1(i)}; \ldots; \Delta\bp_{n_2(i)})$ and $\Delta\bbeta = (\Delta\bbeta_{n_1(i)-1}; \ldots; \Delta\bbeta_{n_2(i)-1})$. We also let $\Delta\bw_k = (\Delta\bp_k; \Delta\bq_k)$ for $k\in[n_1(i), n_2(i)-1]$ and $\Delta\bw_{n_2(i)} = \Delta\bp_{n_2(i)}$. 
Using the matrix partition in Definition \ref{def:5} for $W^i_\eta$, we have $\forall k\in[n_1(i), n_2(i)-1]$
\begin{align}\label{pequ:16}
\|\Delta\bw_k\| = &\|(\Delta\bp_k; \Delta\bq_k)\| \nonumber\\
= &\|(W^i_\eta)_{(k,k),1}(\bz_{k, i}^0 - \tz_k) + (W^i_\eta)_{(k, k), 2}(\blambda_{k, i}^0 - \tlambda_k)\| \nonumber\\
\leq & (1-\eta)\Upsilon_L\|(\bz_{k, i}^0 - \tz_k; \blambda_{k, i}^0 - \tlambda_k)\|\|\bz_{k, i}^0 - \tz_k\| \nonumber\\
&+ (1-\eta)\Upsilon_L\|\bz_{k, i}^0 - \tz_k\|\|\blambda_{k, i}^0 - \tlambda_k\| \nonumber\\
\leq &2(1 - \eta)\Upsilon_L(\Psi_{k, i}^0)^2.
\end{align}
Here, the first equality is due to block matrix multiplication algebra and the fact that $(W_\eta^i)_{(k, n_1(i)-1), 2} = 0$; the first inequality is due to Assumption \ref{ass:M:Lip:cond}. In particular, $(\bz_{k,i}^\eta, \blambda_{k,i}^\eta) \coloneqq \eta(\bz_{k, i}^0, \blambda_{k, i}^0) + (1 - \eta)(\tz_k, \tlambda_k) \in \N_\epsilon(\tz_k, \tlambda_k)$, and (we drop  $\bd_k$)
\begin{align*}
\|(W^i_\eta)_{(k, k), 1}\| = &\|H_k(\bz_{k,i}^0, \blambda_{k,i}^0) - H_k(\bz_{k,i}^\eta, \blambda_{k,i}^\eta)\| \nonumber\\
\leq & (1-\eta)\Upsilon_L\|(\bz_{k, i}^0 - \tz_k; \blambda_{k, i}^0 - \tlambda_k)\|,\\
\|(W^i_\eta)_{(k, k), 2}\| = & \|\nabla_{\bz_k}f_k(\bz_{k, i}^0) - \nabla_{\bz_k}f_k(\bz_{k, i}^\eta)\| \\
\leq& (1-\eta)\Upsilon_L\|\bz_{k, i}^0 - \tz_k\|.
\end{align*}
For $k = n_2(i)$, by Definition \ref{def:4} and Assumption \ref{ass:M:Lip:cond}, we get
\begin{multline*}
\|(W^i_\eta)_{(n_2(i), n_2(i)), 1}\| = \|Q_{n_2(i)}^0 - Q_{n_2(i)}^\eta\| \\
\leq  (1 - \eta)\Upsilon_L\Psi_{n_2(i), i}^0,
\end{multline*}
where $Q_{n_2(i)}^0\coloneqq Q_{n_2(i)}(\bx_{n_2(i), i}^0, \bu_{n_2(i)}^0, \blambda_{n_2(i)}^0)$ and $Q_{n_2(i)}^\eta\coloneqq Q_{n_2(i)}(\bx_{n_2(i), i}^\eta, \bu_{n_2(i)}^0, \blambda_{n_2(i)}^0)$. Thus,
\begin{align}\label{pequ:17}
\|\Delta\bw_{n_2(i)}\| =& \|(W^i_\eta)_{(n_2(i), n_2(i)), 1}(\bx_{n_2(i), i}^0 - \tx_{n_2(i)})\| \nonumber\\
\leq &(1 - \eta)\Upsilon_L(\Psi_{n_2(i), i}^0)^2.
\end{align}
Analogously, we deal with $\Delta\bbeta$. First, $\Delta\bbeta_{n_1(i)-1} = 0$, since $(W_\eta^i)_{(k, n_1(i)-1), 2}^\top = 0$, $\forall k \in[n_1(i), n_2(i)]$. Then, for $k\in[n_1(i), n_2(i)-1]$,
\begin{align}\label{pequ:18}
\|\Delta\bbeta_k\| &= \|(W^i_\eta)^\top_{(k,k),2}(\bz_{k, i}^0 - \tz_k)\| \nonumber\\
\leq& (1 - \eta)\Upsilon_L\|\bz_{k, i}^0 - \tz_k\|^2\leq  (1 - \eta)\Upsilon_L(\Psi_{k, i}^0)^2.
\end{align}
We then simplify the first term in \eqref{pequ:15}. To ease~notations, we denote $(\bv^\star; \balpha^\star)\coloneqq\big(\nabla_{\tbz_i}\mL^{i \star}; \nabla_{\tblambda_i}\mL^{i \star}\big)$ with $\bv^\star = (\bl^\star_{n_1(i)}; \br^\star_{n_1(i)}; \ldots; \bl_{n_2(i)}^\star)$ and $\balpha^\star = (\balpha_{n_1(i)-1}^\star; \ldots; \balpha_{n_2(i)-1}^\star)$. As usual, we let $\bv^\star_k = (\bl_k^\star; \br_k^\star)$. 
By the expression of $\mL^i$ in Definition \ref{def:4} and the first-order necessary condition on $(\tbz_i^\star, \tblambda_i^\star; \\\tbd_i)$, we have
\begin{align}\label{pequ:19}
\bv_{n_1(i):n_2(i)-1}^\star = 0, \quad \balpha_{n_1(i):n_2(i)-1}^\star = 0.
\end{align}
For $\bl_{n_2(i)}^\star$, we let $\bu_{n_2(i)}^\eta = \eta\bu_{n_2(i)}^0 + (1 - \eta)\tu_{n_2(i)}$ (similarly for $\blambda_{n_2(i)}^\eta$) and define (we ignore $\bd_{n_2(i)}$)
\begin{align*}
&\nabla_{\bx_{n_2(i)}}\mL_{n_2(i)}^0 = \nabla_{\bx_{n_2(i)}}\mL_{n_2(i)}(\tx_{n_2(i)}, \bu_{n_2(i)}^0, \tlambda_{n_2(i)-1}, \blambda_{n_2(i)}^0),\\
&\nabla_{\bx_{n_2(i)}}\mL_{n_2(i)}^\star = \nabla_{\bx_{n_2(i)}}\mL_{n_2(i)}(\tx_{n_2(i)}, \bu_{n_2(i)}^\star, \tlambda_{n_2(i)-1}, \blambda_{n_2(i)}^\star).
\end{align*}
Using $\nabla_{\bx_{n_2(i)}}\mL_{n_2(i)}^\star = 0$ as the first-order necessary condition,  we get
\begin{align*}
\bl_{n_2(i)}^\star &=  \nabla_{\bx_{n_2(i)}}\mL_{n_2(i)}^0 + \mu(\tx_{n_2(i)} - \bx_{n_2(i)}^0)\\
=& \nabla_{\bx_{n_2(i)}}\mL_{n_2(i)}^0  - \nabla_{\bx_{n_2(i)}}\mL_{n_2(i)}^\star + \mu(\tx_{n_2(i)} - \bx_{n_2(i)}^0)\\
= & \int_{0}^{1}\begin{pmatrix}
\nabla_{\bx_{n_2(i)}\bu_{n_2(i)}}^2\mL_{n_2(i)}^\eta& A_{n_2(i)}^{T\; \eta}
\end{pmatrix}\begin{pmatrix}
\bu_{n_2(i)}^0 - \tu_{n_2(i)}\\
\blambda_{n_2(i)}^0 - \tlambda_{n_2(i)}
\end{pmatrix}d\eta\\
&+ \mu(\tx_{n_2(i)} - \bx_{n_2(i)}^0),
\end{align*}
where $\mL_{n_2(i)}^\eta = \mL_{n_2(i)}(\tx_{n_2(i)}, \bu_{n_2(i)}^\eta, \tlambda_{n_2(i)-1}, \blambda_{n_2(i)}^\eta)$ and $A_{n_2(i)}^\eta = A_{n_2(i)}(\tx_{n_2(i)}, \bu_{n_2(i)}^\eta)$. By Assumption \ref{ass:M:Ubound} and the fact that $(\bz_{n_2(i)}^0, \blambda_{n_2(i)}^0)\in\N_\epsilon(\tz_{n_2(i)}, \tlambda_{n_2(i)})$, we get
\begin{align}\label{pequ:20}
\|\bl_{n_2(i)}^\star\|\leq &2\Upsilon\|\big(\bu_{n_2(i)}^0 - \tu_{n_2(i)};
\blambda_{n_2(i)}^0 - \tlambda_{n_2(i)}\big)\| \nonumber\\
&+ \mu\|\tx_{n_2(i)} - \bx_{n_2(i)}^0\| \leq  (4\Upsilon + \mu) \epsilon.
\end{align}
Finally, by definition, we have $\balpha_{n_1(i)-1}^\star = \nabla_{\blambda_{n_1(i)-1}}\mL^{i \star} = \tx_{n_1(i)} - \bx_{n_1(i), i-1}^1 = \tx_{n_1(i)} - \bx_{n_1(i), i}^0$. The last equality is~from \eqref{equ:trans1}. Together with \eqref{pequ:16}--\eqref{pequ:20} and plugging into \eqref{pequ:15}, taking integral over $\eta$, and using the structure of $(K^i)^{-1}$ in Corollary \ref{cor:2}, we have $\forall k\in[n_1(i), n_2(i)]$ ($\bz_{n_2(i), i}^\Id = \bx_{n_2(i), i}^\Id$)
\begin{footnotesize}
\begin{align}\label{pequ:21}
& \|\bz_{k, i}^1 - \tz_k\| = \bigg\|\sum_{j_1 = n_1(i)}^{n_2(i)}(K^i)^{-1}_{(k, j_1), 1}\int_{0}^{1}\Delta\bw_{j_1}d\eta \nonumber\\
& \text{\ }  + \sum_{j_2 = n_1(i)-1}^{n_2(i)-1}(K^i)^{-1}_{(k, j_2), 2}\int_{0}^{1}\Delta\bbeta_{j_2}d\eta  - \sum_{j_1 = n_1(i)}^{n_2(i)}(K^i)^{-1}_{(k, j_1), 1}\bv^\star_{j_1} \nonumber\\
&\text{\ }  - \sum_{j_2 = n_1(i)-1}^{n_2(i)-1}(K^i)^{-1}_{(k, j_2), 2}\balpha^\star_{j_2}\bigg\| \nonumber\\
& \leq  \Upsilon_K\Upsilon_L\bigg(\sum_{j_1 = n_1(i)}^{n_2(i)}\rho^{|k - j_1|}(\Psi_{j_1, i}^0)^2 + \sum_{j_2 = n_1(i)}^{n_2(i)-1} \rho^{|k - j_2|}(\Psi_{j_2, i}^0)^2\bigg)\nonumber\\
& \text{\ } + \Upsilon_K\rho^{n_2(i) - k}\big(4\Upsilon + \mu\big)\epsilon  + \Upsilon_K\rho^{k - n_1(i)}\|\bx_{n_1(i), i}^0 - \tx_{n_1(i)}\| \nonumber\\
&  \leq \Upsilon_1\bigg( \sum_{j=n_1(i)}^{n_2(i)}\rho^{|k-j|}(\Psi_{j, i}^0)^2 + \rho^{n_2(i) - k}\epsilon \nonumber\\
&\text{\ } + \rho^{k - n_1(i)}\|\bx_{n_1(i), i}^0 - \tx_{n_1(i)}\|\bigg),
\end{align}
\end{footnotesize}
\hskip -4pt 
where $\Upsilon_1 \coloneqq 2\Upsilon_K\Upsilon_L \vee (4\Upsilon+\mu)\Upsilon_K$. Similarly, we bound multipliers. For $k \in[n_1(i), n_2(i)-1]$,
\begin{footnotesize}
\begin{align}\label{pequ:23}
& \|\blambda_{k, i}^1 - \tlambda_k\| = \bigg\|\sum_{j_1 = n_1(i)}^{n_2(i)}(K^i)^{-1\ T}_{(j_1, k), 2}\int_{0}^{1}\Delta\bw_{j_1}d\eta \nonumber\\
&\text{\ }  + \sum_{j_2 = n_1(i)-1}^{n_2(i)-1}(K^i)^{-1}_{(k, j_2), 3}\int_{0}^{1}\Delta\bbeta_{j_2}d\eta - \sum_{j_1 = n_1(i)}^{n_2(i)}(K^i)^{-1\ T}_{(j_1, k), 2}\bv^\star_{j_1} \nonumber\\
& \text{\ } - \sum_{j_2 = n_1(i)-1}^{n_2(i)-1}(K^i)^{-1}_{(k, j_2), 3}\balpha_{j_2}^\star\bigg\|\nonumber\\
&  \leq  \Upsilon_K\Upsilon_L \bigg(\sum_{j_1 = n_1(i)}^{n_2(i)}\rho^{|k - j_1|}(\Psi_{j_1, i}^0)^2 + \sum_{j_2 = n_1(i)}^{n_2(i)-1}\rho^{|k+1 - j_2|}(\Psi_{j_2, i}^0)^2\bigg) \nonumber\\
& \text{\ } + \Upsilon_K\rho^{n_2(i) - k}(4\Upsilon + \mu)\epsilon + \Upsilon_K\rho^{k +1- n_1(i)}\|\bx_{n_1(i), i}^0 - \tx_{n_1(i)}\| \nonumber\\
&\leq  \Upsilon_2\bigg(\sum_{j=n_1(i)}^{n_2(i)}\rho^{|k-j|}(\Psi_{j, i}^0)^2 + \rho^{n_2(i) - k}\epsilon \nonumber\\
&\text{\ } +  \rho^{k - n_1(i)}\|\bx_{n_1(i), i}^0 - \tx_{n_1(i)}\|\bigg),
\end{align}
\end{footnotesize}
\hskip -4pt 
where $\Upsilon_2 = \frac{2\Upsilon_K\Upsilon_L}{\rho}\vee \Upsilon_K(4\Upsilon_L+\mu)\geq \Upsilon_1$ and the last inequality is due to $\rho^{|k+1 - j_2|}\leq \rho^{|k - j_2|}/\rho$. Combining \eqref{pequ:21}-\eqref{pequ:23} and noting that $\Psi_{k, i}^1\leq \|\bz_{k, i}^1 - \tz\| + \|\blambda_{k, i}^1 - \tlambda_k\|$, we define $\Upsilon_C = 2\Upsilon_2$ and then complete the proof.

\subsection{Proof of Theorem \ref{thm:3}}\label{pf:thm:3}

By the ``discard the tail" technique as  in \eqref{equ:trans1} and \eqref{equ:trans2}, it suffices to show that
\begin{align*}
(\bz_{j_1:j_2, i+1}^0, \blambda_{j_1:j_2, i+1}^0)\in \N_\epsilon(\tz_{j_1:j_2, i+1}, \tlambda_{j_1:j_2, i+1})
\end{align*}
with $j_1 = n_1(i+1)$ and $j_2 =n_2(i+1)-2L$. Since $j_1 = n_1(i) + L$ and $j_2 = n_2(i) - L$, and
\begin{align*}
\bz_{j_1:j_2, i+1}^0 = \bz_{j_1:j_2, i}^1, \quad\quad \blambda_{j_1:j_2, i+1}^0 = \blambda_{j_1:j_2, i}^1,
\end{align*}
we need only show $\Psi_{j_1:j_2, i}^1 \leq \epsilon$. For any $k\in[j_1, j_2]$, by \eqref{equ:error} in Theorem \ref{thm:2}, we get 
\begin{footnotesize}
	\begin{align*}
	\Psi_{k, i}^1\leq & \Upsilon_C\big(\sum_{j=n_1(i)}^{n_2(i)}\rho^{|k-j|}(\Psi_{j, i}^0)^2  + \rho^{k - n_1(i)}\|\bx_{n_1(i), i}^0 - \tx_{n_1(i)}\| \\
	&+ \rho^{n_2(i) - k}\epsilon\big)\\
	\leq & \Upsilon_C \sum_{j=n_1(i)}^{n_2(i)}\rho^{|k-j|}\cdot 3\epsilon^2 + 2\rho^L\Upsilon_C\epsilon \\
	\leq & 3\Upsilon_C\epsilon^2\big(\max_{h\in[n_1(i), n_2(i)]}\sum_{j=n_1(i)}^{n_2(i)}\rho^{|h-j|}\big) + 2\Upsilon_C\epsilon^2 .
	\end{align*}
\end{footnotesize}
\hskip -5pt
The last inequality is due to the condition $\rho^L\leq \epsilon$. For any $h\in[n_1(i), n_2(i)]$,
\begin{align}\label{pequ:24}
\sum_{j=n_1(i)}^{n_2(i)}\rho^{|h-j|} = &1 +\sum_{j=n_1(i)}^{h-1}\rho^{h - j} + \sum_{j=h+1}^{n_2(i)}\rho^{j - h} \nonumber\\
= &\sum_{j=0}^{h - n_1(i)}\rho^j + \sum_{j=1}^{n_2(i) - h}\rho^j \leq \frac{1+\rho}{1 - \rho}.
\end{align}
Combining the above equations, we get
\begin{align*}
\Psi_{k, i}^1\leq (2 \Upsilon_C + 3\Upsilon_C\frac{1+\rho}{1-\rho})\epsilon^2= \frac{\Upsilon_C(5 + \rho)}{1 - \rho}\epsilon^2.
\end{align*}
Using $\frac{\Upsilon_C(5 + \rho)}{1 - \rho}\epsilon \leq 1$ completes the proof.

\subsection{Proof of Theorem \ref{thm:4}}\label{pf:thm:4}

We apply Theorems \ref{thm:1}, \ref{thm:2} and \ref{thm:3} iteratively. For $i = 1$, since $(\tbz_1^0, \tblambda_1^0)$ is initialized by $(\bz^0, \blambda^0)\in \N_\epsilon(\tz, \tlambda)$, by Theorem~\ref{thm:1}, uniform SOSC holds for $\P_1(\tbd_1)$. Thus, $K^1$ is invertible, and \eqref{equ:error} holds for $\P_1(\tbd_1)$. By stability in Theorem~\ref{thm:3}, $(\tbz_2^0, \tblambda_2^0)\in\N_\epsilon(\tbz_2^\star, \tblambda_2^\star)$. Again, using Theorem \ref{thm:1}, we know that uniform SOSC can be applied to $\P_2(\tbd_2)$. Thus, the invertibility of $K^2$ and \eqref{equ:error} are guaranteed for the second problem. Repeating the rationale for all stages and noting that the bounds on the relevant parameters of each subproblem do not change, we finish the proof.

\section{Proofs of Results in Section \ref{sec:4.3}}\label{sec:C}

\subsection{Proof of Theorem \ref{thm:5}}\label{pf:thm:5}

We first define some notations. Let $\Omega_s^i = \max_{k\in O^i_s}\Psi_{k, i}^0$, $\barkappa = \kappa/3$, and let the sequence $\{\bardelta_s\}_s$ be defined as
\begin{align*}
\bardelta_0 = \bardelta_1 = 3\epsilon, \quad \bardelta_2 = \barkappa(\bardelta_1)^2, \quad \bardelta_s = \barkappa^{s-2}(\bardelta_1)^{s-1}, \text{\ } \forall s\geq 3.
\end{align*}
Since $\barkappa\bardelta_1\leq 1$, $\bardelta_s$ is nonincreasing. In what follows, we~will prove $\Omega_s \leq \delta_s \leq \bardelta_s$ for some sequence $\{\delta_s\}_s$, where $\delta_s$ is defined using $\{\bardelta_j\}_{j=0}^{s-1}$.

Starting from the first subproblem, one can easily see $\Omega_0^1\leq 3\epsilon \coloneqq \delta_0\leq \bardelta_0$. Moving to the second subproblem, we have $\Omega_0^2\leq \delta_0\leq \bardelta_0$ and $\Omega_1^2\leq 3\epsilon \coloneqq\delta_1\leq \bardelta_1$ (due to the stability in Theorem \ref{thm:3}). For the third subproblem, we still have $\Omega_0^3\leq \delta_0\leq \bardelta_0$, $\Omega_1^3\leq \delta_1\leq \bardelta_1$, since $(\bz_{k,3}^0, \blambda_{k, 3}^0)  = (\bz_k^0, \blambda_k^0)\in \N_\epsilon(\bz_k^\star, \blambda_k^\star)$ for $k\in O_0^3\cup O_1^3$. Further, by \eqref{equ:error},
\begin{align}\label{pequ:25}
\Omega^3_2 = &\max_{k\in[n_1(3), n_2(1))}\Psi_{k, 3}^0 = \max_{k\in[n_1(2)+L, n_2(2)-L)}\Psi_{k, 2}^1 \nonumber\\
\leq & \Upsilon_C(\bardelta_1)^2\max_{k\in[n_1(2)+L, n_2(2)-L)}\sum_{j=n_1(2)}^{n_2(i)}\rho^{|k-j|} \nonumber\\
&+ \Upsilon_C\rho^L\bardelta_1 + \Upsilon_C\rho^L\epsilon \nonumber\\
\leq & \Upsilon_C\bigg(\frac{1+\rho}{1-\rho}(\bardelta_1)^2 + \rho^L\bardelta_1 + \rho^L\epsilon\bigg)\coloneqq \delta_2,
\end{align}
where the second inequality is due to \eqref{pequ:24}. Note that
\begin{align*}
\delta_2\leq \frac{\Upsilon_C(1+\rho)}{1-\rho}\bigg((\bardelta_1)^2 + \frac{1}{3}(\bardelta_1)^2 + \frac{1}{9}(\bardelta_1)^2\bigg)\leq \barkappa(\bardelta_1)^2 = \bardelta_2.
\end{align*}
Thus, we have $\Omega_2^3\leq \delta_2\leq \bardelta_2$. Next, we move to the $4$th subproblem. Since the iterates on the tail are set to $(\bz_k^0, \blambda_k^0)$ (the same as  for the third subproblem), $\Omega_0^4\leq \delta_0\leq \bardelta_0$, $\Omega_1^4\leq \delta_1\leq \bardelta_1$. Further, we claim that $\Omega_2^4\leq \delta_2$. In fact, we have that
\begin{align}\label{pequ:27}
\Omega_2^4 =& \max_{k\in[n_2(1), n_2(2))}\Psi_{k, 4}^0\leq \max_{k\in[n_1(4), n_2(2))}\Psi_{k, 3}^1 \nonumber\\
=& \max_{k\in[n_1(3)+L, n_2(3)-L)}\Psi_{k, 3}^1,
\end{align}
and bounding the right-hand side term follows the same derivation as bounding $\max_{k\in[n_1(2)+L, n_2(2)-L)}\Psi_{k, 2}^1$ in \eqref{pequ:25}, because $\Omega_2^3\leq \bardelta_2\leq \bardelta_1$ (thus, bounds on the third subproblem can be degraded to the ones on the second subproblem). For bounding $\Omega_3^4$, by \eqref{equ:error}, we get 
\begin{footnotesize}
\begin{align*}
\Omega_3^4 = &\max_{k\in[n_1(4), n_2(1))}\Psi_{k, 4}^0 = \max_{k\in[n_1(3)+L, n_2(1))}\Psi_{k, 3}^1\\
\leq &\Upsilon_C\bigg(\max_{k\in[n_1(3)+L, n_2(1))}\sum_{j=n_1(3)}^{n_2(3)}\rho^{|k - j|}(\Psi_{k, 3}^0)^2 + \rho^L\bardelta_2 + \rho^{2L}\epsilon\bigg)\\
= & \Upsilon_C\bigg(\max_{k\in[n_1(3)+L, n_2(1))}\sum_{j=n_1(3)}^{n_2(1)-1}\rho^{|k-j|}(\bardelta_2)^2 \\
& + \max_{k\in[n_1(3)+L, n_2(1))}\sum_{j=n_2(1)}^{n_2(3)}\rho^{|k-j|}(\bardelta_1)^2 + \rho^L\bardelta_2 + \rho^{2L}\epsilon\bigg)\\
\leq & \Upsilon_C\bigg(\frac{1+\rho}{1-\rho}(\bardelta_2)^2 + \frac{\rho}{1-\rho}(\bardelta_1)^2 + \rho^L\bardelta_2 + \rho^{2L}\epsilon\bigg) \coloneqq\delta_3.
\end{align*}
\end{footnotesize}
\hskip -4pt
Using the definition of $\delta_2$ in \eqref{pequ:25}, we have 
\begin{footnotesize}
\begin{align*}
\delta_3 =& \Upsilon_C\bigg(\frac{1+\rho}{1-\rho}(\bardelta_2)^2 + \frac{\rho - \rho^L(1+\rho)}{1-\rho}(\bardelta_{1})^2 + \frac{\rho^L}{\Upsilon_C}\delta_2 + \rho^L\bardelta_2 - \rho^{2L}\bardelta_1\bigg)\\
\leq & \frac{2\Upsilon_C(1+\rho)}{1-\rho}\big((\bardelta_2)^2 + \epsilon\bardelta_2 + (\bardelta_1)^2\big)\leq \frac{4\Upsilon_C(1+\rho)}{1-\rho}\big(\epsilon\bardelta_2 + (\bardelta_1)^2\big)\\
\leq &\frac{16\Upsilon_C(1+\rho)}{3(1-\rho)}(\bardelta_1)^2\leq \bardelta_3.
\end{align*}
\end{footnotesize}
\hskip -4pt
Here, the first inequality is due to $\delta_2\leq \bardelta_2$; the second inequality is due to $\bardelta_2\leq \bardelta_1$; and the third inequality is due to $\bardelta_2\leq \bardelta_1$ and $\epsilon = \bardelta_1/3$. Using the relation $\delta_3\leq \bardelta_3$, we can further show that $\Omega^5_s\leq \delta_s$, for $s = 0,1,2,3$, since bounds for the $4$th subproblem can be degraded to the ones for the third (the same reason as \eqref{pequ:27}). We can also bound $\Omega^5_4$ similarly. In general, for $s\geq 2$, we apply \eqref{equ:error} and get
\begin{footnotesize}
\begin{align*}
&\Omega_s^{s+1} = \max_{k\in[n_1(s+1), n_2(1))}\Psi_{k, s+1}^0 = \max_{k\in[n_1(s)+L, n_2(1))}\Psi_{k, s}^1\\
&\leq  \Upsilon_C\bigg(\max_{k\in[n_1(s)+L, n_2(1))}\sum_{j=n_1(s)}^{n_2(1)-1}\rho^{|k-j|}(\Psi_{j, s}^0)^2 \\
&\text{\ } + \max_{k\in[n_1(s)+L, n_2(1))}\sum_{h=1}^{s-3}\sum_{j=n_2(h)}^{n_2(h+1)-1}\rho^{|k-j|}(\Psi_{j,s}^0)^2\\
& \text{\ } + \max_{\substack{k\in[n_1(s)+L,\\ n_2(1))}}\sum_{j=n_2(s-2)}^{n_2(s)}\rho^{|k-j|}(\Psi_{j,s}^0)^2 + \rho^L\bardelta_{s-1} + \rho^{(s-1)L}\epsilon\bigg)\\
\leq & \Upsilon_C\bigg(\frac{1+\rho}{1-\rho}(\Omega_{s-1}^s)^2 + \sum_{h=1}^{s-3}(\Omega^s_{s-1-h})^2\underbrace{\max_{\substack{k\in[n_1(s)+L,\\ n_2(1))}}\sum_{j=n_2(h)}^{n_2(h+1)-1}\rho^{|k-j|}}_{\leq \frac{\rho^{1+(h-1)L}}{1-\rho}} \\
&\text{\ } + \frac{\rho^{1+(s-3)L}}{1-\rho}(\Omega_{1}^s)^2 + \rho^L\bardelta_{s-1} + \rho^{(s-1)L}\epsilon\bigg)\\
\leq & \Upsilon_C\bigg(\frac{1+\rho}{1-\rho}(\bardelta_{s-1})^2 + \frac{\rho}{1-\rho}\big((\bardelta_{s-2})^2 + \rho^L(\bardelta_{s-3})^2 \\
&\text{\ } + \cdots + \rho^{(s-3)L}(\bardelta_1)^2\big) + \rho^L\bardelta_{s-1} + \rho^{(s-1)L}\epsilon\bigg)\\
\coloneqq &\delta_s.
\end{align*}
\end{footnotesize}
\hskip -4pt
Here, the second inequality uses the fact that $\Omega_1^s = \Omega_0^s$ (``discard the tail" technique). For $s\geq 3$, we rewrite $\delta_{s-1}$ as
\begin{footnotesize}
\begin{multline*}
\rho^L(\bardelta_{s - 3})^2 + \cdots + \rho^{(s-3)L}(\bardelta_1)^2 \\= \frac{\rho^L(1-\rho)}{\rho}\cdot\bigg(\frac{\delta_{s-1}}{\Upsilon_C} - \frac{1+\rho}{1-\rho}(\bardelta_{s-2})^2 - \rho^L\bardelta_{s-2} - \rho^{(s-2)L}\epsilon\bigg).
\end{multline*}
\end{footnotesize}
\hskip -4pt
Combining these two displays, we obtain 
\begin{footnotesize} 
\begin{align}\label{pequ:26}
\delta_s =& \Upsilon_C\bigg(\frac{1+\rho}{1-\rho}(\bardelta_{s-1})^2 + \frac{\rho - \rho^L(1+\rho)}{1-\rho}(\bardelta_{s-2})^2 + \frac{\rho^L}{\Upsilon_C}\delta_{s-1} \nonumber\\
& +  \rho^L\bardelta_{s-1} - \rho^{2L}\bardelta_{s-2}\bigg) \nonumber\\
\leq & \frac{2\Upsilon_C(1+\rho)}{1-\rho}\big((\bardelta_{s-1})^2 + \epsilon\bardelta_{s-1} + (\bardelta_{s-2})^2\big) \nonumber\\
\leq & \frac{4\Upsilon_C(1+\rho)}{1-\rho}\big( \epsilon\bardelta_{s-1} + (\bardelta_{s-2})^2\big),
\end{align}
\end{footnotesize}
\hskip -4pt
where the first inequality uses the condition $\delta_{s-1}\leq \bardelta_{s-1}$ and $\rho^L\leq \epsilon$ and the second inequality uses the fact that $\bardelta_{s-1}\leq \bardelta_{s-2}$. Then, we only need check $\delta_s\leq \bardelta_s$ for $s\geq 4$. For $s = 4$,
\begin{align*}
\delta_4\leq &\frac{4\Upsilon_C(1+\rho)}{1-\rho}\big(\epsilon\bardelta_3 + (\bardelta_2)^2\big) = \frac{3\barkappa}{4}\big(\epsilon\barkappa(\bardelta_1)^2 + \barkappa^2(\bardelta_1)^4\big) \\
=& \barkappa^2(\bardelta_1)^3\big(\frac{1}{4}+\frac{3\barkappa\bardelta_1}{4}\big)\leq \bardelta_4.
\end{align*}
For $s\geq 5$, we have $(\bardelta_{s-2})^2  = \barkappa^{2s-8}(\bardelta_1)^{2s-6} = \epsilon\barkappa^{s-3}(\bardelta_1)^{s-2}\cdot 3\barkappa^{s-5}(\bardelta_1)^{s-5}\leq 3\epsilon\bardelta_{s-1}$. Hence,
\begin{align*}
\delta_s\leq& \frac{4\Upsilon_C(1+\rho)}{1-\rho}\big(\epsilon\bardelta_{s-1} + (\bardelta_{s-2})^2\big)\leq \frac{16\Upsilon_C(1+\rho)}{1-\rho}\epsilon\bardelta_{s-1} \\
= &\barkappa\bardelta_1\bardelta_{s-1} = \bardelta_s.
\end{align*}
Until now, we have shown that a nonincreasing sequence $\{\bardelta_s\}$ controls $\delta_s$, which is the upper bound of $\Omega_{s}^{s+1}$. We only need show that $\Omega_s$, $\forall s$, is also bounded by $\delta_s$ (we partially  demonstrated this in \eqref{pequ:27}). In fact, $\Omega_s = \max_{i\in[1, T]}\Omega_s^i = \max_{i\in[s+1, T]}\Omega_s^i$. For $i \in[s+2, T]$, $\Omega^i_s$ can be bound with the same derivation as bounding $\Omega^{i-1}_s$ and results in an upper bound $\delta_s$, since the $(i-1)$th subproblem (which is used to bound $\Omega^i_s$) can be degraded to have the same error structure as the $(i-2)$th subproblem (which is used to bound $\Omega^{i-1}_s$); see \eqref{pequ:27} for an example. Therefore, $\Omega_s\leq \delta_s\leq \bardelta_s$. This completes the proof.